\documentclass[letterpaper, 11pt,  reqno]{amsart}

\usepackage{amsmath,amssymb,amscd,amsthm,amsxtra, esint}

\usepackage[margin=1.2in,marginparwidth=1.5cm, marginparsep=0.5cm]{geometry}

\usepackage[implicit=true]{hyperref}
\allowdisplaybreaks[2]

\usepackage{color}

\definecolor{gr}{rgb}   {0.,   0.69,   0.23 }
\definecolor{bl}{rgb}   {0.,   0.5,   1. }
\definecolor{mg}{rgb}   {0.85,  0.,    0.85}
\definecolor{yl}{rgb}   {0.8,  0.7,   0.}
\definecolor{or}{rgb}  {0.7,0.2,0.2}

\newtheorem{theorem}{Theorem} [section]

\newtheorem{lemma}[theorem]{Lemma}
\newtheorem{proposition}[theorem]{Proposition}
\newtheorem{remark}[theorem]{Remark}

\newtheorem{definition}[theorem]{Definition}

\DeclareMathOperator*{\supp}{supp}

\DeclareMathOperator{\Law}{Law}
\DeclareMathOperator*{\wlim}{w-lim}

\newcommand{\1}{\hspace{0.5mm}\text{I}\hspace{0.5mm}}
\newcommand{\II}{\text{I \hspace{-2.8mm} I} }
\newcommand{\III}{\text{I \hspace{-2.9mm} I \hspace{-2.9mm} I}}

\newcommand{\IV}{\text{I \hspace{-2.9mm} V}}

\newcommand{\noi}{\noindent}
\newcommand{\Z}{\mathbb{Z}}
\newcommand{\R}{\mathbb{R}}

\newcommand{\D}{\mathcal{D}}
\newcommand{\T}{\mathbb{T}}
\newcommand{\bul}{\bullet}

\let\Re=\undefined\DeclareMathOperator*{\Re}{Re}
\let\Im=\undefined\DeclareMathOperator*{\Im}{Im}

\let\P= \undefined
\newcommand{\P}{\mathbf{P}}
\newcommand{\PP}{\mathbb{P}}

\newcommand{\Q}{\mathbf{Q}}

\newcommand{\E}{\mathbb{E}}
\newcommand{\EE}{\mathcal{E}}

\newcommand{\F}{\mathcal{F}}

\newcommand{\al}{\alpha}
\newcommand{\be}{\beta}
\newcommand{\dl}{\delta}

\newcommand{\nb}{\nabla}

\newcommand{\Dl}{\Delta}
\newcommand{\eps}{\varepsilon}
\newcommand{\kk}{\kappa}
\newcommand{\g}{\gamma}
\newcommand{\G}{\Gamma}
\newcommand{\ld}{\lambda}
\newcommand{\Ld}{\Lambda}
\newcommand{\s}{\sigma}

\newcommand{\ft}{\widehat}

\newcommand{\wt}{\widetilde}
\newcommand{\cj}{\overline}

\newcommand{\dt}{\partial_t}
\newcommand{\dd}{\partial}

\newcommand{\ta}{\theta}

\newcommand{\Ta}{\Theta}

\newcommand{\W}{\mathbf{W}}

\renewcommand{\l}{\ell}
\renewcommand{\o}{\omega}
\renewcommand{\O}{\Omega}

\newcommand{\les}{\lesssim}

\newcommand{\jb}[1]
{\langle #1 \rangle}

\newcommand{\ind}{\mathbf 1}

\renewcommand{\S}{\mathcal{S}}

\newcommand{\A}{\mathcal{A}}
\newcommand{\X}{\mathcal{X}}

\newcommand{\M}{\mathcal{M}}

\newcommand{\N}{\mathbb{N}}

\newcommand{\NN}{\mathcal{N}}

\newcommand{\bC}{\mathbf{C}}

\newcommand{\GG}{\mathcal{G}}

\newcommand{\J}{\mathcal{J}}

\renewcommand{\H}{\mathcal{H}}

\newcommand{\too}{\longrightarrow}

\newcommand{\WW}{\mathbb{W}}
\newcommand{\YB}{\mathbf{Y}}
\newcommand{\ZB}{\mathbf{Z}}

\usepackage{tikz}

\usetikzlibrary{shapes.misc}
\usetikzlibrary{shapes.symbols}
\usetikzlibrary{shapes.geometric}
\usepackage{pgfplots}
\pgfplotsset{%
every x tick/.style={black, thick},
every y tick/.style={black, thick},
every tick label/.append style = {font=\footnotesize},
every axis label/.append style = {font=\footnotesize},
compat=1.12
  }

\newcommand{\per}{\text{\rm per}}

\newcommand{\rhoo}{\vec{\rho}}
\newcommand{\muu}{\vec{\mu}}

\newcommand{\ze}{\lambda}

\newcommand{\z}{\zeta}

\newcommand{\ZZ}{\mathcal{Z}}

\newcommand{\plan}{\mathfrak{p}}

\newcommand{\ff}{{\bf f}}
\newcommand{\vv}{{\bf v}}

\newcommand{\eval}{\text{eval}_{t = 0}}
\newcommand{\proj}{\text{proj}}

\newcommand{\jbb}[1]
{[\hspace{-0.6mm}[ #1 ]\hspace{-0.6mm}]}

\newtheorem*{ackno}{Acknowledgments}

\numberwithin{equation}{section}
\numberwithin{theorem}{section}

\makeatletter
\@namedef{subjclassname@2020}{%
  \textup{2020} Mathematics Subject Classification}
\makeatother

\begin{document}
\baselineskip = 14pt

\title[Hyperbolic $P(\Phi)_2$-model on the plane]
{Hyperbolic $P(\Phi)_2$-model on the plane}

\author[T. Oh, L. Tolomeo, Y. Wang, and G. Zheng]
{Tadahiro Oh,  Leonardo Tolomeo,  Yuzhao Wang, and Guangqu Zheng}

%

\address{
Tadahiro Oh, School of Mathematics\\
The University of Edinburgh\\
and The Maxwell Institute for the Mathematical Sciences\\
James Clerk Maxwell Building\\
The King's Buildings\\
Peter Guthrie Tait Road\\
Edinburgh\\ 
EH9 3FD\\
United Kingdom,
 and 
 School of Mathematics and Statistics, Beijing Institute of Technology, Beijing 100081,
China}

\email{hiro.oh@ed.ac.uk}

\address{
Leonardo Tolomeo,  School of Mathematics\\
The University of Edinburgh\\
and The Maxwell Institute for the Mathematical Sciences\\
James Clerk Maxwell Building\\
The King's Buildings\\
Peter Guthrie Tait Road\\
Edinburgh\\ 
EH9 3FD\\
United Kingdom}

\email{l.tolomeo@ed.ac.uk}

\address{
Yuzhao Wang\\
School of Mathematics\\
Watson Building\\
University of Birmingham\\
Edgbaston\\
Birmingham\\
B15 2TT\\ United Kingdom}

\email{y.wang.14@bham.ac.uk}

\address{
Guangqu Zheng, School of Mathematics\\
The University of Edinburgh\\
and The Maxwell Institute for the Mathematical Sciences\\
James Clerk Maxwell Building\\
The King's Buildings\\
Peter Guthrie Tait Road\\
Edinburgh\\ 
EH9 3FD\\
 United Kingdom, 
 Department of Mathematical Sciences\\
University of Liverpool\\
Mathematical Sciences Building\\ 
Liverpool, L69 7ZL
United Kingdom, 
and
Department of Mathematics and Statistics, Boston University, 665 Commonwealth Avenue, Boston, MA 02215, USA
 }

\email{gzheng90@bu.edu}

%
%
%
%

%
%


%
%
%

\begin{abstract}

In this paper, we construct 
invariant Gibbs dynamics for  
the hyperbolic $\Phi^{k+1}_2$-model
 (namely,  defocusing stochastic damped nonlinear
wave equation forced 
by an additive space-time white noise)
on the plane.
(i)~For this purpose, we first revisit the construction
of a $\Phi^{k+1}_2$-measure on the plane.
More precisely, 
by establishing coming down from infinity
for the associated stochastic nonlinear heat equation (SNLH)
on the plane,  
we first construct a $\Phi^{k+1}_2$-measure 
on the plane
as a limit of the $\Phi^{k+1}_2$-measures on large tori.
(ii)~We then construct invariant Gibbs dynamics for  
the hyperbolic $\Phi^{k+1}_2$-model
on the plane, 
by taking a  limit
 of the invariant Gibbs dynamics on large tori
 constructed  by the first two  authors with Gubinelli and Koch (2022). 
Here, 
our main strategy 
is to develop further the ideas
from  a recent work on the hyperbolic $\Phi^3_3$-model 
on the three-dimensional torus by the first two authors and Okamoto (2025), 
and to 
study convergence of the so-called enhanced Gibbs measures,
for which 
 coming down from infinity
for the associated SNLH with positive regularity plays a crucial role.
By combining wave and heat analysis together with 
 ideas from optimal transport theory, 
 we then 
 conclude global well-posedness of 
the hyperbolic $\Phi^{k+1}_2$-model
on the plane
and invariance of the associated Gibbs measure.
As a byproduct of our argument, 
we also obtain invariance 
of the limiting 
 $\Phi^{k+1}_2$-measure on the plane
 under the dynamics
 of the parabolic
  $\Phi^{k+1}_2$-model.

\end{abstract}


\maketitle


%
%
\tableofcontents

\section{Introduction}
\label{SEC:1}

\subsection{Hyperbolic $\Phi^{k+1}_2$-model}

We study the following stochastic damped nonlinear wave equation
(SdNLW) forced by an additive space-time white noise, posed on the plane $\R^2$:
\begin{align}
\dt^2 u + \dt u + (1- \Dl)u \,  +  u^k = \sqrt{2}\xi, 
\label{NLW1}
\end{align}

\noi
where $k \in 2\N+1$, $u$ is a real-valued unknown, 
and $\xi(x, t)$ is a Gaussian space-time white noise on $\R^2 \times \R_+$
with the space-time covariance given by
\[ \E\big[ \xi(x_1, t_1) \xi(x_2, t_2) \big]
= \dl(x_1 - x_2) \dl (t_1 - t_2).\]

With  $\vec{u} = (u, \dt u)$, 
define the energy $\mathcal{E}(\vec u)$ by 
\begin{align}
\begin{split}
\EE(\vec{u})
& = E(u) +  \frac 12 \int_{\R^2} (\dt u)^2 dx \\
& = \frac 12 \int_{\R^2} |\jb{\nabla} u|^2 dx + \frac 12 \int_{\R^2} (\dt u)^2 dx 
+ \frac 1 {k+1} \int_{\R^2} u^{k+1} dx, 
\end{split}
\label{Hamil1}
\end{align}

\noi
where $E(u)$ is given by 
\begin{align}
 E(u)
 = \frac 12 \int_{\R^2} |\jb{\nabla} u|^2 dx 
+ \frac 1 {k+1} \int_{\R^2} u^{k+1} dx.
\label{Hamil2}
\end{align}

\noi
Note that the energy $\EE(\vec u)$ is 
 precisely the energy (= Hamiltonian) 
for  the  (deterministic) nonlinear wave equation (NLW):
\begin{align}
\dt^2 u + (1 -  \Dl)  u  + u^k  = 0.
\label{NLW2}
\end{align}

\noi
Namely, 
with  $v = \dt u$, we can write NLW \eqref{NLW2} 
in the following Hamiltonian formulation:
\begin{align}
\dt   \begin{pmatrix}
u \\ v
\end{pmatrix}
=     
\begin{pmatrix}
0 & 1\\-1 & 0 
\end{pmatrix}
\begin{pmatrix} 
\frac{\dd\mathcal{E}}{\dd u}
\rule[-3mm]{0pt}{2mm}
\\
 \frac{\dd\mathcal{E}}{\dd v}
\end{pmatrix}.
\label{NLW2a}
\end{align}

\noi
Similarly, we can write  SdNLW  \eqref{NLW1} as
\begin{align}
\dt   \begin{pmatrix}
u \\ v
\end{pmatrix}
=     
\begin{pmatrix}
0 & 1\\-1 & 0 
\end{pmatrix}
\begin{pmatrix} 
\frac{\dd\mathcal{E}}{\dd u}
\rule[-3mm]{0pt}{2mm}
\\
 \frac{\dd\mathcal{E}}{\dd v}
\end{pmatrix}
+      \begin{pmatrix} 
0 
\\
- v + \sqrt 2 \xi
\end{pmatrix}.
\label{NLW3}
\end{align}

\noi
Consider  a 
 Gibbs measure $\rhoo$ of the form 
\begin{align}
``d\rhoo(\vec u ) = \ZZ^{-1}e^{-\EE(\vec u  )}d\vec u 
=  d\rho \otimes d\mu_0(\vec u )\text{''}, 
\label{Gibbs0}
\end{align}

\noi
where $\vec u = (u, v) = (u, \dt u)$,  $\rho$ is a $\Phi^{k+1}_2$-measure on $\R^2$, 
and $\mu_0$ is the (spatial) white noise measure on $\R^2$.
By drawing an analogy to the finite-dimensional Hamiltonian dynamics, 
we expect that the Gibbs measure $\rhoo$ in \eqref{Gibbs0}
is invariant under the NLW dynamics \eqref{NLW2} on $\R^2$.
Moreover, it is easy to see that the white noise measure $\mu_0$ on 
the second component $v= \dt u$
is invariant under the 
Ornstein-Uhlenbeck dynamics: 
\begin{align}
\dt  v = - v + \sqrt 2 \xi.
\label{OU1}
\end{align}

\noi
Thus, by viewing the SdNLW dynamics \eqref{NLW3}
 as a superposition
of the NLW dynamics~\eqref{NLW2a} and the 
Ornstein-Uhlenbeck dynamics~\eqref{OU1}, 
we expect  the Gibbs measure $\rhoo$ in \eqref{Gibbs0}
to be invariant under the SdNLW dynamics \eqref{NLW1}.

Indeed, from the stochastic quantization point of view
\cite{PW, RSS}, 
 the equation \eqref{NLW1}
corresponds to  the so-called canonical stochastic quantization equation\footnote{Namely, 
 the `Hamiltonian' stochastic quantization equation given as the Langevin equation
 with the momentum $v = \dt u$; see \eqref{NLW3}.
The parabolic $\Phi^{k+1}_d$-model
 is the stochastic gradient flow for the energy functional~$E(u)$ defined in \eqref{Hamil2}.}
  for the $\Phi^{k+1}_2$-measure
 and thus is of importance in mathematical physics (in particular, constructive quantum field theory); see \cite{RSS}, where the terminology ``canonical stochastic quantization'' was introduced.
In the parabolic setting, 
the stochastic quantization equation 
of the $\Phi^{k+1}_d$-measure\footnote{The subscript $d$ 
denotes the dimension of the underlying space.} is given by the 
following  parabolic $\Phi^{k+1}_d$-model
(= stochastic nonlinear heat equation (SNLH)):
\begin{align}
 \dt X + (1- \Dl)X \,  +  X^k = \sqrt{2}\xi, 
\label{NLH1}
\end{align}

\noi
under which the $\Phi^{k+1}_d$-measure remains invariant;
see 
\cite{DD, Hairer, MW17a, MW17b, HM, GH1, AK1, GH2, AK2}.
The equation \eqref{NLW1}
is the hyperbolic counterpart of the parabolic 
$\Phi^{k+1}_d$-model when $d = 2$.
 For this reason, we refer  to  \eqref{NLW1} 
as the {\it hyperbolic $\Phi^{k+1}_2$-model}.
Our main goal in this paper is  to construct the global-in-time dynamics
at the Gibbs equilibrium for 
the hyperbolic $\Phi^{k+1}_2$-model~\eqref{NLW1}
posed on~$\R^2$. 

In the parabolic case, Mourrat and Weber \cite{MW17a}
proved pathwise global well-posedness
of the parabolic $\Phi^{k+1}_2$-model \eqref{NLH1}
on $\R^2$
with given deterministic initial data
(of negative regularity) 
for any $k \in 2\N+1$ (see Remark 4 in \cite{MW17a}).\footnote{Remark 1.5 in the arXiv version.}
When $k = 3$, the second author proved
pathwise global well-posedness of the hyperbolic $\Phi^{4}_2$-model 
on $\R^2$
with given deterministic initial data
(but of positive regularity).
However, such a pathwise global well-posedness result 
is not available for higher values of $k$.
In order to overcome this difficulty, we
apply 
Bourgain's invariant measure argument \cite{BO94, BO96} `in spirit'
and construct global-in-time dynamics in a probabilistic manner.

For this purpose, 
 we first revisit the construction
of a $\Phi^{k+1}_2$-measure on the plane.
More precisely, 
by establishing coming down from infinity
for the associated SNLH
on the plane,  
we  construct a $\Phi^{k+1}_2$-measure 
on the plane
as a limit of the $\Phi^{k+1}_2$-measures on large tori.
In order to construct invariant Gibbs dynamics
for the hyperbolic $\Phi^{k+1}_2$-model~\eqref{NLW1}, 
we 
develop further the ideas
from  a recent work on the hyperbolic $\Phi^3_3$-model~\cite{OOT2}
 by the first two authors and Okamoto.
 More precisely, 
our main strategy is to
study convergence of the so-called {\it enhanced Gibbs measures}
(= distributions
of the enhanced data sets), 
which provides
the main statistical control, 
by combining wave and heat analysis
 together with 
 ideas from optimal transport theory.
In our analysis, we crucially exploit 
the finite speed of propagation
for the hyperbolic $\Phi^{k+1}_2$-model \eqref{NLW1}.
Moreover, 
 coming down from infinity
for the associated SNLH with positive regularity plays a fundamental role, 
which is of independent interest.
As a byproduct of our argument, 
we also obtain invariance 
of the limiting 
 $\Phi^{k+1}_2$-measure on the plane
 under the dynamics
 of the parabolic
  $\Phi^{k+1}_2$-model;
  see Remarks~\ref{REM:inv1} and~\ref{REM:conv2}.

\subsection{Review of the periodic problem}
\label{SUBSEC:1.2}

Following the previous works \cite{OTh2, GKO}, 
Gubinelli, Koch, and the first two authors
\cite{GKOT}
studied SdNLW \eqref{NLW1} on the two-dimensional torus $\T^2= (\R/\Z)^2$.
By introducing 
a proper renormalization (see  \eqref{Wick2} below; see also \cite{OTh2, GKO}), 
they  constructed  global-in-time 
invariant Gibbs dynamics for  SdNLW \eqref{NLW1} on $\T^2$
via the so-called Bourgain's invariant measure argument~\cite{BO94, BO96}.
See \cite{OTh2} for the construction of invariant Gibbs dynamics
for the deterministic NLW \eqref{NLW2} on $\T^2$, preceding \cite{GKOT}.
See also \cite{ORTz} for the corresponding results (for both SdNLW \eqref{NLW1} and NLW \eqref{NLW2}) on a two-dimensional 
compact Riemannian manifold without boundary.

Given $L > 0$, 
define a dilated torus $\T_L^2$ by setting 
\[\T^2_L = (\R/L\Z)^2.\]

\noi
Then, the aforementioned result in \cite{GKOT}
applies to SdNLW posed on $\T^2_L$ for any $L > 0$.
Our main strategy for studying SdNLW \eqref{NLW1}
on $\R^2$ is then to take 
 a large torus limit  $L \to \infty$
 of the $L$-periodic SdNLW dynamics on $\T^2_L$.
As such, we first provide a review of the 
$L$-periodic problem on the dilated torus $\T^2_L$
in this subsection 
(especially since the presentation in \cite{GKOT}
is only for the $L = 1$ case).

Let us first introduce some notations.
Fix $L > 0$ and set 
\begin{align*}
\Z_L^2 =  (\Z/L)^2.
\end{align*}

\noi
Given $\ze  \in \Z_L^2$, 
we set 
\begin{align}
e_{\ze}^L(x) = \frac 1L e^{2\pi i \ze \cdot x}
\label{exp1}
\end{align}

\noi
for $x \in \T^2_L$.
Then, $\{e_{\ze}^L\}_{\ze \in \Z_L^2}$
forms an orthonormal basis of $L^2(\T^2_L)$.
We define the Fourier transform $\ft f(\ze)$ of a function $f$ on $\T^2_L$
by 
\begin{align*}
\ft f(\ze) = \int_{\T^2_L} f(x) \cj {e_\ze^L(x)} dx, \quad \ze \in \Z^2_L, 
\end{align*}

\noi
with the associated Fourier series expansion:\footnote{Hereafter, 
we may drop the inessential factor $2\pi$.}
\begin{align*}
f(x) = \sum_{\ze \in \Z_L^2} \ft f(\ze) e_\ze^L(x).
\end{align*}

We first go over the construction of the Gibbs measure on $\T^2_L$;
see \cite{DPT1, OTh1} for details (with $L = 1$).
See also \cite{LOZ} 
for the construction of analogous log-correlated Gibbs measures in the one-dimensional setting.
Given $ s \in \R$, 
let $\mu_{s, L}$ denote
a Gaussian measure on $L$-periodic distributions 
with the covariance operator $(1-\Dl_{\T^2_L})^{-s}$, 
%
%
 formally defined by\footnote{We use $\ZZ_{s, L}$, etc. to denote various normalizing constants, which 
may vary line by line.}
\begin{align}
\begin{split}
 d \mu_{s, L} (u)
&    = \ZZ_{s, L}^{-1} \exp\bigg(-\frac 12 \| u\|_{{H}^s(\T^2_L)}^2\bigg) du\\
& =  \ZZ_{s, L}^{-1} \prod_{\ze \in \Z_L^2} 
 e^{-\frac 12 \jb{\ze}^{2s} |\ft u(\ze)|^2}   
 d\ft u(\ze) , 
\end{split}
\label{mu0}
\end{align}

\noi
where $\jb{\,\cdot\,} = (1 + |\cdot |^2)^\frac 12$ and $\ft u(\ze)$, $\ze \in \Z^2_L$,  denotes the Fourier transform of $u$ on $\T^2_L$.
We note that 
$\mu_{1, L}$ corresponds to the massive Gaussian free field
on $\T^2_L$, 
while $\mu_{0, L}$ corresponds to the white noise on $\T^2_L$.
We then set 
\begin{align}
\muu_L = \mu_{L} \otimes \mu_{0, L},  \quad \text{where} \quad 
\mu_{L} = \mu_{1, L}.
\label{mu1}
\end{align}

\noi
Given $s \in \R$, let
 \[ \vec H^s(\T^2_L) = H^s(\T^2_L) \times  H^{s-1}(\T^2_L).\]
 
\noi
Then, $\muu_L$ in \eqref{mu1} is formally given by 
\begin{align}
 d \muu_{L} (u, v)
&    = \ZZ_{L}^{-1} \exp\bigg(-\frac 12 \| (u, v) \|_{ \vec H^1(\T^2_L)}^2\bigg) du dv.
\label{mu2}
\end{align}

\noi
Namely, 
 the measure $\muu_L$ is defined as 
   the induced probability measure
under the map:
\begin{equation*}
\o \in \O \longmapsto \big(u(\o), v(\o)\big) \in 
\D'(\T_L^2)\times \D'(\T_L^2)
\subset 
\D'(\R^2)\times \D'(\R^2), 
 \end{equation*}

\noi
where $u(\o) = u_L(\o)$ and $v(\o) = v_L(\o)$ are given by
the following Gaussian Fourier series:
\begin{equation}\label{series}
u(\o) = u_L(\o) = \sum_{\ze \in \Z^2_L} \frac{g_{L\ze}(\o)}{\jb{\ze}}e_\ze^L
\quad\text{and}\quad
v(\o)= v_L(\o) = \sum_{\ze \in \Z^2_L} h_{L\ze}(\o)e_\ze^L.
\end{equation}

\noi
Here, 
   $\{g_n,h_n\}_{n\in\Z^2}$ denotes  a family of independent standard complex-valued  Gaussian random variables conditioned  that $g_{-n} = \cj{g_n}$ and $h_{-n} = \cj{h_n}$, 
 $n \in \Z^2$.
From \eqref{series}, 
it is easy to see   that  $\muu_L = \mu_{1, L}\otimes\mu_{0, L}$ is supported on
$\vec H^{s}(\T^2_L) \setminus \vec H^0(\T^2_L)$
for $s < 0$.

In view of 
 \eqref{Hamil1}
and  \eqref{mu2}, 
  the Gibbs measure $\rhoo_L$ on $\T^2_L$ 
  is formally given by 
\begin{align}
d\rhoo_L(u,\dt u ) = \ZZ_L^{-1} e^{
- \frac1{k+1} \int_{\T^2_L} u^{k+1} dx} 
d\muu_L (u, \dt u ).
\label{Gibbs1}
\end{align} 

\noi
Due to  the roughness of the support of $\muu_L$, 
the interaction potential  $\int_{\T^2_L} u^{k+1} dx $ in \eqref{Gibbs1} is not well defined
and thus a  renormalization is required to give a proper meaning to 
the expression in~\eqref{Gibbs1}.

Given $N \in \N$, let $\P_N$
be the Dirichlet projection onto the frequencies $\{|\ld| \le N\}$.
Then, with  $u$  as in \eqref{series}, 
it follows from \eqref{series}, \eqref{exp1},  and  a Riemann sum approximation that 
 \begin{align}
\begin{split}
 \s_{N,  L} :\! &=  \E\big[\big(\P_{N}u(x)\big)^2\big] 
  = \sum_{\substack{\ld\in\Z_L^2\\|\ld|\leq  N}}\frac{1}{\jb{ \ld }^2}\frac 1{L^2} \\
&  \hspace{0.5mm} = \sum_{\substack{n\in\Z^2\\|n|\leq L N}}\frac{1}{\jb{ \frac nL  }^2}\frac 1{L^2} 
 \sim \int_{\R^2} \ind_{|z|\le N}
\frac{dz}{1 + |z|^2} 
\sim \log  N
\end{split}
\label{sig1}
 \end{align}

\noi
for $N \gg 1$, 
independent of $x\in\T^2_L$,
which diverges to $\infty$ as $N\to \infty$ (for each fixed period $L > 0$).
In particular, $u = \lim_{N \to \infty} \P_N u $ is only a distribution
and thus,  for any integer $\l \ge2 $,  the power $(\P_N u)^\l$
does not converge to any limit.
For each $x \in \T^2_L$, we now define 
 the Wick power $:\!(\P_{N} u)^{\l}(x) \!: \, = \, 
 :\!(\P_{N} u)^{\l}(x) \!:_L$ by\footnote{Note that the definition \eqref{Wick1}
 of the Wick power depends on the period $L > 0$.
We, however, suppress the subscript $L$ from the Wick power, 
when it is clear from the context.}
\begin{align}
:\!(\P_{N} u)^{\l}(x) \!:
 \,  =   H_{\l} (\P_{N} u(x); \s_{N, L}),
\label{Wick1}
\end{align}

\noi
where $H_{\l}(x;\s)$ is the Hermite polynomial of degree $\l \in \Z_{\ge 0}  := \N\cup\{0\}$
with a variance parameter $\s> 0$.
Arguing as in \cite{GKO, GKO2, GKOT}, 
we can show that 
$:\!(\P_{N} u)^{\l} \!:$ converges, almost surely and 
in $L^p(\O)$ for any finite $p \ge 1$, 
to a limit, denoted by 
$:\! u^{\l} \!:$ in $H^s(\T^2_L)$, $s < 0$.
Then, 
by defining
 the truncated renormalized potential energy:
\begin{align}
R_N^L(u) = 
  \frac{1}{k+1} 
\int_{\T^2_L} :\!(\P_N u)^{k+1}(x) \!: dx, 
\label{R1}
\end{align}

\noi
a standard computation shows  that  $\{R_N^L (u) \}_{N \in \N}$
converges to some limit, denoted by $R^L(u)$, 
in  $L^p(d\mu_L)$ for any finite $p \geq 1$, as $N \to \infty$.
The convergence 
of  the truncated renormalized potential energy
$\{R_N^L (u) \}_{N \in \N}$
together with 
Nelson's estimate
implies  that the truncated density
$\{e^{-R^L_N(u)}\}_{N\in \N}$
converges to 
the limiting density 
$e^{-R^L(u)}$ in $L^p(d\mu_L)$ for any finite $p \ge 1$,
 as $N \to \infty$.
Hence, by defining  the renormalized  truncated Gibbs measure:
\begin{align}
d\rhoo_{N, L}(u,\dt u )= \ZZ_{N, L}^{-1}e^{-R_N^L(u)}d\muu_L(u,\dt u), 
\label{Gibbs2}
\end{align}

\noi
we then conclude that 
the renormalized truncated  Gibbs measure $\rhoo_{N, L}$  converges
in total variation 
 to the limiting Gibbs measure $\rhoo_L$ given by
\begin{align}
\begin{split}
d\rhoo_L(u,\dt u )
& = \ZZ_L^{-1} e^{- R^L(u)}d\muu_L(u, \dt u )\\
& = \ZZ_L^{-1} \exp\bigg( - \frac{1}{k+1} 
\int_{\T^2_L} :\! u^{k+1}(x) \!: dx\bigg)d\muu_L(u, \dt u ).
\end{split}
\label{Gibbs3}
\end{align}

\noi
Furthermore, for each fixed $0<   L < \infty$, 
the resulting Gibbs measure $\rhoo_L$ is equivalent\footnote{Namely, 
$\rhoo_L$ and $\muu_L$ are mutually absolutely continuous.} 
to the base Gaussian measure $\muu_L$.

\begin{remark}\label{REM:Gibbs}\rm

The first marginal of the Gibbs measure $\rhoo_L$
(namely, after integrating \eqref{Gibbs3}
in $\dt u$)
is precisely the $\Phi^{k+1}_2$-measure $\rho_L$ on $\T^2_L$, given by 
\begin{align}
\begin{split}
d\rho_L(u)
& = \ZZ_L^{-1} e^{- R^L(u)}d\mu_L(u)\\
& = \ZZ_L^{-1} \exp\bigg( - \frac{1}{k+1} 
\int_{\T^2_L} :\! u^{k+1}(x) \!: dx\bigg)d\mu_L(u ), 
\end{split}
\label{Gibbs4}
\end{align}

\noi
where
$\mu_{L} = \mu_{1, L}$ is as in \eqref{mu1}.
Note that $\rho_L$ is the Gibbs measure associated with the energy functional 
 $E(u)$  in \eqref{Hamil2}.
The Gibbs measure $\rhoo_L$ in \eqref{Gibbs3} can then be written as
$\rhoo_L = \rho_L \otimes \mu_{0,  L}$, 
where $\mu_{0,  L}$ is the (spatial) white noise measure on $\T^2_L$.

\end{remark}

Next, we discuss stochastic dynamics associated with the Gibbs measure
$\rhoo_L$ in \eqref{Gibbs3}.
This process is known as 
 stochastic quantization \cite{PW}.
 In the parabolic setting, 
 Da Prato and Debussche \cite{DD}
 studied the parabolic $\Phi^{k+1}_2$-model \eqref{NLH1} on $\T^2$, 
 associated with the $\Phi^{k+1}_2$-measure $\rho_{L = 1}$ in \eqref{Gibbs4}.
 With the Wick renormalization, 
 they constructed global-in-time invariant dynamics for \eqref{NLH1} on $\T^2$
 with the initial data distributed by  the $\Phi^{k+1}_2$-measure~$\rho_{L = 1}$. 
 This result is readily applicable to 
 the parabolic $\Phi^{k+1}_2$-model \eqref{NLH1} posed on 
 the dilated torus~$\T^2_L$ for any $L > 0$.
 In \cite{MW17a}, with an intricate use of weighted Besov spaces (see \eqref{besov1} below), 
 Mourrat and Weber extended this result
 to  the parabolic $\Phi^{k+1}_2$-model on the  plane~$\R^2$.

We now consider the hyperbolic $\Phi^{k+1}_2$-model on $\T^2_L$
for fixed $L > 0$:
\begin{align}
\dt^2 u + \dt u + (1- \Dl)u \,  +  u^k = \sqrt{2}\xi_L
\label{SNLW0}
\end{align}

\noi
with the Gibbsian initial data distributed by the Gibbs measure $\rhoo_L$
in \eqref{Gibbs3}, 
where $\xi_L$ is a space-time white noise on $\T^2_L\times \R_+$.
In view of the equivalence of the Gibbs measure $\rhoo_L$ in~\eqref{Gibbs3}
and the base Gaussian measure $\muu_L$ in \eqref{mu1}
for each fixed  $L > 0$, 
it suffices to study~\eqref{SNLW0} with the 
Gaussian initial data distributed by 
$\muu_L$.

Let $\Phi = \Phi_L$  be the solution to the linear stochastic damped wave equation on $\T^2_L$
with the Gaussian initial data distributed by~$\muu_L$ in~\eqref{mu1}:
\begin{align*}
\begin{cases}
\dt^2 \Phi + \dt\Phi +(1-\Dl)\Phi  = \sqrt{2}\xi_L\\
(\Phi,\dt\Phi)|_{t=0}=(\phi_0,\phi_1)
\quad \text{with }\ 
\Law (\phi_0, \phi_1) = \muu_L.
\end{cases}
\end{align*}

\noi
Here, $\Law(X)$ of a random variable $X$ denotes the law of $X$.
Define the linear damped wave propagator $\D(t)$ by 
\begin{equation}
\D(t) = e^{-\frac{t}2}\frac{\sin\Big(t\sqrt{\frac34-\Dl}\Big)}{\sqrt{\frac34-\Dl}}
\label{lin2}
\end{equation} 

\noi
as a Fourier multiplier operator.
Then, the stochastic convolution $\Phi$ defined above can be expressed as 
\begin{align*} 
\Phi (t) 
 = \big(\dt\D(t) +\D(t)\big)\phi_0 + \D(t)\phi_1+ \sqrt{2}\int_0^t\D(t - t')dW_L(t'), 
\end{align*}

\noi
where  $W_L$ denotes a cylindrical Wiener process on $L^2(\T^2_L)$:
\begin{align}
W_L(t)
 = \sum_{\ze \in \Z_L^2 } B_{\ze} (t) e_\ze^L
\label{WW1}
\end{align}

\noi
and  
$\{ B_\ze \}_{\ze \in \Z^2_L}$ 
is defined by 
$B_{\ze}(t) = \jb{\xi_L, \ind_{[0, t]} \cdot e_\ze^L}_{\T^2_L\times \R_+}$.
Here, $\jb{\cdot, \cdot}_{\T^2_L\times \R_+}$ denotes 
the duality pairing on $\T^2_L\times \R_+$.
Then, we see that  $\{ B_\ze \}_{\ze \in \Z^2_L}$ is a family of mutually independent complex-valued\footnote
{In particular, $B_0$ is  a standard real-valued Brownian motion.} 
Brownian motions conditioned  that $B_{-\ze} = \cj{B_\ze}$, $\ze \in \Z^2_L$. 
Note that  $\text{Var}(B_\ze(t)) = t$, $\ze \in \Z^2_L$.

Let $N \in \N$.
Given   $(x, t) \in \T^2_L \times \R_+$,
we see that 
 $\Phi_N(x, t)=\P_N\Phi(x, t)$
 is a mean-zero real-valued Gaussian random variable with variance
 \begin{align*}
 \E \big[\Phi_N^2(x, t)\big] = \s_{N, L} \sim \log N \too \infty
 \end{align*}

\noi
as $N \to \infty$, where $\s_{N, L}$ is as in \eqref{sig1}.
As in \eqref{Wick1}, we define the Wick power $:\!\Phi_N^{\l}(x, t) \!:$ by setting
\begin{align}
:\!\Phi_N^{\l}(x, t) \!:
 \, =   H_{\l} (\Phi_N(x, t); \s_{N, L}).
\label{Wick2}
\end{align}
 
 \noi
Then, 
$:\!\Phi_N^{\l} \!:$ converges, almost surely and 
in $L^p(\O)$ for any finite $p \ge 1$, 
to a limit, denoted by 
$:\!\Phi^{\l} \!:$\,,  in $C(\R_+; H^s(\T^2_L))$, $s < 0$.\footnote{Here, 
we endow the space
$C(\R_+; H^s(\T^2_L))$ with the compact-open topology in time, 
namely with the topology of uniform convergence on compact intervals.
}

Given $N \in \N$, 
consider the following truncated SdNLW on $\T^2_L$:
\begin{align}
\dt^2 u_N   + \dt u_N  +(1-\Dl)  u_N 
+
\P_N\big((\P_N u_N)^{k}  \big) 
   = \sqrt{2} \xi_L .
\label{SNLW2}
\end{align}

 \noi
 Proceeding with the first order expansion (\cite{McKean, BO96, DD}):
\begin{equation}
u_N = \Phi_N + v_N,
\label{decomp1}
\end{equation}

\noi
 we see that 
 the remainder  term $v_N = u_N - \Phi_N$ satisfies 
 \begin{equation}
\dt^2 v_N + \dt v_N + (1 - \Dl) v_N +  
\sum_{\ell=0}^k {k\choose \ell} \P_N \big(\Phi_N^\ell v_N^{k-\ell}\big) = 0
\label{SNLW3}
\end{equation}

\noi
with the zero initial data.
As pointed out above, 
the power $\Phi_N^\ell$ does not converge to any limit  as $N\to \infty$. 
Furthermore, 
a triviality is known for \eqref{SNLW3} (at least for $k = 3$).
Namely, 
the solution $u_N$ to \eqref{SNLW2}
tends to 0 as we remove the regularization ($N \to \infty$); see \cite{OOR}.
This triviality result necessitates the use of a renormalization.
Therefore, 
we instead consider the following renormalized version of \eqref{SNLW3}:
 \begin{equation}
\dt^2 v_N + \dt v_N + (1 - \Dl) v_N + 
 \sum_{\ell=0}^k {k\choose \ell} \P_N \big( :\!\Phi_N^\l\!:  v_N^{k-\ell}\big) = 0
\label{SNLW3a}
\end{equation}

\noi
with the zero initial data.
By formally taking a limit as $N \to \infty$, 
we then obtain the limiting equation:
 \begin{equation}
\dt^2 v  +\dt v+  (1 - \Dl) v +  \sum_{\ell=0}^k {k\choose \ell} :\!\Phi^\l\!:  v^{k-\ell} = 0.
\label{SNLW4}
\end{equation}

\noi
Given the almost sure space-time regularity of the Wick powers
$\{\, :\!\Phi^\l\!: \,\}_{\l = 1}^k$, standard deterministic analysis with  the product estimates 
(Lemma~\ref{LEM:bilin1})
and Sobolev's inequality yields 
local well-posedness of \eqref{SNLW4}
with the continuous map, 
sending the enhanced data set to the solution:
\begin{align*}
 (\Phi  ,  :\!\Phi^2\!:, \dots, :\!\Phi^k\!:\,) &  \in \big(C([0, T]; H^{-\eps}(\T^2_L))\big)^{\otimes k} \\
& \longmapsto (v, \dt v) \in C([0, T]; \vec H^{1-\eps}(\T^2_L))
\end{align*}

\noi
for some small $\eps > 0$.
This deterministic local well-posedness analysis also applies
to \eqref{SNLW3a}, uniformly in $N \in \N$.

In view of the decomposition~\eqref{decomp1}, 
the renormalized truncated SdNLW  \eqref{SNLW3a} for $v_N$
corresponds to  the following
 the renormalized truncated  SdNLW for $u_N = \Phi_N + v_N $:
\begin{align}
 \dt^2 u_N +\dt u_N +  (1 -  \Dl)  u_N   + \P_N \big(:\! (\P_N u_N)^k\!:  \big) =  \sqrt 2\xi_L.
\label{SNLW5}
 \end{align}

\noi
Here, the renormalized nonlinearity in \eqref{SNLW5}
 is interpreted as 
\begin{align*}  
\P_N \big(:\! (\P_N u_N)^k\!:  \big)
& =  \P_N \big(:\! (\Phi_N + \P_N v_N)^k\!:  \big)\\
& = 
 \sum_{\ell=0}^k {k\choose \ell} \P_N \big(:\!\Phi_N^\l\!:  v_N^{k-\ell}\big).
\end{align*}

\noi
Then, 
the aforementioned local well-posedness of \eqref{SNLW4}, 
together with the convergence of 
the truncated enhanced data set 
$\{\, :\!\Phi_N^\l\!: \,\}_{\l = 1}^k$
to the limiting enhanced data set
$\{\, :\!\Phi^\l\!: \,\}_{\l = 1}^k$, implies
that $u_N$ 
converges almost surely to  a stochastic 
process $u
 = \Phi + v$, 
where $v$ satisfies~\eqref{SNLW4}.
It is in this sense that we say that 
the renormalized SdNLW on the dilated torus $\T^2_L$:
\begin{align}
\dt^2 u + \dt u +  (1 -  \Dl)  u   +  :\!u^k\!: \,= \sqrt 2 \xi_L
\label{SNLW6}
\end{align}

\noi
is locally well-posed with the Gaussian initial data distributed by $\muu_L$, 
and hence with the Gibbsian initial data in view of the equivalence of 
the Gibbs measure $\rhoo_L$ and the base Gaussian measure $\muu_L$.

Once the local-in-time dynamics is constructed, 
Bourgain's invariant measure argument~\cite{BO94, BO96}
allows us to construct global-in-time dynamics
for the hyperbolic $\Phi^{k+1}_2$-model~\eqref{SNLW6} on $\T^2_L$
and to prove invariance of the Gibbs measure $\rhoo_L$ in \eqref{Gibbs3}.
This argument is  based essentially on the following three ingredients:

\smallskip

\begin{itemize}
\item invariance of the truncated Gibbs measure $\rhoo_{N, L}$ in \eqref{Gibbs2}
under the truncated SdNLW dynamics \eqref{SNLW5}, 
which provides a probabilistic growth bound on the solution $u_N$, uniformly in $N \in \N$, 

\smallskip

\item a PDE approximation argument
(analogous to the local well-posedness argument in the current setting),

\smallskip

\item convergence (in total variation) of the truncated Gibbs measure $\rhoo_{N, L}$ in \eqref{Gibbs2}
to the limiting Gibbs measure $\rhoo_L$ in \eqref{Gibbs3}.

\end{itemize}

\noi
See \cite{ORTz} for full details of this argument.

\subsection{Main result}

Our goal in this paper is to extend the construction of the invariant Gibbs dynamics
of the hyperbolic $\Phi^{k+1}_2$-model 
on $\T^2_L$ described in the previous subsection to the plane~$\R^2$.
The main idea is to apply
Bourgain's invariant measure argument `in spirit', 
where we replace
the frequency truncation parameter $N \to \infty$
by the growing period $L \to \infty$.

We first state our main result.
Given $s \in \R$, $1 \le p < \infty$, 
and $\mu > 0$, 
let 
$W^{s, p}_\mu(\R^2)$ be  the weighted Sobolev space 
defined in \eqref{WSP} below.
We set 
\begin{align*}
\vec W^{s, p}_\mu(\R^2)
=  W^{s, p}_\mu(\R^2) \times  W^{s-1, p}_\mu(\R^2).
\end{align*}

\noi
We also set
\begin{align*}
\vec H_\text{loc}^{s}(\R^2)
=  H_\text{loc}^{s}(\R^2) \times  H_\text{loc}^{s-1}(\R^2).
\end{align*}

\noi
Here,  $H_\text{loc}^{s}(\R^2)$ denotes the class of functions $u$ belonging
to $H^s(K)$ for each compact subset $K \subset \R^2$,
where 
the 
$H^s(K)$-norm  is defined as the restriction norm onto the set $K$;
see \eqref{loc1} below.
The topology on 
$H_\text{loc}^{s}(\R^2)$ is induced by the following metric:
\[
d_{H_\text{loc}^{s}}(f, g) 
= \sum_{j=1}^\infty 2^{-j}
 \frac{  \| f - g\|_{H^{s}([-j, j]^2)}    }{1 +  \| f - g\|_{  H^{s}([-j, j]^2) }  }.
\]

\noi
By definition, we have $d_{H_\text{loc}^{s}}(f_n, f)\to 0$ if and only if
$f_n$ converges to $f$ in $H^{s}([-j, j]^2)$ for each $j\in\N$.
In the following, we endow
$C(\R_+; \vec H_\textup{loc}^{s}(\R^2))$
with the compact-open topology in time.\footnote{
We say that $u \in C(\R_+; \vec H_\textup{loc}^{s}(\R^2))$
if $u \in C(\R_+; \vec H^{s}(K))$
 for each compact subset $K \subset \R^2$.
 }

Throughout  the paper, we assume that 
random initial data are independent of a stochastic forcing
(such as $\xi_L$ and $\xi$).
Given $R> 0$, let 
 $B_R = \{ x \in  \R^2: |x|\le R\}$ denotes the (closed) ball of radius $R$ centered at the origin
and 
$\bC_R$ denote the cone given by 
\begin{align}
\begin{split}
\bC_R 
& = \{ (x, t)\in\R^2\times\R_+:  |x| + |t| \leq R   \} \\
& = \{ (x, t)\in\R^2\times [0, R]:  x \in B_{R-t}  \} .
\end{split}
\label{cone}
\end{align}

\begin{theorem}\label{THM:1}
Let $k \in 2\N+1$.
There exists a subset 
 $\A = \{ L_j: j \in \N\} \subset \N$
\textup{(}with $L_j < L_{j'}$ for $j < j'$\textup{)}
such that the following statements hold.

\smallskip

\noi
\textup{(i)} 
Let $s < 0$, finite $p \ge 1$, and $\mu > 0$.
Let $\rhoo_L$ be the Gibbs measure 
on the dilated torus $\T^2_{L}$ defined in \eqref{Gibbs3}.
When viewed as a probability measure on 
the weighted Sobolev space~$\vec W^{s, p}_\mu(\R^2)$, 
the $L_j$-periodic Gibbs  measure $\rhoo_{L_j}$ 
converges 
weakly to a limiting Gibbs measure $\rhoo_\infty$
as $j \to \infty$.
The limiting Gibbs measure $\rhoo_\infty$ on $\R^2$
can be written as 
\[\rhoo_\infty = \rho_\infty \otimes \mu_{0, \infty},\]

\noi
where $\rho_\infty$ is the $\Phi^{k+1}_2$-measure
on $\R^2$ constructed as a limit
of the $L_j$-periodic $\Phi^{k+1}_2$-measure $\rho_{L_j}$,
and $\mu_{0, \infty}$ is the white noise measure on $\R^2$.

\smallskip

\noi
\textup{(ii)} 
The hyperbolic $\Phi^{k+1}_2$-model  on $\R^2$\textup{:}
\begin{align}
\dt^2 u + \dt u +  (1 -  \Dl)  u \,   +  :\!u^k\!: \,= \sqrt 2 \xi
\label{SNLW7}
\end{align}

\noi
is globally well-posed
almost surely with respect to the Gibbsian initial data
distributed by the Gibbs measure $\rhoo_\infty$ on $\R^2$, 
constructed in Part \textup{(i)}.
Furthermore, 
 the  Gibbs measure $\rhoo_\infty$ is invariant under the resulting 
 hyperbolic $\Phi^{k+1}_2$-dynamics on $\R^2$.

More precisely,
the following statements hold.
There exist
a sequence 
 $\{(u_{L_j}, \dt u_{L_j})\}_{j \in \N}$ 
 and
 a 
non-trivial stochastic process $(u, \dt u)$
almost surely belonging to 
$C(\R_+; \vec H_\textup{loc}^{-\eps}(\R^2))$
for any $\eps >0$
 such that 

\smallskip

\begin{itemize}
\item for each $j \in \N$, 
 $(u_{L_j}, \dt u_{L_j})$ 
 is  the global-in-time solution to   SdNLW~\eqref{SNLW6} on $\T^2_{L_j}$
 with 
 $\Law \big((u_{L_j}(0), \dt u_{L_j}(0))\big) = \rhoo_{L_j}$, constructed in \cite{GKOT},

\smallskip

\item 
as $j \to \infty$, 
$(u_{L_j}(0), \dt u_{L_j}(0))$
converges almost surely to $(u(0), \dt u(0))$,
distributed by the 
limiting  Gibbs measure $\rhoo_\infty$,
in 
$ \vec H_\textup{loc}^{-\eps}(\R^2)$,

\smallskip

\item 
 given any $R>0$,  
 $(u_{L_j}, \dt u_{L_j})$ 
converges in probability to $(u, \dt u)$ 
on the cone $\bC_R$
\textup{(}more precisely, in $L^\infty([0, R]; \vec H^{-\eps}(B_{R-t}))\textup{;}~see \eqref{space1})$\textup{)}, 
as $j \to \infty$.

\end{itemize}

%
%
%

\noi
Furthermore, the law of $(u(t), \dt u(t))$ for any $t \in \R_+$
is given by the renormalized Gibbs measure $\rhoo_\infty$.

\end{theorem}

In Theorem \ref{THM:1}\,(i), 
we needed to take a sequence $\{L_j\}_{j \in \A}$
 due to the non-uniqueness of the limiting Gibbs measure
on $\R^2$; see Remark \ref{REM:sing}\,(ii).

As in the periodic case, 
the limit 
$(u, \dt u)$
constructed in Theorem \ref{THM:1}\,(ii)
is a solution to 
 the  hyperbolic $\Phi^{k+1}_2$-model~\eqref{SNLW7} on $\R^2$
with the Gibbsian initial data
in the sense that 
the remainder 
$v = u -   \Phi$ 
satisfies the equation \eqref{SNLW4}
on $\R^2\times \R_+$, 
where $\Phi$ denotes 
the solution to the linear stochastic damped wave equation on $\R^2$
with the Gibbsian initial data:
\begin{align*}
\begin{cases}
\dt^2 \Phi + \dt\Phi +(1-\Dl)\Phi  = \sqrt{2}\xi\\
(\Phi,\dt\Phi)|_{t=0}=(u_0, u_1)
\quad \text{with }\ 
\Law (u_0, u_1) = \rhoo_\infty.
\end{cases}
\end{align*}

\noi
We note that 
the solution $(v, \dt v)$ to \eqref{SNLW4} 
thus constructed is 
unique in $C(\R_+; \vec H^{1-\eps}_\text{loc}(\R^2))$,
which follows from 
the finite speed of propagation
and a local well-posedness result
presented in 
 Subsection \ref{SUBSEC:4.1}.
Namely, 
the solution $(u, \dt u)$ 
to  the  hyperbolic $\Phi^{k+1}_2$-model~\eqref{SNLW7} on $\R^2$
with the Gibbsian initial data, 
constructed in Theorem \ref{THM:1}\,(ii), 
is unique in the class
\[ (\Phi, \dt \Phi) + 
C(\R_+; \vec H^{1-\eps}_\text{loc}(\R^2)).\]

As for the construction 
of the limiting $\Phi^{k+1}_2$-measure $\rho_\infty$ on $\R^2$, 
see also \cite{Simon, BG1}
and the references therein.\footnote{See also  \cite{DDJ} which appeared after the 
the first version of this paper.}
Note that the assumption on the almost sure convergence of 
$(u_{L_j}(0), \dt u_{L_j}(0))$
in Theorem \ref{THM:1}\,(ii)
indeed follows from the weak convergence
of  $\rhoo_{L_j}$ 
to  $\rhoo_\infty$
 in Theorem \ref{THM:1}\,(i) and 
the Skorokhod representation theorem
(Lemma \ref{LEM:Sk})
and therefore is not an additional assumption.

Stochastic nonlinear wave equations (SNLW)
 have attracted extensive attention from both applied and theoretical 
 points of view; 
see \cite[Chapter 13]{DZ} and \cite{OOcomp} for the references therein.
In particular, over the last five years, we have seen a significant progress
in the well-posedness theory of 
 SNLW in the singular setting:\footnote{Some of the works mentioned below
are on SNLW without damping.}
\begin{align*}
\dt^2 u + \dt u + (1 -  \Dl)  u + \NN(u)  =  \zeta, 
\end{align*}

\noi
where the noise $\zeta$
is primarily taken to be  a space-time  white noise $\xi$.
Here,  $\NN(u)$ denotes a nonlinearity
which may be 
of a power-type 
 \cite{GKO, GKO2, GKOT,   ORTz, OOR, Tolo2, OOT1, Bring2, OOT2, OWZ, BDNY}
and trigonometric and exponential nonlinearities
\cite{ORSW,  ORW,  ORSW2, GHOZ, OZ2}.
We also  mention the works 
 \cite{OTh2,   OPTz,  OOTz, STzX, OZ, OOPTz}
 on the (deterministic) 
 nonlinear wave equations \eqref{NLW2} with  rough random initial data
and
\cite{Deya1, Deya2,  OOcomp} 
on SNLW with 
more singular (both in space and time) noises.
We point out that
the only known well-posedness  result up to date
for singular SNLW posed  on an unbounded domain
is the work  \cite{Tolo2}  by the second author, 
where he  established
{\it pathwise} global well-posedness
of the cubic SNLW on $\R^2$ with an additive space-time white noise forcing:
\begin{align}
\dt^2 u  + (1- \Dl)u \,  +  u^k = \xi,
\label{NLW4}
\end{align}

\noi
where $k = 3$.
Note that a slight modification of the argument yields
pathwise global well-posedness
for SdNLW \eqref{NLW1} on $\R^2$ when $k = 3$.

Theorem \ref{THM:1} provides 
the second well-posedness result for singular SNLW on an unbounded domain.
Our construction of the global-in-time dynamics, however, 
is quite different from that in \cite{Tolo2}.
In particular, it is 
not pathwise,  but is based on 
a probabilistic argument, more precisely, 
on Bourgain's invariant measure argument `in spirit'.
Here, by pathwise global well-posedness, 
we mean 
global well-posedness for  any {\it deterministic}
initial data  in a given function space, 
where only  a priori control  of explicitly given stochastic terms 
is used
(namely the statistical information of a solution is not used), 
whereas 
Bourgain's invariant measure argument
crucially relies on  the statistical information 
of solutions
(to approximating equations).
We point out that 
pathwise global well-posedness
(in the sense described above)
of SdNLW \eqref{NLW1} or SNLW \eqref{NLW4}
for the (super-)quintic case $k \ge 5$ remains 
 a challenging open question
even on the torus $\T^2$,\footnote{As pointed out in \cite{Tolo2}, 
local well-posedness
of of SdNLW \eqref{NLW1} or SNLW \eqref{NLW4} on $\R^2$ 
essentially requires a global control
(which also yields global well-posedness)
and remains open for $k \ge 5$.} 
and therefore, 
the situation for the stochastic wave equation is completely different 
from SNLH \eqref{NLH1} on $\R^2$, 
where Mourrat and Weber \cite{MW17a}
proved pathwise global well-posedness of \eqref{NLH1} on $\R^2$
 for any $k \in 2\N + 1$.

\medskip
%

Thanks to the finite speed of propagation, 
the argument in \cite{GKO, GKOT}
yields local well-posedness of \eqref{SNLW7}
on the ball $B_R \subset \R^2$ for each $R > 0$; see Proposition \ref{PROP:LWP} below.
As $R \to \infty$, 
however, the local existence time shrinks to $0$.
In order to construct a solution to~\eqref{SNLW7}
on some time interval $[0, \tau]$, uniformly on $\R^2$, 
we need to make use of statistical ingredients,
and thus, establishing local well-posedness of \eqref{SNLW7}, 
uniformly
on $\R^2$, is essentially as difficult as 
establishing its global  well-posedness, 
as already observed in \cite{Tolo2}.

On the dilated torus $\T^2_L$, 
 the Gibbs measure $\rhoo_L$
and the base Gaussian measure $\muu_L$ are equivalent for each finite $L > 0$.
However, this equivalence of $\rhoo_L$ and $\muu_L$ is not uniform when $L \to \infty$,
since 
 the potential energy
$R^L$, defined as the limit of $R^L_N$ in \eqref{R1}, 
grows like $\sim L$ as $L \to \infty$.
Indeed from 
\eqref{R1}, \eqref{Wick1}, and Lemma \ref{LEM:W1} with \eqref{exp1}, we have
\begin{align*}
 \E_{\muu_L}&   \Big[\big(R^L(u)\big)^2\Big]
 = \lim_{N \to \infty} \E_{\muu_L}  \Big[\big(R_N^L(u)\big)^2\Big]\\
&  = C_k 
\lim_{N \to \infty}
\iint_{\T^2_L \times \T^2_L} \E_{\mu_L}\Big[ H_{k+1}(\P_N u(x); \s_{N, L})
H_{k+1}(\P_N u(y); \s_{N, L}) \Big]dx dy \\
&  = C_k 
\lim_{N \to \infty} 
\iint_{\T^2_L \times \T^2_L} \Big(\E_{\mu_L}\big[ \P_N u(x)\P_N u(y) \big]\Big)^{k+1}dx dy \\
&  = C_k 
\lim_{N \to \infty} 
\iint_{\T^2_L \times \T^2_L} 
\bigg(\sum_{\substack{\ld \in \Z^2_L\\ |\ld| \le N}}
\frac{e_\ld^L(x-y)}{\jb{\ld}^2}\frac 1L \bigg)^{k+1} dx dy \\
&  = C_k L^2
\lim_{N \to \infty}\sum_{\substack{n_1, \dots, n_{k+1} \in \Z^2\\ 
n_1 + \cdots +n_{k+1} = 0\\ |n_j| \le LN}}
\bigg(\prod_{j = 1}^{k+1}\frac{1}{\jb{\frac{n_j}{L}}^2}\bigg) \frac{1}{L^{2k}} \\
& \sim L^2, 
\end{align*}

\noi
where the last step follows from a Riemann sum approximation.
See also Remark \ref{REM:sing}\,(i).
This non-uniformity causes difficulty in studying the large torus limit $L \to \infty$.
In the one-dimensional case, 
there is no need  for a renormalization in constructing a Gibbs measure, 
and 
Bourgain \cite{BO00} used the 
Brascamp-Lieb concentration inequality \cite[Theorem~5.1]{BL2}
to reduce the relevant analysis to that for the Gaussian case, 
uniformly in the period $L\gg1$.
In the current two-dimensional case, 
due to the use of the renormalization, 
the log-concavity needed for 
the 
Brascamp-Lieb concentration inequality
is not available, and thus 
we need an alternative approach.

In this work, 
the main statistical control
comes from the study on the so-called 
{\it enhanced Gibbs measures}, 
namely the distributions
of the enhanced data sets.
This idea played a crucial  role
in a recent work \cite{OOT2}
on the hyperbolic $\Phi^3_3$-model on 
the three-dimensional torus $\T^3$ by the first two authors
and Okamoto.
Given  $R>0$, 
our goal is to construct a solution to~\eqref{SNLW7}
on the cone $\bC_R$ in \eqref{cone}
as a limit of the solution $u_L$ to the $L$-periodic problem~\eqref{SNLW6};
see~\eqref{NW2} below.
When $L \gg R$, 
we see that the $L$-periodic spatial white noise
$\z_L$ defined in~\eqref{wh1}
and
the $L$-periodic space-time white noise $\xi_L$
defined in~\eqref{wh4}
agree on the cone  $\bC_R$ with 
the spatial white noise 
$\z$ on $\R^2$ (see Definition \ref{DEF:white})
and the space-time white noise on $\R^2\times \R_+$, respectively;
see \eqref{wh7} below.
Therefore, by restricting our attention to the cone $\bC_R$, 
we conclude from 
 the finite speed of propagation
 and the observation above 
 that the difference
of the $L$-periodic
problem \eqref{NW2} 
for different values of $L\ge1$
appears {\it only in the first component}~$u_{0, L}$ of the initial data
with 
$\Law (u_{0, L}) = \rho_L$, 
where $\rho_L$ is the $L$-periodic $\Phi^{k+1}_2$-measure in \eqref{Gibbs4}.
This essentially reduces
the  convergence problem on the cone $\bC_R$
 of the $L$-periodic
problem \eqref{NW2} to 
studying convergence properties
of the enhanced Gibbs measure 
\begin{align}
\nu_L = \Law(\Xi_0(u_{0, L})), 
\label{ex1}
\end{align}

\noi
where $\Xi_0(u_{0, L})$ denotes
the enhanced data set,\footnote{Namely, the Wick powers of the associated linear solution.} emanating from the first component 
$u_L|_{t = 0} =  u_{0, L}$ of the initial data
with $\Law (u_{0, L}) = \rho_L$;
see \eqref{nuL} and \eqref{xi1}.
As mentioned above, 
by restricting our attention to the cone $\bC_R$, 
when $L \gg R$, 
 the second enhanced data set $\Xi_1(u_1, \xi)$
in \eqref{xi2} (see also \eqref{NW5}), 
involving the second component $u_1$
of the initial data and the space-time white noise forcing $\xi$, 
does {\it not} depend on $L$, which simplifies some analysis.
As for the
 second enhanced data set $\Xi_1(u_1, \xi)$, 
 its {\it mapping properties} 
 (viewing its elements as (random) multiplication operators) play an important role;
 see Definition \ref{DEF:X} and Proposition \ref{PROP:need23}.

Let us briefly  describe  four main steps of the proof of Theorem \ref{THM:1}.

\smallskip
\noi
$\bullet$ {\bf Step 1:}
Coming down from infinity
for the associated  SNLH on $\R^2$.

In this first step, we establish
coming down from infinity 
(namely, an estimate on a solution 
independent of initial data)
for
SNLH 
\eqref{SHE5}
in (i)~weighted Lebesgue spaces $L^p_\mu(\R^2)$ (Proposition \ref{PROP:CI1})
and (ii)~weighted Sobolev spaces $W^{s, p}_\mu(\R^2)$
of positive regularities (Proposition~\ref{PROP:CI3}).
The coming down from infinity in weighted Lebesgue
spaces
 yields tightness
of the $L$-periodic Gibbs measures $\rhoo_L$,
which then allows us to extract a sequence $\{\rhoo_L\}_{L \in \A}$
converging weakly to a limiting Gibbs measure~$\rhoo_\infty$;
see Subsection \ref{SUBSEC:3.3}.
On the other hand, the coming down from infinity 
in weighted Sobolev spaces of positive regularities
plays a crucial role in Step 3.

In recent years, coming down from infinity has been studied 
in the context of (singular) SNLH; see, for example,  \cite{TW18, MoiW1, MoiW2}.
See the introduction in \cite{MoiW1}
for a further discussion.
In the case of weighted Lebesgue spaces (Proposition~\ref{PROP:CI1}), 
our argument 
follows closely that in~\cite{TW18}
and aims to establish a certain differential inequality (see Lemma~\ref{LEM:CI2}).
Due to the use of the weight, however, the argument is more complicated
and requires a careful decomposition of the physical space into unit cubes, 
combined with the Littlewood-Paley decomposition; see \eqref{C6}-\eqref{C15}.
In the case of weighted Sobolev spaces of positive regularities
(Proposition~\ref{PROP:CI3}), 
our argument is based on a Gronwall-type argument
with the coming down from infinity for weighted Lebesgue spaces.
We present details in  Section \ref{SEC:3}.

\smallskip
\noi
$\bullet$ {\bf Step 2:} Local well-posedness and stability of SdNLW \eqref{SNLW7}
on the cone $\bC_R$.

This second step
 is entirely deterministic, by viewing initial data $(u_0, u_1)$
 and a forcing $\xi$ as given deterministic spatial\,/\,space-time distributions,  and follows from a slight modification
of the local well-posedness argument in \cite{GKOT}.
Here, stability refers
to that with respect to the enhanced data set
$\Xi_0(u_0)$ in 
\eqref{xi1}, 
emanating from 
the first component $u_0$ of the initial data.
See Subsection \ref{SUBSEC:4.1} for details.

\smallskip
\noi
$\bullet$ {\bf Step 3:}
Convergence of the enhanced Gibbs measures.

In this step and Step 4, we restrict our attention to $L\in \A \subset \N$
constructed in Step 1.  This is due to the non-uniqueness of the limiting Gibbs measure
on $\R^2$; see Remark \ref{REM:sing}\,(ii).

Our main goal in this step is to prove convergence of $\{\nu_L\}_{L \in \A}$
to a natural limit 
\begin{align}
\nu_\infty = \Law (\Xi_0(u_{0}))
\label{ex2}
\end{align}
with $\Law (u_0) = \rho_\infty$ constructed in Step 1.
Here, the mode of convergence
is weak convergence as well as convergence in 
the Wasserstein-1 metric.
The proof is broken into two parts, 
where we first establish tightness
of $\{\nu_L\}_{L \in \A}$ (Proposition \ref{PROP:tight}) and then show that the limit is
indeed unique, given by $\nu_\infty$ in \eqref{ex2}
(Proposition \ref{PROP:uniq}).

The (first) enhanced data set 
$\Xi_0(u_{0, L})$
consists of the Wick powers 
$: \!(\S(t) u_{0, L})^\l\! :$\,, $\l = 1, \dots, k$, 
of the linear solution 
$\S(t) u_{0, L}
= (\dt \D(t) + \D(t))u_{0, L}$
with $\Law (u_{0, L}) = \rho_L$, 
where $\rho_L$ is the $L$-periodic $\Phi^{k+1}$-measure in~\eqref{Gibbs4}.
Hence, 
in view of the invariance of $\rho_L$ under the parabolic $\Phi^{k+1}_2$-model, 
instead of  studying  $\Xi_0(u_{0, L})$, 
it suffices to  study $\Xi_0(X_L(1))$, 
where $X_L$ is the solution to the $L$-periodic
SNLH \eqref{SHE7} (namely, the parabolic $\Phi^{k+1}_2$-model)
with the Gibbsian initial data $\Law(X_{0, L}) = \rho_L$.
As a result, the argument involves an intricate combination of wave and heat analysis
(see, in particular, the proof of Proposition \ref{PROP:uniq})
as well as the coming down from infinity in weighted Sobolev spaces
of positive regularities.
See Subsection \ref{SUBSEC:4.2} for details.

\smallskip
\noi
$\bullet$ {\bf Step 4:} Global well-posedness and invariance of the limiting Gibbs measure.

We first prove well-posedness of the hyperbolic $\Phi^{k+1}_2$-model
on the cone $\bC_R$ for each $R > 0$.
In view of the global well-posedness of the $L$-periodic hyperbolic $\Phi^{k+1}_2$-model
\eqref{NW2}, 
the local well-posedness and stability results established in Step 2
allow us to reduce the problem
to estimating the size of the first enhanced data set $\Xi_0(u_{0, L})$, 
studying the mapping properties of the second enhanced data set 
$\Xi_1(u_2, \xi)$, 
and convergence in the Wasserstein-1 metric
of the enhanced Gibbs measure $\nu_L$ to $\nu_\infty$
established in Step 3.
Invariance of the limiting Gibbs measure $\rhoo_\infty$
then follows from the weak convergence of $\{\rhoo_L\}_{L\in \A}$ to $\rhoo_\infty$
(Theorem \ref{THM:1}\,(i)), 
the convergence in law, as $\D'(\R_2\times \R_+)$-valued random variables,  of the solution $u_L$, $L \in \A$, 
to the $L$-periodic hyperbolic $\Phi^{k+1}_2$-model \eqref{NW2}
to the solution $u$ to the hyperbolic $\Phi^{k+1}_2$-model~\eqref{NW1} on $\R^2$,
which follows as a corollary of the global well-posedness
(Remark \ref{REM:conv}).

\medskip

In view of the finite speed of propagation,
when we work on the cone $\bC_R$,  
the idea of breaking the enhanced data set
into two groups $\Xi_0(u_{0, L})$ in \eqref{xi1}
and $\Xi_1(u_{1}, \xi)$ in \eqref{xi2}, 
where the latter is independent of $L \gg R$, 
seems new but is natural.
Let us make a brief comparison to  the work 
\cite{OOT2} on the hyperbolic
$\Phi^3_3$-model by the first two authors and Okamoto.
What is common is the use of the enhanced Gibbs measures
(= the laws of the enhanced data sets
of the $L$-periodic (or  frequency-truncated) Gibbs measures), 
while, in the current paper, 
we only consider the enhanced Gibbs measure
emanating from the first component of the initial data
for the reason explained above.
In this work, 
we reduce various problems to those
for the associated stochastic heat equation, 
which leads to an interesting mixture of wave and heat analysis.
Lastly, we remark that, in a recent remarkable preprint \cite{BDNY}
resolving a challenging open problem on well-posedness of the hyperbolic $\Phi^4_3$-model, 
a mixture of wave and heat analysis also played an important role
(but in a different context from our analysis).

\begin{remark}\rm
(i) A slight modification of the proof of Theorem \ref{THM:1}
applies to  the deterministic NLW \eqref{NLW2} on $\R^2$
with the Gibbsian initial data, thus yielding
global well-posedness of (the renormalized version of) \eqref{NLW2}
with $\Law (u(0), \dt u(0)) = \rhoo_\infty$
and invariance of $\rhoo_\infty$ under the resulting dynamics.\footnote{In 
a recent preprint \cite{BL22}, 
Barashkov and Laarne independently obtained 
global well-posedness and invariance of the Gibbs measure 
 the deterministic NLW \eqref{NLW2} on $\R^2$
 (with $k = 3$).}
See  \cite{MV} for the one-dimensional case.
We also mention \cite{BO00, BTT, Deng, Xu, CdS1, CdS2, KMV, Bring3, OQS}
for works on the 
construction of invariant Gibbs dynamics
for Hamiltonian PDEs 
on  unbounded domains,
(including those with confining potentials).

\smallskip

\noi
(ii) In this paper, we only consider the defocusing case $k \in 2\N+1$.
If (a)~$k \in 2\N+2$ or (b)~$k \in 2\N+ 1$
but with the $-$ sign on the potential energy $\frac 1 {k+1}\int u^{k+1} dx$ in \eqref{Hamil1}, 
namely,  the focusing case, 
then it is known that the associated (renormalized) Gibbs measure
on the two-dimensional torus $\T^2$ is not normalizable
even with a taming by a power of the Wick-ordered $L^2$-norm; see 
\cite{OST}.  See also \cite{BS}.
We refer interested readers to \cite{LRS, OST2, OOT1, OOT2}
on the (non-)construction of focusing Gibbs measures on the general $d$-dimensional torus.

When $k = 2$, 
it is possible to construct the Gibbs measure on $\T^2$ with a taming by a power
of the Wick-ordered $L^2$-norm
(see \cite{BO99, OST})
and  the associated invariant Gibbs dynamics
for the hyperbolic $\Phi^3_2$-model on $\T^2$;
see Remark 1.8 in~\cite{OTh2}.
See also \cite{OOT2} for the three-dimensional case.
Due to the non-defocusing nature of the problem, 
however, we expect a certain triviality phenomenon to take place in taking 
a large torus limit of the $L$-periodic $\Phi^3_2$-measure, 
just as in the one-dimensional focusing  case
\cite{Rider, TolW}.\footnote{See a recent preprint
\cite{SS} on the triviality of the $\Phi^3_2$-measure on the plane.}

\smallskip

\noi
(iii) In \cite{Tolo, Tolo3}, the second author
proved ergodicity of the Gibbs measure
for the hyperbolic $\Phi^{k+1}_d$-model posed on $\T^d$ when $d = 1, 2$.
See also \cite{FT}.
It is of interest to study the corresponding problem on $\R^d$.

\end{remark}

\begin{remark}\label{REM:sing}\rm
(i) 
As mentioned above, the equivalence
of the $L$-periodic $\Phi^{k+1}_2$-measure $\rho_L$
and the base Gaussian free field $\mu_L$ on $\T^2_L$
is not uniform.
In fact, it is expected that the limiting $\Phi^{k+1}_2$-measure 
$\rho_\infty$ on $\R^2$ constructed in Theorem \ref{THM:1}\,(i)
and the base Gaussian measure (= the large torus limit of $\mu_L$)
are mutually singular.
We will address this issue 
in a forthcoming work.

\smallskip

\noi
(ii)
The  limiting Gibbs measure $\rho_\infty$  
constructed in Theorem \ref{THM:1}\,(i)
in principle depends on the sequence $L \in \A$
and, in general,  there is no uniqueness statement
for the  limiting Gibbs measure.
See \cite{GJS1, GJS2, GJS3}.
See also the discussion at the end of Section 2
in \cite{BG1}.

In a recent preprint \cite{BDW}
that appeared after the first version of the current paper, 
Bauerschmidt, Dagallier, and Weber
proved uniqueness of the $\Phi^4_2$-measure on the plane
in the high temperature regime.
Namely, by replacing $\frac 1{k+1}$
in~\eqref{Gibbs4}
by $\frac \be {k+1}$ with $k = 3$, 
they proved 
that there exists $\be_* > 0$ such that 
 the 
$\Phi^4_2$-measure on the plane 
with the inverse temperature $\be $
is unique
for $0 < \be <  \be_*$.
As a consequence, 
when $k = 3$, 
 by replacing 
  $\frac 1{k+1}$
in~\eqref{Gibbs3}
by $\frac \be {k+1}$ and 
$:\!u^k\!:$
by 
$\be :\!u^k\!:$ in \eqref{SNLW7}, 
the convergence claims in 
Theorem \ref{THM:1}
hold
for the entire families $\{\rhoo_L\}_{L \ge 1}$ 
and  $\{(u_{L}, \dt u_{L})\}_{L \ge 1}$, 
thus giving uniqueness of the invariant
hyperbolic $\Phi^4_2$-dynamics on the plane.
While we expect that the argument in \cite{BDW}
 to hold also for a higher power
(which then would imply uniqueness
of the invariant hyperbolic $\Phi^{k+1}_2$-dynamics on the plane), 
we do not pursue this issue here.
See also a very recent preprint \cite{DHYZ}
for uniqueness of the $\Phi^4_3$-measure on $\R^3$
in the high temperature regime.

\end{remark}

\begin{remark}\label{REM:inv1}\rm

Let $\rho_\infty$ be the limiting $\Phi^{k+1}$-measure
in Theorem \ref{THM:1}\,(i)
constructed as a limit of $\{\rho_L\}_{L \in \A}$.
Consider 
the parabolic $\Phi^{k+1}_2$-model \eqref{SHE1} on $\R^2$
with the Gibbsian initial data $\Law(X_0) = \rho_\infty$.
Then, it follows from 
%
%
the proof of Proposition \ref{PROP:uniq}
(see Remark \ref{REM:conv2})
that $\rho_\infty$
is invariant under 
the parabolic $\Phi^{k+1}_2$-dynamics on   $\R^2$.
While this fact is not difficult to prove, 
to the authors' knowledge, it is not explicitly written
and thus we decided to mention it here.
\end{remark}

\section{Preliminary}

\subsection{Notations}
By $A\les B$, we mean  $A\leq CB$ 
for some  constant $C> 0$.
We use   $A \sim  B$
to mean 
 $A\les B$ and $B\les A$.
We write $A\ll B$, if there is some small $c>0$
such that  $A\leq cB$.  
We may use subscripts to denote dependence on external parameters; for example,
 $A\les_{\dl} B$ means $A\le C(\dl) B$.
We use $a-$ 
(and $a+$) to denote $a- \eps$ (and $a+ \eps$, respectively)
for arbitrarily small $\eps > 0$.
If this notation appears in an estimate, 
then  an implicit constant 
is allowed to depend on $\eps> 0$ (and it usually diverges as $\eps \to 0$).
We also use the notation $p = \infty - $
to denote a sufficiently large number $p \gg 1$, depending on the context.
For conciseness of notation, 
we set
$a \vee b = \max(a, b)$.

Throughout this paper, we fix a rich enough probability 
space $(\O, \F, \mathbb{P})$, on which all the random objects
are defined.
The realization $\o\in\O$ is often omitted in the writing.
Given  a random variable $X$, we denote by $\Law(X)$
the law  of $X$.

In the remaining part of the paper,  we only work with 
real-valued functions\,/\,distributions.
Given a time-differentiable function, 
we set
$\vec u = (u, \dt u)$.

We 
define the Fourier transform of a function $f$ on $\R^d$
by setting
\[ \ft f(\eta) =  \F(f) (\eta)
=  \int_{\R^d} f(x) e^{- 2\pi i \eta \cdot x} dx\]

\noi
with the inverse Fourier transform given by 
$\F^{-1}(f) (\eta ) = \ft f(-\eta)$.

Let $s \in \R$ and $1 \leq p \leq \infty$.
We define the $L^2$-based Sobolev space $H^s(\R^d)$
by the norm:
\begin{align*}
\| f \|_{H^s} =  \|\jb{\nb}^s f\|_{L^2}  =  \| \jb{\eta}^s \ft f (\eta) \|_{L^2}, 
\end{align*}

\noi
where $\jb{\,\cdot\,} = (1 + |\cdot|^2)^\frac 12$.
We also define the $L^p$-based Sobolev space $W^{s, p}(\R^d)$
by the norm:
\begin{align}
\| f \|_{W^{s, p}} =    \|\jb{\nb}^s f\|_{L^p} = \big\| \F^{-1} [\jb{\eta}^s \ft f(\eta)] \big\|_{L^p}.
\label{sob1}
\end{align}

\noi
When $p = 2$, we have $H^s(\R^d) = W^{s, 2}(\R^d)$. 
We recall the following interpolation result
which follows from the Littlewood-Paley characterization of Sobolev norms via the square function
and H\"older's inequality;
let $s, s_1, s_2 \in \R$ and $p, p_1, p_2 \in (1,\infty)$
such that $s = \ta s_1 + (1-\ta) s_2$ and $\frac 1p = \frac \ta{p_1} + \frac{1-\ta}{p_2}$
for some $0< \ta < 1$.
Then, we have
\begin{equation}
\| u \|_{W^{s,  p}} \les \| u \|_{W^{s_1, p_1}}^\ta \| u \|_{W^{s_2, p_2}}^{1-\ta}.
\label{interp}
\end{equation}

Let $X(\R^d)$ be a function space on $\R^d$ such as
the Sobolev space $W^{s, p}(\R^d)$
defined above or the weighted Sobolev\,/\,Besov spaces defined below.
Given a (nice) subset $K \subset \R^d$, 
we define the localized version of a function space $X(\R^d)$
by the restriction norm:
\begin{align}
\| f\|_{X(K)} = \inf  \big\{ \| g \|_{X(\R^d)}: \, f \equiv g \text{ on } K\big\}.
\label{loc1}
\end{align}

In dealing with a space of space-time functions, 
we often use short-hand notations such as
$L^q_T H^s_x$
for $L^q([0, T]; H^s(\R^d))$.
Given a space $X(\R^d)$ of functions on $\R^d$, 
we set $L^\infty([0, R]; X(B_{R-t}))$
by 
\begin{align}
L^\infty([0, R]; X(B_{R-t})) = 
\big\{ u \in \D'(\mathring \bC_R): 
t \in [0, R] \mapsto \|u(t)\|_{X(B_{R-t})}\text{ is in $L^\infty$}
\big\}, 
\label{space1}
\end{align}

\noi
where 
$\mathring \bC_R$ denotes the interior of the cone
$ \bC_R$ and 
$X(B_{R-t})$ is defined by the restriction norm as in \eqref{loc1}.
Here, $B_R \subset \R^d$
 denotes the (closed) ball of radius $R$ centered at the origin
and $\bC_R$ denotes the cone as in \eqref{cone}
 (but in  $ \R^d \times \R_+$).

\smallskip

Let $\D(t)$ denote the linear damped wave propagator defined in 
\eqref{lin2}.
We then set
\begin{align}
\S(t) = \dt \D(t) + \D(t).
\label{lin3}
\end{align}

\noi
Namely, $u = \S(t) u_0$ satisfies
\begin{align*}
\begin{cases}
\dt^2 u + \dt u + (1-\Dl) = 0\\
(u, \dt u)|_{t= 0} = (u_0, 0).
\end{cases}
\end{align*}

\noi
Let  $p_t$  denote the standard heat kernel on $\R^2$
given by  $p_t(x) = \frac{1}{4\pi t} e^{- \frac{|x|^2}{4t}  } $.
Then, 
let $P(t)$ denote  the linear heat propagator given by 
\begin{align}
P(t)f  = e^{t(\Dl - 1)} f = e^{-t} (p_t*f).
\label{heat1}
\end{align}


We now introduce the Littlewood-Paley projector, 
which is adapted to weighted Sobolev and Besov spaces
defined in the next subsection.
Recall the following definition \cite[Definition~1]{MW17a}.\footnote{Definition 2.1 in the arXiv version.}
Given $\ta_0 \ge 1$, 
 the Gevrey class 
 $\GG^{\ta_0} \subset C^\infty(\R^d; \R)$ 
 of order $\ta_0$
consists of functions $f$ satisfying
the following;
given any compact set $K \subset \R^d$, 
there exists $C_K > 0$ such that 
$\sup_{x \in K}|\dd^\al f(x)| \le (C_K)^{|\al|+1} (\al!)^{\ta_0}$
for any multi-index $\al$.
We use   $\GG^{\ta_0}_c$ to denote
the set of compactly supported functions in  $\GG^{\ta_0}$.

Let  $\Z_{\ge 0} := \N \cup\{0\}$.
Let us now introduce the 
Littlewood-Paley projector
$\Q_k$, $k \in \Z_{\ge 0}$,  defined by Gevrey class multipliers.
As seen in \cite{MW17a}, such a choice is suitable
for weighted spaces with a stretched exponential weight;
see \eqref{eta2}.
As in Section 3 of \cite{MW17a}, let 
\begin{align}
\chi_0, \chi_1 \in \GG^{\ta_0}_c(\R^d; [0, 1])
\label{LP1}
\end{align}
with 
\begin{align*}
\supp \chi_0 \subset \big\{|\xi|\le \tfrac{4}{3}\big\}
\qquad \text{and}\qquad 
\supp \chi_1 \subset \big\{\tfrac 34 \le |\xi|\le \tfrac{8}{3}\big\}
\end{align*}

\noi
such that 
\begin{align*}
\sum_{k = 0}^\infty \chi_k(\xi) \equiv 1
\end{align*}

\noi
on $\R^d$, where $\chi_k(\xi) = \chi_1(2^{1-k} \xi)$ for $k \ge 2$.
We then define the Littlewood-Paley projector $\Q_k$ by 
\begin{align}
\Q_k(f) = \F^{-1} (\chi_k \ft f) = \eta_k * f, 
\label{LP2}
\end{align}

\noi
where $\eta_k$ is defined by  
\begin{align}
\eta_k = \F^{-1}(\chi_k)
\label{eta1}
\end{align}

\noi
for $k \in \Z_{\ge 0}$.
Namely, we have $\eta_k(x) = 2^{dk} \eta(2^kx)$ for $k \in \N$, 
where 
$\eta(x)  = 2^{-d}\eta_1(2^{-1} x)$.
We also recall 
Proposition 1 in \cite{MW17a}\footnote{Proposition 2.2 in the arXiv version.}
that 
\begin{align}
|\eta(x)| \les e^{-c|x|^\frac{1}{\ta_0}}.
\label{eta2}
\end{align}

\subsection{Weighted Sobolev and Besov spaces}

In this subsection, we introduce weighted Sobolev and Besov spaces
and discuss their basic properties.
We first recall the definition of the stretched exponential weight introduced in \cite{MW17a}.
Let $\ta_0$ be as in \eqref{LP1}, appearing in the definition 
of the Littlewood-Paley
projector $\Q_k$ in \eqref{LP2}.
Given $0 < \dl < \frac 1{\ta_0}$ ($\le 1$)
and $\mu > 0$, 
we define the weight $w_\mu(x)$
by setting
\begin{align}
w_\mu(x) = e^{-\mu \jb{x}^\dl}.
\label{weight1}
\end{align}

\noi
It is easy to check that 
the weight $w_\mu$ is $w_{-\mu}$-moderate
in the sense:
\begin{align}
w_\mu(x+y) \le w_{-\mu}(x) w_\mu(y)
\label{weight2}
\end{align}

\noi
for any $x, y \in \R^2$;
see \cite[(2.6)]{MW17a}.
In particular, 
it follows from \eqref{weight2} that 
\begin{align}
w_\mu(x) \les w_\mu(n)
\label{weight2a}
\end{align}

\noi
for any  $x \in Q_n := n + \big[-\frac 12, \frac 12\big)^2$.
Let $0 < \al < 1$.
Then, recall from \cite[(3.25)]{MW17a}\footnote{(3.24) in the arXiv version.}
that we have 
\begin{align}
|\Q_k w_\mu(x)| \les 2^{-\al k}w_\mu(x)
\label{weight3}
\end{align}

\noi
for any $x \in \R^2$ and $k \in \Z_{\ge 0}$.
In the remaining part of the paper, we fix $0 < \dl < \frac{1}{\ta_0}$ in \eqref{weight1}
and we often drop the dependence on $\dl$ in various estimates.

Let $1 \le p <  \infty$ and $\mu > 0$.
Then,  the weighted Lebesgue space
 $ L^p_\mu(\R^d)$  is defined by the norm:
\begin{align}
\|f \|_{L^p_\mu}   =\bigg(\int_{\R^d} |f(x)|^p w_\mu (x) dx\bigg)^\frac 1p.
\label{Lpmu}
\end{align}

\noi
While it is possible to introduce a weighted Sobolev space
by the $L^p_\mu$-norm of $\jb{\nb}^s f$ (see~\eqref{WSP2} below), 
we use a slightly different definition by first introducing spatial localization.

Let $\phi:\R \to [0,1]$ be a smooth even 
function such that $\phi$ is non-increasing on $\R_+$, 
$\phi =1$ on $\big[-\frac 54, \frac54\big]$,  and  $\phi=0$  on  $(-\infty, -\frac 85\big]\cup \big[\frac85, \infty)$.
Given $j \in \Z_{\ge 0}$, 
we define $\phi_j$ by 
\begin{align}
\phi_0(x) = \phi(|x|) 
\qquad
{\rm and}
\qquad
\phi_j(x) = \phi\big(\tfrac{|x|}{2^{j}}\big) -  \phi\big(\tfrac{|x| }{2^{j-1}}\big), \ \ j \in \N.
\label{phi}
\end{align}

\noi
Then, it is easy to check that
\begin{align}
\sum_{j=0}^\infty \phi_j(x) =1 \,\,\,\, 
\text{for any $x\in\R^d$}, 
\label{phi1a}
\end{align}

\noi
and 
\begin{align}
\supp \phi_0 \subset \big\{ |x| \le  \tfrac85\}
\quad \text{and}\quad 
\supp \phi_j \subset \big\{ \tfrac 58 \cdot 2^j \le |x| \le  \tfrac85 \cdot 2^j \}, \  \ j \in \N.
\label{phi2}
\end{align}

\noi
By convention, we set $\phi_{-1} = 0$.
Given $s \ge 0$ and $1 \le p \le \infty$, 
it follows from \eqref{phi} that 
\begin{align}
\| \phi_j\|_{W^{s, p}(\R^d)} \sim 2^{\frac d p j}.
\label{phi2a}
\end{align}

\noi
We also define slightly fattened cutoff functions $\wt \phi_j$, $j \in \Z_{\ge 0}$,  by setting 
\begin{align}
\wt \phi_j =   \phi_{j-1} + \phi_j + \phi_{j+1}, \ \ j \in \N.
\label{phi3a}
\end{align}

\noi
Then, we have 
\begin{align}
\wt \phi_j \phi_j = \phi_j.
\label{ID1}
\end{align}

\noi
for any $j\in \Z_{\ge 0}$.

Given $1 \le p < \infty$ and $\mu > 0$, 
we  define the weighted Sobolev space
 $W^{s,p}_\mu(\R^d)$ by the norm:
\begin{align}
\|f \|_{W^{s,p}_\mu}  = 
\bigg(\sum_{j=0 }^\infty w_\mu(2^{j})\| \phi_j f\|_{W^{s,p}}^p\bigg)^\frac 1p , 
\label{WSP}
\end{align}

\noi
where  $W^{s,p} = W^{s,p}(\R^d)$ denotes the usual  $L^p$-based Sobolev space on $\R^d$
defined in \eqref{sob1}.
When $p = 2$, we set 
\[H^s_\mu(\R^d) = W^{s, 2}_\mu(\R^d).\]

\noi
The following embedding follows from the definition \eqref{WSP}
with Lemma \ref{LEM:bilin1} and \eqref{phi2a}:
\[W^{s, p_1}_{\mu_1} (\R^d) \subset W^{s, p_2}_{\mu_2}(\R^d) 
\quad   \text{ for $1 \le p_2 \le p_1 < \infty$
and $\mu_2 \ge \mu_1 > 0$.}
\]

\noi
From \eqref{WSP} and \eqref{Lpmu} with \eqref{weight1} and \eqref{phi2}, 
there exists $a_1 > 0$ such that 
\begin{align}
\begin{split}
\|f \|_{W^{s,p}_\mu}
& \le 
\sum_{j=  0}^\infty w_{\frac 1p \mu}(2^j) 
\| \phi_j f \|_{W^{s,p}  }
 \le 
\sum_{j=  0}^\infty e^{-  \frac{\mu}{p}  2^{j\dl}    }   
\| \phi_j f \|_{W^{s,p}  }, \\
\| \phi_j f \|_{W^{s,p}}
& \le e^\frac \mu p
 e^{  \frac{\mu}{p}  2^{j\dl}   }   
\| f \|_{W^{s,p}_\mu } ,\\
\| \phi_j f \|_{L^p}
& \le 
 e^{  \frac{\mu}{p} a_1  2^{j\dl}   }   
\| f \|_{L^p_\mu } 
\end{split}
\label{Lpm1}
\end{align}

\noi
for any $j \in \Z_{\ge 0}$.
Similar bounds hold when we replace $\phi_j$ by $\wt \phi_j$.

Let $1 \le q \le p < \infty$.
By applying Sobolev's inequality (with $\frac sd \ge \frac 1q - \frac 1p$)
and the embedding $\l^q(\Z_{\ge 0}) \subset \l^p(\Z_{\ge 0})$,  we have 
\begin{align}
\begin{split}
\|f \|_{W^{0,p}_\mu}  
& = 
\big\| w_{\frac \mu p } (2^{j})\| \phi_j f\|_{L^{p}}\big\|_{\l^p_j(\Z_{\ge 0})}\\
& \les
\big\| w_{\frac \mu p } (2^{j})\| \phi_j f\|_{W^{s, q}}\big\|_{\l^q_j(\Z_{\ge 0})}\\
& = \|f \|_{W^{s,q}_{\frac{q}{p}\mu}}  .
\end{split}
\label{sob2}
\end{align}

\begin{remark}\rm
We point out that when $s = 0$, 
the $W^{0, p}_\mu$-norm defined in  \eqref{WSP}
is not equivalent to the $L^p_\mu$-norm defined in \eqref{Lpmu}.
See also Remark \ref{REM:equiv} below.

\end{remark}

We also recall the definition of the weighted Besov space
$B^{s, \mu}_{p, q}(\R^d)$
from \cite[Section 3]{MW17a}.
Given $s \in \R$, 
$1 \le p, q \le \infty$, and $\mu \ge 0$, 
we define
the weighted Besov space
$B^{s, \mu}_{p, q}(\R^d)$
as the completion of $C^\infty_c(\R^d)$ under the norm:
\begin{align}
\| f\|_{B^{s, \mu}_{p, q}}
= \Big\|2^{sk} \|\Q_k f\|_{L^p_\mu(\R^2)}\Big\|_{\l^q_k(\Z_{\ge 0})}.
\label{besov1}
\end{align}

\noi
See also \cite{GH1}.
When $\mu = 0$, 
the weighted Besov space
$B^{s, \mu}_{p, q}(\R^d)$
reduces to the usual
 Besov space
$B^{s}_{p, q}(\R^d)$.
We recall the following embeddings
for (unweighted) Sobolev and Besov spaces:
\begin{align}
\| f \|_{B^s_{p, \infty}}
\les 
\|f\|_{W^{s, p}}
\les 
\| f \|_{B^s_{p, 1}}
\les 
\| f \|_{B^{s+\eps}_{p, \infty}}
\label{besov2}
\end{align}

\noi
for any $\eps > 0$.

%
%
%

We first establish the following compact embedding for weighted Besov spaces, 
which plays a crucial role in proving Theorem \ref{THM:1}\,(i).

\begin{lemma}\label{LEM:cpt}
Let $1 \le p < \infty$.
Then, given any $s > s'$ and $\mu < \mu'$, 
the embedding
\[
W^{s, p}_{\mu} (\R^d) \hookrightarrow W^{s', p}_{\mu'}(\R^d)
\]

\noi
is compact.

\end{lemma}

\begin{proof} 

Let $\{ f_n\}_{n\in\N}$ be a bounded sequence in $W^{s,p}_\mu(\R^d)$.
Then it follows from \eqref{WSP} and~\eqref{phi2} that for any $j\in\N$,
the sequence $\{ \phi_j f_n\}_{ n\in\N }$
is a bounded sequence in $W^{s,p}(\R^2)$
with a bounded support:
$ \big\{  |x| \le  \tfrac85 \cdot 2^j \}$, 
$j \in \Z_{\ge0}$.

In the following, we implement a diagonal argument.
Let $j = 0$.
Then,  by Rellich's lemma (with $s' < s$), 
there exists a subsequence 
 $\{ \phi_0 f_{n_k^{(0)}}\}_{ k\in\N }$ 
 which is convergent in 
$W^{s',p}(\R^d)$.
Then,
for $j \in \N$, 
by  Rellich's lemma, 
we can choose a subsequence
$\{n_k^{(j)}\}_{k \in \N}$
of $\{n_k^{(j-1)}\}_{k \in \N}$
such that 
 $\{ \phi_j f_{n_k^{(j)}}\}_{ k\in\N }$ is convergent in 
$W^{s',p}(\R^d)$.
Now, consider the sequence 
 $\{ \phi_j f_{n_k^{(k)}}\}_{ k\in\N }$.
Then, it is clear from the construction that, for each  $j \in \Z_{\ge 0}$, 
$ \phi_j f_{n_k^{(k)}}$
 converges to some $F_j $ in 
$W^{s',p}(\R^d)$ as $k \to \infty$.

By setting 
\[
F = \sum_{j =  0}^\infty F_j,
\]

\noi
 we claim that  $F\in W^{s', p}_{\mu'}(\R^d)$
 for any $\mu' > \mu$
and
that 
$f_{n_k^{(k)}}$ converges to $F$  in $W_{\mu'}^{s', p}(\R^d)$
for $\mu'> 2^\dl \mu$.

First,  note that
we have 
\begin{align}
\|  \phi_j f \|_{W^{s', p}}
\les \|  f \|_{W^{s', p}}, 
\label{cpt1}
\end{align}

\noi
uniformly in $j \in \Z_{\ge 0}$.
When $s' = 0$, \eqref{cpt1}  follows from H\"older's inequality.
When $s' < 0$, it follows from Lemma \ref{LEM:bilin1}\,(ii)
and \eqref{phi2a}
that 
\begin{align*}
\|  \phi_j f \|_{W^{s', p}}
\les 
\| \phi_j \|_{W^{- s', \infty}}
\|   f \|_{W^{s', p}}
\les \|  f \|_{W^{s', p}}.
\end{align*}

\noi
When $s' > 0$, 
the fractional Leibniz rule (see Lemma \ref{LEM:bilin1}\,(i) below), 
\eqref{phi2a}, 
and Sobolev's inequality yield
\begin{align*}
\|  \phi_j f \|_{W^{s', p}}
\les 
\| \phi_j \|_{W^{s', \infty}}
\|   f \|_{W^{s', p}}
\les \|   f \|_{W^{s', p}}.
\end{align*}

\noi
This proves the bound \eqref{cpt1}.
Then, 
from \eqref{Lpm1}, 
 \eqref{phi2}, 
\eqref{cpt1}, 
$w_{\mu'}(2^j)  = w_{\mu}(2^j) w_{\mu'-\mu}(2^j)$, 
and the uniform bound on the $W^{s, p}_\mu$-norm
of $f_{n_k^{(k)}}$, 
 we have
\begin{align*}
\|F\|_{W_{\mu'}^{s', p}}
& \le \sum_{j = 0}^\infty \| F_j\|_{W_{\mu'}^{s', p}}
 \le \sum_{j = 0}^\infty\sum_{\l = 0}^\infty 
 w_{\frac 1p \mu'}(2^\l) 
 \| \phi_\l F_j \|_{W^{s', p}}\\
&  \le \sum_{j = 0}^\infty\sum_{\l = 0}^\infty 
 w_{\frac 1p \mu'}(2^\l) 
 \lim_{k \to \infty}
 \| \phi_\l \phi_j f_{n_k^{(k)}} \|_{W^{s', p}}\\
&  =  \sum_{\l = 0}^\infty\sum_{i  = -1}^1
 w_{\frac 1p \mu'}(2^\l) 
 \lim_{k \to \infty}
 \| \phi_\l \phi_{\l + i} f_{n_k^{(k)}} \|_{W^{s', p}}\\
&  \les 
\sum_{\l = 0}^\infty 
 w_{\frac 1p \mu'}(2^\l) 
 \lim_{k \to \infty}
 \|  \phi_\l f_{n_k^{(k)}} \|_{W^{s, p}}\\
&  \le 
\sum_{\l = 0}^\infty
 w_{\frac 1p (\mu'-\mu)}(2^\l) 
 \cdot 
\sup_{k \in \N} \Big(  w_{\frac 1p \mu}(2^\l)  \| \phi_\l f_{n_k^{(k)}}\|_{W^{s, p}}\Big)\\
& < \infty, 
\end{align*}

\noi
provided that $s' < s$ and $\mu' > \mu$.

Next, we show   convergence.  First, note that, as a limit of 
$\phi_j f_{n_k^{(k)}}$, 
we have $F_j(x) = 0$  whenever $\phi_j(x) = 0$.
Thus,  in view of \eqref{phi2}, we can write $F_j = \phi_j G_j$
with a well-defined function $G_j = \ind_{\phi_j \ne 0 } \cdot F_j / \phi_j$
and thus,  we have
\begin{align}
 f_{n_k^{(k)}} - F = \sum_{j = 0}^\infty \phi_j ( f_{n_k^{(k)}} - G_j) .
\label{cpt2}
 \end{align}

\noi
We set $F_{-1} = G_{-1} = 0$.

From \eqref{WSP}, we have
\begin{align}
\begin{split}
\| f_{n_k^{(k)}} - F  \|_{W^{s',p}_{\mu'}}^p
& = \sum_{j = 0}^\infty 
w_{\mu'}(2^j) 
 \| \phi_j (
f_{n_k^{(k)}}  -F) \|_{W^{s', p}}^p.
\end{split}
\label{cpt3}
\end{align}

\noi
Then, from \eqref{cpt2} with \eqref{phi2}
followed by \eqref{cpt1} 
we have 
\begin{align}
\begin{split}
w_{\mu'}(2^j)  \| \phi_j (
f_{n_k^{(k)}}  -F) \|_{W^{s', p}}^p
& = w_{\mu'}(2^j) 
\bigg \| \phi_j
\sum_{i= - 1}^1 \phi_{j+i}  (f_{n_k^{(k)}}  - G_{j+i}) \bigg \|_{W^{s', p}}^p\\
& 
\les   w_{\mu'}(2^j)
\bigg \| 
\sum_{i= - 1}^1 \phi_{j+i}  (f_{n_k^{(k)}}  - G_{j+i}) \bigg \|_{W^{s', p}}^p\\
& 
\les 
\sum_{\l = j- 1}^{j+1}
 w_{\mu'}(2^{\l-1})
 \| \phi_{\l}  (f_{n_k^{(k)}}  - G_{\l})  \|_{W^{s', p}}^p\\
 & 
= \sum_{\l = j- 1}^{j+1}
 w_{\mu'}(2^{\l-1})
 \| \phi_{\l}  f_{n_k^{(k)}}  - F_{\l}  \|_{W^{s', p}}^p.
 \end{split}
 \label{cpt4}
\end{align}

\noi
For each $j \in \Z_{\ge0}$, 
the right-hand side of \eqref{cpt4} tends to $0$ as $k \to \infty$.
Namely, for each $j \in \Z_{\ge 0}$, 
the summand on the right-hand side of  \eqref{cpt3} tends to $0$ as $k \to \infty$. 
On the other hand, in view of the uniform bound
on the $W^{s, p}_\mu$-norm of $f_{n_k^{(k)}}$
and the fact that $F\in W^{s', p}_{\mu'}(\R^d)$, we have
\begin{align*}
w_{\mu'}(2^j) 
&  \| \phi_j (
 f_{n_k^{(k)}}  -F) \|_{W^{s', p}}^p\\
& \les w_{\mu'}(2^j) 
 \| \phi_j f_{n_k^{(k)}} \|_{W^{s', p}}^p
+ w_{\mu'}(2^j) 
 \| \phi_j   F \|_{W^{s', p}}^p, 
\end{align*}

\noi
which is summable in $j$.
Therefore, by the dominated convergence theorem
applied to \eqref{cpt3}, 
we conclude that 
$f_{n_k^{(k)}}$ converges to $ F$ in $W^{s',p}_{\mu'}(\R^d)$
as $k \to \infty$.
This concludes the proof of Lemma \ref{LEM:cpt}.
\end{proof}

There seems to be no embedding relation 
analogous to \eqref{besov2}
for weighted Sobolev spaces $W^{s, p}_\mu(\R^d)$ in \eqref{WSP}
and weighted Besov spaces.
With a slight loss in $s$ and $\mu$, however, 
we have the following embeddings
for weighted Sobolev spaces and weighted Besov spaces.
We present the proof in Appendix \ref{SEC:A}.

\begin{lemma}\label{LEM:equiv}

Let $ s \in \R$,  $1 \le p < \infty$, and $\mu > 0$.
Then, there exist $c_1, c_2  > 0$ such that 
\begin{align*}
\| f \|_{W^{s, p}_\mu} \les \|f\|_{B^{s', \mu'}_{p, 1}}
\end{align*}

\noi
for any $s' > s$ and $0 < \mu' < c_1 \mu$, 
and
\begin{align*}
\|f\|_{B^{s, \mu}_{p, \infty}} \les \| f \|_{W^{s, p}_{\mu'}}
\end{align*}

\noi
\noi
for any $0 < \mu' < c_2 \mu$.

\end{lemma}

\begin{remark}\label{REM:equiv} \rm

Another way to define a weighted Sobolev space 
would be 
to define a space~$\W^{s, p}_\mu(\R^2)$ via the norm:
\begin{align}
\| f \|_{\W^{s, p}_\mu} = \| \jb{\nb}^s f \|_{L^p_\mu}.
\label{WSP2}
\end{align}

\noi
When $s = 0$, we have $\W^{0, p}_\mu(\R^2) = L^p_\mu(\R^2)$.
Furthermore, this weighted Sobolev space~$\W^{s, p}_\mu(\R^2)$ is compatible with the weighted Besov space
defined in \eqref{besov1} in the following sense:
\begin{align}
\| f \|_{B^{s, \mu}_{p, \infty}}
\les 
\|f\|_{\W^{s, p}_\mu}
\les 
\| f \|_{B^{s, \mu}_{p, 1}}, 
\label{WSP3}
\end{align}

\noi
where the first inequality follows from Young's inequality
with \cite[Lemma 1]{MW17a}\footnote{Lemma 2.6 in the arXiv version.}
while the second inequality follows from Minkowski's inequality.
It follows from Lemma \ref{LEM:equiv}
 and \eqref{WSP3} that 
\begin{align*}
 \|f\|_{\W^{s - \eps, p}_{\mu_1}} \les \| f \|_{W^{s, p}_\mu} \les \|f\|_{\W^{s + \eps, p}_{\mu_2}}.
\end{align*}

\noi
for any $\eps > 0$ and $ \mu_1 \gg  \mu \gg \mu_2> 0$.

In this paper, 
we use the weighted Sobolev space $W^{s, p}_\mu(\R^2)$ defined in~\eqref{WSP}
since it is more convenient to work with
 $W^{s, p}_\mu(\R^2)$ in 
 establishing coming down from infinity
in a weighted Sobolev spaces of positive regularities.

See Remark \ref{REM:wave} for a discussion
on the bounded property of the linear damped wave propagator on $\W^{s, 2}_\mu(\R^d)$.

\end{remark}

\subsection{Linear estimates on weighted Sobolev spaces}

In this subsection, we establish 
linear estimates
for the linear heat propagator and the linear damped wave propagator
on weighted Sobolev spaces.

\begin{lemma}\label{LEM:heat}
There exist small $ C_0,   c  > 0 $ 
such that
 for 
any $s \ge 0$, $1 \le p < \infty$, 
$\mu > 0$, 
and $0 < \mu' <  C_0 \mu$, we have
\begin{align}
\| P(t)  f \|_{W^{s, p}_{\mu}(\R^d)}
\les t^{-\frac{s}{2}} e^{-c t} \| f \|_{L^p_{\mu'}(\R^d)}
\label{heat2}
\end{align}

\noi
for any $t>0$, 
where  $P(t) = e^{t (\Dl-1)}$ is as in \eqref{heat1}.
\end{lemma}

\begin{proof}

We first introduce a cutoff function $\phi_m^t$ in terms of the self-similar variable $\frac{x}{t^\frac 12}$
for the (homogeneous) heat equation $\dt u - \Dl u = 0$
by setting
\begin{align}
\phi_m^t(x) := \phi_m\bigg( \frac{x}{ t^{\frac12}} \bigg)
\label{heat3}
\end{align}

\noi
for $m\in\Z_{\ge 0}$. 
Then, given $j \in \Z_{\ge 0}$, 
it follows
from \eqref{phi1a}, \eqref{phi2},  and \eqref{heat3} that 
\begin{align}
\begin{aligned}
\phi_j (P(t) f) &= e^{-t} \sum_{m, m' = 0}^\infty \phi_j \big[ (\phi_m^t  p_t)\ast ( \phi_{m'}  f) \big] \\
&= e^{-t} \sum_{(m, m') \in\Ld_{t, j}} 
\phi_j \big[ (\phi_m^t   p_t)\ast (  \phi_{m'}  f) \big], 
\end{aligned}
\label{heat4}
\end{align}

\noi
where $ \Ld_{t, j} \subset (\Z_{\ge 0})^2$  is given by 
\begin{align}
\begin{split}
\Ld_{t, j} = \Big\{(m, m') \in (\Z_{\ge 0})^2:\
& \tfrac 58\cdot 2^{m'} 
\cdot \ind_{\{ m'\geq 1 \}}
\le
    \tfrac 85 \cdot 2^j + \tfrac 85 \cdot
    2^{m} t^{\frac12} 
    \\
& \text{and }\,  \tfrac 58\cdot 2^{m}t^\frac 12  
\cdot  \ind_{\{ m\geq 1 \}}
\le
    \tfrac 85 \cdot 2^j
+  \tfrac 85 \cdot
    2^{m'} 
\Big\}.
\end{split}
\label{heat5}
\end{align}

In view of \eqref{heat3}, 
a direct computation shows that 
\begin{align*}
\|\phi^t_m p_t \|_{L^{1}} & \les   \exp ( -c_1 4^m ),\\
\|\phi^t_m p_t \|_{W^{1,1}} & \les t^{-\frac12}  \exp ( -c_1 4^m )
\end{align*}

\noi
for some $c_1 > 0$.
Hence, by the interpolation \eqref{interp}, we have
\begin{align} 
\|\phi^t_m p_t \|_{W^{s,1}} 
\les t^{-\frac{s}{2} } \exp ( -c_1 4^m )
\label{heat6}
\end{align}

\noi
for any $t > 0$ and $0 < s < 1$.
We point out that \eqref{heat6}
also holds for any $s \ge 0$.
Then, 
from \eqref{heat1}, \eqref{Lpm1},  \eqref{heat4}, 
 the fractional Leibniz rule (Lemma \ref{LEM:bilin1}\,(i) below), 
 \eqref{phi}, 
 Young's inequality, 
and \eqref{heat6}, 
we have
\begin{align}
\begin{split}
\| P(t)  f \|_{W^{s, p}_{\mu}}
&\leq e^{-t} \sum_{j = 0 }^\infty e^{-\frac{\mu}{p}2^{j\dl}} 
\|  \phi_j   ( p_t\ast  f )\|_{W^{s, p}} \\
&\leq
e^{-t}
\sum_{j  = 0 }^\infty e^{-\frac{\mu}{p}  2^{j\dl}}
\sum_{(m, m')\in\Ld_{t, j} } \| \phi_j  \|_{W^{s,\infty}}
\|   ( \phi^t_m p_t) \ast  (\phi_{m'} f )\|_{W^{s, p}} \\
&\les
e^{-t} 
\sum_{j  = 0 }^\infty e^{-\frac{\mu}{p} 2^{j\dl}}
\sum_{(m, m')\in\Ld_{t, j}}
 \|   \phi^t_m p_t \|_{W^{s,1}}   \| \phi_{m'} f  \|_{L^{p}} \\
&\les   
t^{-\frac{s}{2}} e^{-t}
\sum_{j  = 0 }^\infty e^{-\frac{\mu}{p}  2^{j\dl}}
\sum_{(m, m')\in\Ld_{t, j}} 
  e^{-c_1 4^m}
  \|  f  \|_{L^{p}_{\mu'}} e^{\frac{\mu'}{p} a_1 2^{m'\dl}  }.
  \end{split}
\label{heat6a}
\end{align}

Hence, it remains  to estimate 
\begin{align*} 
A_t  := \sum_{j = 0}^\infty e^{-\frac{\mu}{p} 2^{j\dl} }
\sum_{(m, m')\in\Ld_{t, j}}  e^{-c_1 4^m}
  e^{\frac{\mu'}{p} a_1 2^{m'\dl}  }.
\end{align*}

\noi
First, we consider the case $0 < t \les 1$.
By summing over $m'$ with \eqref{heat5}, 
we have
\begin{align} 
A_t  \les  \sum_{j, m = 0}^\infty e^{-\frac{\mu}{p}2^{j\dl} }
e^{-c_1 4^m}
  e^{\frac{\mu'}{p} c_2 ( 2^{j \dl}  +  2^{m\dl} t^\frac{\dl}2 ) } \les 1, 
  \label{heat7}
\end{align}

\noi
provided that $ \mu > c_2  \mu'$.
Next, we consider the case $ t \gg  1$.
Note that $4^m \sim 2^{m\dl} t^\frac{\dl}{2}$
implies 
\begin{align}
 2^{m\dl}t^\frac{\dl}{2} \sim  t^{\frac{\dl}{2-\dl}}
\ll t
\label{heat7a}
\end{align}

\noi
 for $t \gg 1$ since $\dl < 1$.
By first summing over $m'$ as in~\eqref{heat7}
and then summing over $m$, 
\begin{align} 
A_t  \les  \sum_{j = 0}^\infty e^{
(-\frac{\mu}{p}  + \frac{\mu'}p c_2 ) 2^{j \dl} }
e^{c_3 t^{\frac{\dl}{2}\frac{2}{2-\dl}}}.
  \label{heat8}
\end{align}

\noi
Therefore, the desired bound \eqref{heat2} follows
from \eqref{heat6a}, \eqref{heat7}, and \eqref{heat8}
with \eqref{heat7a}.
\end{proof}

Next, we establish
estimates for 
the linear damped wave operator.

\begin{lemma}\label{LEM:wave}
Let $s \in \R$.
Then, there exists $C_0 > 0$ such that 
\begin{align}
\| \D(t) f\|_{H^s_\mu} \les e^{-\frac{t}{4}} \|  f\|_{H^{s-1}_{\mu'}} 
\label{Damp1}
\end{align}

\noi
and 
\begin{align}
\| \S(t) f\|_{H^s_\mu} \les  e^{-\frac{t}{4}}  \|  f\|_{H^{s}_{\mu'}}
\label{Damp2}
\end{align}

\noi
for $\mu > C_0 \mu' > 0$.
Here, $\D(t)$ and $\S(t)$ are as in \eqref{lin2} and \eqref{lin3}.

\end{lemma}

\begin{proof}
In view of  \eqref{phi1a}, write 
\begin{align}
\begin{aligned}
 \| \phi_j  \D(t) f \|_{H^s}
 =  \bigg\|  \sum_{\l = 0}^\infty \phi_j  \D(t)   (\phi_\l f) \bigg\|_{H^s}.
\end{aligned}
\label{Damp3}
\end{align}

\noi
In the following, we only consider $j, \l \ge 1$
but a similar argument holds for the case $j = 0$ or $\l = 0$.
Thanks to  the finite speed of propagation and \eqref{phi2}, 
we have
\begin{align*}
\supp \D(t) (\phi_\l f) \subset \big\{ \tfrac 58 \cdot 2^\l -t
 \le |x| \le  \tfrac85 \cdot 2^\l +t \}.
\end{align*}

\noi
As a result, 
 $\phi_j  \D(t)   (\phi_\l f) \neq 0$
only when
\begin{align}
\tfrac 58\cdot  2^{j} \leq \tfrac 85 \cdot 2^\l+ t 
\qquad
{\rm and}
\qquad
\tfrac 58 \cdot 2^{\l} - t \leq   \tfrac 85 \cdot 2^j.
\label{range1}
\end{align}

%

From the triangle inequality, we have  $\jb{\xi}^s \les \jb{\xi_1}^{|s|} \jb{\xi_2}^s $
for $\xi = \xi_1 + \xi_2$.
Then, from  Young's  inequality on the Fourier side, we have
\begin{align}
\| f g \| _{H^s}
\les \| f\|_{\F L ^{|s|,1}} \|g\|_{H^s}, 
\label{Damp5}
\end{align}

\noi
where the Fourier-Lebesgue norm is defined by 
\begin{align}
 \| f\|_{\F L ^{s,1}} = \| \jb{\xi}^s \ft f(\xi) \|_{L^1_\xi}.
\label{Damp5a}
 \end{align}

\noi
Note that from  \eqref{phi}, we have 
\begin{align}
\| \phi_j \|_{\F L^{s,1}} \les_s 1,
\label{Damp5b}
\end{align} 
 uniformly in  $j \in \Z_{\ge 0}$.

Given $j \in \Z_{\ge 0}$ and $t \ge 0$, 
define $\G_{j, t}$ by 
\[ \G_{j, t} = \big\{\l \in \Z_{\ge 0}: 
\l \text{ satisfies \eqref{range1}}\big\}.\]
Then, we deduce from 
\eqref{Damp3}, 
\eqref{Damp5}
\eqref{lin2}, and 
\eqref{Lpm1}, we have
\begin{align}
\begin{aligned}
\| \phi_j \D(t) f \|_{H^s}
&\les   
e^{-\frac{t}{2}} 
\sum_{ \l \in \G_{j, t} }
\| \phi_\l f\|_{H^{s-1}} \\
&\les e^{-\frac{t}{2}} 
\bigg(\sum_{ \l \in \G_{j, t} } e^{\frac{\mu'}{2} 2^{\l \dl}}\bigg)
\| f\|_{H^{s-1}_{\mu'}}\\
&\les e^{-\frac{t}{2}} 
 e^{c_0\mu'( 2^{j\dl} + t^\dl)}
\| f\|_{H^{s-1}_{\mu'}}, 
\end{aligned}
\label{Damp6}
\end{align}

\noi
where the implicit constant is independent of $j \in \Z_{\ge 0}$.
Then, from \eqref{Lpm1} and \eqref{Damp6}, we have 
\noi
\begin{align*}
\| \D(t) f \|_{H^s_\mu}
&\le \sum_{j = 0}^\infty e^{-\frac{\mu}{2} 2^{j \dl} } \| \phi_j \D(t) f \|_{H^s} \\
&\les e^{-\frac t2} 
\sum_{j = 0}^\infty e^{-\frac{\mu}{2} 2^{j \dl} }
 e^{c_0\mu'( 2^{j\dl} + t^\dl)}
   \| f\|_{H^{s-1}_{\mu'}} \\
&\les e^{-\frac t2} 
 e^{c_0t^\dl}
   \| f\|_{H^{s-1}_{\mu'}} 
  \les e^{-\frac t4} 
   \| f\|_{H^{s-1}_{\mu'}} ,    
\end{align*}

\noi
provided that $\mu > 2c_0 \mu'$.
This proves \eqref{Damp1}.
Recalling that
\[
\S(t) = \frac{1}{2}\D(t) + e^{-\frac{t}{2}}\cos\Big(t \sqrt{\tfrac34 -\Dl} \Big),
\]

\noi
the second  bound \eqref{Damp2} follows from an analogous
computation.
 \end{proof}

\begin{remark}\label{REM:wave}\rm

(i) 
Let $w_\mu$ be as in \eqref{weight1}.
Then, for any $0 < \dl \le 2$ and $\mu > 0$, 
 the Fourier transform of the weight $w_\mu$ is positive.
Indeed, it follows from 
Lemma 5 in \cite{EOR}
on completely monotonic functions that 
\[ w_\mu(x) = \int_0^\infty e^{-t \jb{x}^2} d\al_\mu (t)
= e^{-t} \int_0^\infty e^{-t |x|^2} d\al_\mu (t), \]

\noi
where $\al_\mu(t)$ is bounded and non-decreasing and the integral converges for any $x \in \R^2$.
Hence, as a superposition of the Gaussians $e^{-t |x|^2}$, 
we conclude that the Fourier transform of the weight $w_\mu$ is positive.

\smallskip

\noi
(ii)  Let $\W^{s, 2}_\mu(\R^d)$ be the weighted Sobolev space defined in \eqref{WSP2} with $p = 2$.
Then, an analogue of Lemma \ref{LEM:wave}
also holds for $\W^{s, 2}_\mu(\R^d)$.
Using the positivity of the Fourier transform of the weight $w_{\frac{\mu}{2}}$, we have
\begin{align*}
\|\D(t) f \|_{\W^{s, 2}_\mu}
& = \| \jb{\nb}^s \D(t) f\|_{L^2_\mu}
= \Big\| \big(\jb{\,\cdot\,}^s \ft {\D(t) f}\big)*\ft w_{\frac{\mu}{2}}(\xi) \Big\|_{L^2_\xi}\\
& \les  \Big\| \big(\jb{\,\cdot\,}^{s-1} |\ft { f}|\big)*\ft w_{\frac{\mu}{2}}(\xi) \Big\|_{L^2_\xi}
= \|f\|_{\W^{s-1, 2}_\mu}.
\end{align*}

\noi
A similar computation holds for $\S(t)$.
Note that, unlike Lemma \ref{LEM:wave},
there is no loss in the coefficient $\mu$ of the weight.

\end{remark}

\subsection{Product estimates}

We first state the product estimate on $\R^d$.

\begin{lemma}\label{LEM:bilin1}

\textup{(i)} 
 Let $s\ge 0$.
 Suppose that 
 $1<p_j,q_j \le \infty$ and $1 \le r < \infty$
 such that  $\frac1{p_j} + \frac1{q_j}= \frac1r$, $j = 1, 2$. 
%
%
%
%
%
%
%
%
%
 Then, we have  
\begin{align}  
\| fg \|_{W^{s, r}(\R^d)} 
\les \Big( \| f \|_{L^{p_1}(\R^d)} 
\|  g \|_{W^{s, q_1}(\R^d)} + \|  f \|_{W^{s, p_2}(\R^d)} 
\|  g \|_{L^{q_2}(\R^d)}\Big).
\label{B1}
\end{align}

\smallskip

\noi
\textup{(ii)} 
Let $s > 0$.
 Suppose that 

\smallskip
 \begin{itemize}
 \item[(ii.a)]
$ 1< p \le \infty$ and $1 < q,r < \infty$, 

\smallskip
 \item[(ii.b)]
$ 1< p = r \le \infty$ and $q = \infty$, or

\smallskip
 \item[(ii.c)]
$1<p = q'\le \infty$ 
and $r=1$ 

 \end{itemize}
 
\smallskip 
 
\noi
such that  
$\frac 1r \le \frac 1p + \frac 1 q \le \frac1 r + \frac sd$
and 
$q, r' \ge p'$.
Then, we have 
\begin{align}
\| fg \|_{W^{-s, r}(\R^d)  } \les  \| f\|_{W^{-s, p}(\R^d) } \|g \|_{W^{s,q}(\R^d)}.
\label{B2}
\end{align}

\end{lemma}

As for the fractional Leibniz rule \eqref{B1}
see  \cite{KP, CW}
for $p_j,q_j <  \infty$
and 
\cite[(2.1)]{MS}
(for the one-dimensional case which can be easily extended
to higher dimensions); see also \cite{GO}.
As for the second estimate \eqref{B2}, 
see \cite{GKO, BOZ}.
The estimates \eqref{B1} and \eqref{B2} indeed hold 
for wider ranges of indices;
see \cite{BOZ} for a further discussion.

\begin{remark}\label{REM:R}\rm
Note that, given a (nice) subset $K \subset \R^d$, 
the estimates \eqref{B1} and \eqref{B2} 
also hold on the localized version of Sobolev spaces $W^{s, p}(K)$
defined in \eqref{loc1}.
Given $f$ and $g$ on $K$, 
let ${\bf f}$ and ${\bf g}$
be extensions onto $\R^d$.
Then, from \eqref{loc1} and \eqref{B1}, we have 
\begin{align*}  
\| fg \|_{W^{s, r}(K)} 
& \le \| {\bf f}{\bf g} \|_{W^{s, r}(\R^d)} \\
& \les \Big( \| {\bf f} \|_{L^{p_1}(\R^d)} 
\|  {\bf g} \|_{W^{s, q_1}(\R^d)} + \|  {\bf f} \|_{W^{s, p_2}(\R^d)} 
\|  {\bf g} \|_{L^{q_2}(\R^d)}\Big).
\end{align*}

\noi
By taking infima over extensions 
${\bf f}$ and ${\bf g}$, we then obtain
\begin{align*}  
\| fg \|_{W^{s, r}(K)} 
& \
\les \Big( \| f \|_{L^{p_1}(K)} 
\|  g \|_{W^{s, q_1}(K)} + \|  f \|_{W^{s, p_2}(K)} 
\|  g \|_{L^{q_2}(K)}\Big).
\end{align*}  

A similar comment applies to all the preliminary estimates 
presented in this section.

\end{remark}

Next, we state a bi-parameter version
of the fractional Leibniz rule.

\begin{lemma}\label{LEM:bilin0}
Let $0 \le s \le 1$, $ 1 < p_1, p_2 \le \infty$, 
and $ 1 \le p < \infty$
with $\frac 1p = \frac 1{p_1}+ \frac 1 {p_2}$.
Then, we have
\begin{align}
\| \jb{\nb_x}^s
\jb{\nb_y}^s (fg) \|_{L^p_{x, y}(\R^{2d})}
\les \| \jb{\nb_x}^s
\jb{\nb_y}^s f \|_{L^{p_1}_{x, y}(\R^{2d})}
\| \jb{\nb_x}^s
\jb{\nb_y}^s g \|_{L^{p_2}_{x, y}(\R^{2d})}
\label{bi0}
\end{align}

\noi
for any functions  $f= f(x, y)$ and $g= g(x, y)$  on $\R^d_x \times \R^d_y$.

\end{lemma}

\begin{proof}
Let $0 \le s \le 1$ and $ 1 < p < \infty$.
Let $f= f(x, y)$ and $g= g(x, y)$ be functions on $\R^d_x \times \R^d_y$.
Then, by the Marcinkiewicz multiplier theorem
(Theorem 6.2.4 in \cite{Gra14}), we have 
\begin{align}
\begin{split}
\| \jb{\nb_x}^s
\jb{\nb_y}^s (fg) \|_{L^p_{x, y}(\R^{2d})}
& \les 
\| fg \|_{L^p_{x, y}(\R^{2d})}
+ \big\| |\nb_x|^s (fg) \big\|_{L^p_{x, y}(\R^{2d})}\\
& \quad + \big\||\nb_y|^s (fg) \big\|_{L^p_{x, y}(\R^{2d})}
+ \big\| |\nb_x|^s|\nb_y|^s (fg) \big\|_{L^p_{x, y}(\R^{2d})}.
\end{split}
\label{bi1}
\end{align}

\noi
Indeed, noting that 
\begin{align*}
\text{RHS of }\eqref{bi1}
\sim \big\|(1+ |\nb_x|^s)
(1+ |\nb_y|^s) (fg) \big\|_{L^p_{x, y}(\R^{2d})}, 
\end{align*}

\noi
 the bound \eqref{bi1} follows
since the multiplier $m(\eta_1, \eta_2)$, $(\eta_1, \eta_2) \in \R^d \times \R^d$,  defined by 
\begin{align*}
 m(\eta_1, \eta_2) = \frac{\jb{\eta_1}^s\jb{\eta_2}^s}
{(1 + |\eta_1|^s) (1 + |\eta_2|^s)}
\end{align*}
	
\noi
is a Marcinkiewicz multiplier.

By H\"older's inequality and the usual fractional Leibniz rule (Lemma \ref{LEM:bilin1}\,(i)), 
the first three terms on the right-hand side of \eqref{bi1} are bounded 
by the right-hand side of \eqref{bi0}.
As for the last term on the right-hand side of \eqref{bi1}, 
it follows from  the bi-parameter fractional Leibniz rule (\cite[(61) on p.\,295]{MPTT};
see also \cite[(3.3)]{MS})
that 
\begin{align*}
\big\| |\nb_x|^s|\nb_y|^s (fg) \big\|_{L^p_{x, y}(\R^{2d})}
& \les 
\big\| |\nb_x|^s|\nb_y|^s f \big\|_{L^{p_1}_{x, y}(\R^{2d})}
\|  g \|_{L^{p_2}_{x, y}(\R^{2d})}\\
& \quad + \| f \big\|_{L^{p_1}_{x, y}(\R^{2d})}
\big\| |\nb_x|^s|\nb_y|^s g \big\|_{L^{p_2}_{x, y}(\R^{2d})}\\
& \quad + \big\| |\nb_x|^sf \big\|_{L^{p_1}_{x, y}(\R^{2d})}
\big\| |\nb_y|^s g \big\|_{L^{p_2}_{x, y}(\R^{2d})}\\
& \quad + \big\| |\nb_y|^sf \big\|_{L^{p_1}_{x, y}(\R^{2d})}
\big\| |\nb_x|^s g \big\|_{L^{p_2}_{x, y}(\R^{2d})}, 
\end{align*}

\noi
which is once again bounded
by the right-hand side of \eqref{bi0}.
This concludes the proof of Lemma \ref{LEM:bilin0}.
\end{proof}

We  extend the product estimates in Lemma \ref{LEM:bilin1} to  weighted Sobolev spaces.

\begin{lemma} 
\label{LEM:bilin2}

\smallskip
\noi
 {\rm (i)} 
 Let $s \ge 0$ and $\mu>0$.
Suppose that 
 $1<p_j,q_j \le \infty$ and $1 \le r < \infty$
such that $\frac1{p_j} + \frac1{q_j}= \frac1r$, $j = 1, 2$. 
 Then, we have  
\begin{align}
\| fg \|_{W^{s,r}_\mu} 
\les  \| f\|_{W^{0, p_1}_\mu} \|g\|_{W^{s,q_1}_{\frac{\mu}{2^\dl} }}
+  \| f\|_{W^{s, p_2}_\mu} \|g\|_{W^{0, q_2}_{ \frac{\mu}{2^\dl} }}.
\label{prod1}
\end{align}

\smallskip
\noi
{\rm (ii)} 
Let $s> 0$ and $\mu>0$, 
and let $1 \le  p, q, r \le \infty$ be  as in Lemma \ref{LEM:bilin1}\,(ii).
Then, we have 
\begin{align}
\| fg \|_{W^{-s,r}_\mu} 
\les  \|  f\|_{W^{-s, p}_\mu} \|g\|_{W^{s,q}_{\frac{\mu}{2^\dl}}}.
\label{prod2}
\end{align}

\end{lemma}

\begin{proof} (i) 
From \eqref{WSP}
  \eqref{ID1},
 the fractional Leibniz rule (Lemma \ref{LEM:bilin1}\,(i)), 
and H\"older's inequality (in $j$) with $\frac 1r = \frac 1{p_1} + \frac 1{q_1}$, \eqref{phi3a}, 
and \eqref{weight1}, 
 we have
\begin{align*}
\| fg \|_{W^{s,r}_\mu}
&\les \bigg(\sum_{j = 0}^\infty w_\mu(2^j) 
   \|  \phi_j  f  g  \| _{W^{s,r}}^r
   \bigg)^\frac 1r 
= \bigg(\sum_{j = 0}^\infty w_\mu(2^j)
 \| \phi_j  f \cdot \wt \phi_j  g  \| _{W^{s,r}}^r
 \bigg)^\frac 1r  \\
& \les
\bigg(\sum_{j = 0}^\infty w_\mu(2^j) 
 \| \phi_j f \| _{L^{p_1}} ^r
\| \wt \phi_j g  \| _{W^{s,q_1}}^r  \bigg)^\frac 1r  \\
& \quad +
\bigg(\sum_{j = 0}^\infty w_\mu(2^j)
\| \phi_j f \|_{W^{s,p_2}}^r \| \wt \phi_j  g  \| _{L^{q_2}}^r
 \bigg)^\frac 1r \\
&  \les  \| f\|_{W^{0, p_1}_\mu} \|g\|_{W^{s,q_1}_{\frac{\mu}{2^\dl} }}
+  \| f\|_{W^{s, p_2}_\mu} \|g\|_{W^{0, q_2}_{ \frac{\mu}{2^\dl} }}, 
\end{align*}

\noi
which yields \eqref{prod1}.

\smallskip

\noi
(ii) By  Lemma \ref{LEM:bilin1}\,(ii), we have
\[
\| \phi_j f \cdot   \wt \phi_j g  \| _{W^{-s,r}}
\les  \| \phi_j f   \| _{W^{-s,p}}  
 \|\wt \phi_j  g    \| _{W^{s,q}}.  
\]

\noi
Then, \eqref{prod2} follows from applying
 H\"older's inequality as in Part\,(i).
\end{proof}

\subsection{Tools from stochastic analysis}

In the following, we review some basic facts on  
the Hermite polynomials 
and the Wiener chaos estimate. 
See,  for example,  \cite{Kuo, Nu}.

We define the $k$th Hermite polynomials 
$H_k(x; \s)$ with variance $\s$ via the following generating function:
\begin{equation*}
e^{tx - \frac{1}{2}\s t^2} = \sum_{k = 0}^\infty \frac{t^k}{k!} H_k(x;\s) 
 \end{equation*}
 
\noi
for $t, x \in \R$ and $\s > 0$.
When $\s = 1$, we set $H_k(x) = H_k(x; 1)$.
Then, we have
\begin{align}
 H_k(x ; \s) = \s^{\frac{k}{2}}H_k(\s^{-\frac{1}{2}}x ).
\label{H1a}
 \end{align}

\noi
It is well known that $\big\{ H_k /\sqrt{k!}\big\}_{k \in \Z_{\ge 0}}$
form an orthonormal basis of $L^2(\R; \frac{1}{\sqrt{2\pi}} e^{-x^2/2}dx)$.
The following identity:
\begin{align*}
 H_k(x+y) 
& = \sum_{\l = 0}^k
\begin{pmatrix}
k \\ \l
\end{pmatrix}
x^{k - \l} H_\l(y), 
\end{align*}

\noi
which, together with \eqref{H1a}, yields
\begin{align}
\begin{split}
H_k(x+y; \s )
&  = \s^\frac{k}{2}\sum_{\l = 0}^k
\begin{pmatrix}
k \\ \l
\end{pmatrix}
\s^{-\frac{k-\l}{2}} x^{k - \l} H_\l(\s^{-\frac{1}{2}} y)\\
&  = 
\sum_{\l = 0}^k
\begin{pmatrix}
k \\ \l
\end{pmatrix}
 x^{k - \l} H_\l(y; \s).
 \end{split}
\label{Herm3}
\end{align}

Let $(H, B, \mu)$ be an abstract Wiener space.
Namely, $\mu$ is a Gaussian measure on a separable Banach space $B$
with $H \subset B$ as its Cameron-Martin space.
Given  a complete orthonormal system $\{e_j \}_{ j \in \N} \subset B^*$ of $H^* = H$, 
we  define a polynomial chaos of order
$k$ to be an element of the form $\prod_{j = 1}^\infty H_{k_j}(\jb{x, e_j})$, 
where $x \in B$, $k_j \ne 0$ for only finitely many $j$'s, $k= \sum_{j = 1}^\infty k_j$, 
$H_{k_j}$ is the Hermite polynomial of degree $k_j$, 
and $\jb{\cdot, \cdot} = \vphantom{|}_B \jb{\cdot, \cdot}_{B^*}$ denotes the $B$--$B^*$ duality pairing.
We then 
denote the closure  of 
polynomial chaoses of order $k$ 
under $L^2(B, \mu)$ by $\mathcal{H}_k$.
The elements in $\H_k$ 
are called homogeneous Wiener chaoses of order $k$.
Then, we have the following Ito-Wiener decomposition:
\begin{equation*}
L^2(B, \mu) = \bigoplus_{k = 0}^\infty \mathcal{H}_k.
\end{equation*}

\noi
See Theorem 1.1.1 in \cite{Nu}.
See also \cite{J, Bog}.
We also set
\[ \H_{\leq k} = \bigoplus_{j = 0}^k \H_j\]

\noi
 for $k \in \N$.

We now state the Wiener chaos estimate, 
which is a consequence of Nelson's hypercontractivity \cite{Nelson2}.
See, for example,   \cite[Theorem I.22]{Simon}.
See also \cite[Proposition 2.4]{TTz}.

\begin{lemma}\label{LEM:hyp}
Let $k \in \Z_{\ge 0}$.
Then, we have 
 \begin{equation*} 
 \Big( \E \big[|X|^p \big]\Big)^{\frac 1p} \le (p-1)^\frac{k}{2} \Big( \E\big[|X|^2\big] \Big)^{\frac 12}
 \end{equation*}

\noi
for any random variable $X \in \mathcal{H}_k$ and  any $2 \leq p < \infty$.

\end{lemma}

Next, we recall the following orthogonal property of Wick powers;
see \cite[Lemma 1.1.1]{Nu}.

\begin{lemma}\label{LEM:W1}

Let $Y_1, Y_2$ be two real-valued, mean-zero, and jointly Gaussian random variables with variances
$\s_1 = \E[ Y_1^2] > 0$ and $\s_2 = \E[ Y_2^2] > 0$.
Then, for $k, m \in \N \cup\{0\}$, we have 
\begin{align*}
\E\big[ H_k(Y_1 ; \s_1) H_m(Y_2;  \s_2)\big]
=  \ind_{k=m} \cdot k!   \big(\E[ Y_1 Y_2] \big)^k.
\end{align*}

\end{lemma}

The following lemma allows us to compute 
the regularity of a stochastic term
by testing it against a test function.
See \cite[Proposition A.3.3]{Tolo1}
for the proof.

\begin{lemma}\label{LEM:Kol}

Let $X \in \H_{\le k}$ for some $k \in \N$.\footnote{This means that 
$\jb{X, \varphi} \in \H_{\le k}$
for any test function $\varphi$.}
Let $R \ge 1$.
Suppose 
that  there exist $\s \in \R$ and $2\le q < \infty$
such that 
\[
\E\Big[ \big|  \jb{X, \varphi} \big|^2 \Big] \leq A^2 \| \varphi\|_{W^{\s, q'}}^2
\]

\noi
for any test function $\varphi\in C^\infty_c$ supported on a ball $B$ of radius 1
with $B \subset B_{2R}$, 
where $\frac{1}{q'}+ \frac{1}{q} =1$. 
Then, given any small $\eps >0$ and  any finite $p\geq \frac{4}{\eps}$, we have 
\[
\E\Big[ \| X\|_{W^{-\frac{2}{q} -\sigma - \eps,\infty   } (B_R)  }^p \Big] \les_{\s, q, R }  A^p.
\]
\end{lemma}

\medskip

\noi
$\bul$ {\bf Wasserstein-1 metric and weak convergence. }

Let $(X, d)$ be  a Polish space (separable complete metric  space).
Then, the Wasserstein-1 metric 
for two probability measures $\mu,  \nu$
on $X$ is defined by 
\begin{align}
d_{\rm Wass}(\mu, \nu) 
= \inf_{\plan\in\Pi(\mu, \nu)} 
\int_{X \times X} d (x,y)  d\plan (x, y).
\label{Wass}
\end{align}

\noi
Here,  $\Pi(\mu, \nu)$ is the set of probability measures $\plan$
on $X \times X$
whose first and second marginals are given by $\mu$ and $\nu$, 
namely, 
\begin{align}
 \int_{y \in X}  d\plan (x, y) = d\mu(x) \qquad \text{and}
\qquad 
 \int_{x \in X}  d\plan (x, y) = d\nu(y).
\label{Wass2}
\end{align}

\noi
The Wasserstein-1 metric is also known as  the 
Kantorovich-Rubinstein distance;
see  the 
Kantorovich-Rubinstein theorem (\cite[Theorem 1.14]{Villani})
which provides the dual characterization of \eqref{Wass}.
We point out that the infimum in \eqref{Wass}
is indeed attained; see \cite[Theorem~1.3]{Villani}.

In general, convergence in the Wasserstein-1 metric 
is stronger than weak convergence.
We recall the following characterization of convergence
in the Wasserstein-1 metric; 
see \cite[Theorem 7.12]{Villani}.

\begin{lemma}\label{LEM:wass}
Let $(X, d)$ be a Polish space. 
Given a sequence of probability measures 
$\{\mu_n\}_{n \in \N}$ on $X$ and a  probability measure $\mu$ on $X$, 
 the sequence $\{\mu_n\}_{n \in \N}$
 converges to $\mu$ in the Wasserstein-1 metric as $n \to \infty$
 if and only if

\begin{itemize}
\item[(i)]    $\mu_n$ converges weakly to $\mu$ as $n \to \infty$, and 

\smallskip

\item[(ii)]
for some 
\textup{(}and thus for any\textup{)} $x_0\in X$, we have 
\begin{align}
\lim_{R\to \infty} \limsup_{n\to \infty}
 \int_{d(x, x_0) \ge R} d(x, x_0) d\mu_n(x) = 0.
\label{C22}
\end{align}


\end{itemize}

\end{lemma}

\begin{remark}\label{REM:wass}\rm

When the metric $d$ on $X$ is bounded, 
the condition \eqref{C22} trivially holds, 
and thus 
convergence in the Wasserstein-1 metric coincides
with weak convergence.
See also Remark 7.13\,(iii) in \cite{Villani}.

\end{remark}

Lastly, 
we recall the Prokhorov theorem
and the Skorokhod representation theorem.
We first recall the definition of tightness.

\begin{definition}\label{DEF:tight}\rm
Let $\J$ be a nonempty index set.
A family 
 $\{ \mu_n \}_{n \in \J}$ of probability measures
on a metric space $\M$ is said to be   tight
if, for every $\eps > 0$, there exists a compact set $K_\eps\subset \M$
such that $ \sup_{n\in \J}\mu_n(K_\eps^c) \leq \eps$. 
We say  that $\{ \mu_n \}_{n \in \J}$ is relatively compact, if every sequence 
in  $\{ \mu_n \}_{n \in \J}$ contains a weakly convergent subsequence. 

\end{definition}

%
We now 
recall the following 
Prokhorov theorem from \cite{Bill99}.

\begin{lemma}[Prokhorov theorem]
\label{LEM:Pro}

If a sequence of probability measures 
on a metric space $\M$ is tight, then
it is relatively compact. 
If in addition, $\M$ is separable and complete, then relative compactness is 
equivalent to tightness.

\end{lemma}

Lastly, we recall  the following Skorokhod representation 
theorem from   \cite[Chapter 31]{Bass}.

\begin{lemma}[Skorokhod representation theorem]\label{LEM:Sk}

Let $\M$ be a complete  
separable metric space \textup{(}i.e.~a Polish space\textup{)}.
Suppose that    
a sequence 
 $\{\mu_n\}_{n\in\N}$
 of  probability measures 
 on $\M$ converges  weakly   to a probability measure $\mu$
as $n \to\infty$.
Then, there exist a probability space $(\wt \O, \wt \F, \wt\PP)$,
and random variables $X_n, X:\wt \O \to \M$ 
such that 
\begin{align*}
\Law( X_n) = \mu_n
\qquad \text{and}\qquad
\Law(X) = \mu ,
\end{align*}

\noi
and $X_n$ converges $\wt\PP$-almost surely to $X$ as $n\to\infty$.

\end{lemma}

%
%
%
%
%
%
%
%
%
%
%
%

\section{Coming down from infinity}
\label{SEC:3}

In this section, we study the  following stochastic nonlinear heat equation (SNLH) on $\R^2$:
\begin{align}
\begin{cases}
\dt X +(1 -\Dl) X  + :\!X^k\!: \, = \sqrt 2\xi \\
X|_{t=0} = X_0,
\end{cases} 
\label{SHE1}
\end{align}

\noi
where $k \in 2\N + 1$ and 
$\xi$ is  a Gaussian space-time white noise on $\R_+\times\R^2$.
Let $Z$ denote the 
solution to the following linear equation:
\begin{align}
\begin{cases}
\dt Z + (1-\Dl ) Z  = \sqrt 2\xi \\
Z|_{t = 0} = 0.
\end{cases}
\label{SHE2}
\end{align}

\noi
Namely, $Z$ is the stochastic convolution, 
formally given by 
\begin{align*}
Z(t) = \sqrt 2\int_0^t P(t - t')  \xi (dt'), 
\end{align*}

\noi
where $P(t)= e^{t(\Dl-1)}$ is the linear heat propagator.
Then, with the first order expansion
\begin{align}
X = Y + Z,
\label{SHE3a}
\end{align}

\noi
 the remainder term $Y = X-Z$ satisfies
\begin{align}
\begin{cases}
\dt Y + (1-\Dl) Y 
+ \sum_{\l=0}^k {k\choose \ell} :\!Z^\l\!: \,Y^{k-\l} = 0\\
Y|_{t=0} = X_0,
\end{cases} 
\label{SHE4}
\end{align}

\noi
where 
$:\!Z^\l\!:$ denotes the Wick-renormalized power of $Z$.
See Section 5 in \cite{MW17a} for a further discussion.
We say that $X$ is a solution to \eqref{SHE1}
if it satisfies \eqref{SHE3a} together with \eqref{SHE2}
and~\eqref{SHE4}.
Our main goal in this section is to establish
coming down from infinity for a solution $Y$ to~\eqref{SHE4} 
in weighted Sobolev spaces of positive regularities
(Proposition \ref{PROP:CI3}),
which will play a crucial role in Subsection \ref{SUBSEC:4.2}.
For this purpose, we  consider 
\begin{align}
\begin{cases}
\dt Y + (1-\Dl) Y 
+ \sum_{\ell=0}^k Z^{(\l)}Y^{k-\ell} = 0\\
Y|_{t=0} = X_0,
\end{cases} 
\label{SHE5}
\end{align}

\noi
where $Z^{(0)}= 1$ and $Z^{(\l)}$, $\l = 1 , \dots, k$, 
are given space-time distributions.
We 
first establish
coming down from infinity for a solution $Y$ to \eqref{SHE5}
in  weighted Lebesgue spaces 
in Subsection~\ref{SUBSEC:3.1}.
In Subsection~\ref{SUBSEC:3.2}, 
we then 
establish
coming down from infinity
in weighted Sobolev spaces of positive regularities
(Proposition \ref{PROP:CI3}).
As a corollary to the coming down from infinity
in  weighted Lebesgue spaces, 
we construct a limiting Gibbs measure $\rhoo_\infty$
on the plane $\R^2$ 
(Theorem \ref{THM:1}\,(i));
see Subsection \ref{SUBSEC:3.3}.
We point out that
while the construction of a limiting Gibbs measure on $\R^2$
requires only 
the coming down from infinity 
in  weighted Lebesgue spaces
(Proposition \ref{PROP:CI1}), 
the coming down from infinity 
in  weighted Sobolev spaces of positive regularity
(Proposition \ref{PROP:CI3})
plays an essential role in the construction
of the enhanced Gibbs measure
as well as 
global-in-time dynamics on $\R^2$
presented in Section \ref{SEC:GWP}.

In the following discussion, 
the regularity  of the initial data $X_0$ in \eqref{SHE5} does not
play an important role
(which is precisely the point of coming down from infinity)
as long as a solution $X = Y+Z$ exists.
For this reason, we do not precisely state the regularity
of the initial data $X_0$ in
Propositions \ref{PROP:CI1} and \ref{PROP:CI3}.
For our application to \eqref{SHE1}
(and its $L$-periodic counterpart~\eqref{SHE7}), 
it suffices to take $X_0 \in B^{-\eps, \mu}_{p_0, \infty}(\R^2)$
for some $\eps > 0$,  $p_0 \gg1 $,  and $\mu > 0$
so that a global solution is guaranteed to exist; see \cite{MW17a}.
In particular, the initial data distributed by the $L$-periodic $\Phi^{k+1}_2$-measure
$\rho_L$ in \eqref{Gibbs4} almost surely satisfies this regularity condition.

%

\subsection{Coming down from infinity on weighted Lebesgue spaces}
\label{SUBSEC:3.1}

In this subsection, 
our main goal is to prove the following coming down from infinity
for a solution $Y$ to \eqref{SHE5}
in the weighted Lebesgue space~$L^p_\mu(\R^2)$.
As compared to the previous work \cite{TW18}
on $\T^2$, 
we need to proceed with more care,  
using a decomposition of the physical space into unit cubes
and  the Littlewood-Paley decomposition; see \eqref{C6}-\eqref{C15}.

\begin{proposition}\label{PROP:CI1}
Let $k \in 2\N + 1$ and $T> 0$.
Let $Y$ be a solution to  \eqref{SHE5}
on the time interval $[0, T]$.
Then, 
given any finite $p\geq 1$ and $\mu > 0$, 
there exist $\ld > 1$, small $\eps > 0$, finite $p_0 \gg 1$, and $0 < \mu_0 < \mu$ such that 
\begin{align}
\| Y(t) \|_{L^p_\mu} 
\leq C \bigg\{ t^{-\frac{1}{(\ld-1)p}} 
\vee   
\sum_{\l = 0}^k \|Z^{(\l)}(t)\|_{B^{-\eps, \mu_0}_{p_0, p_0}}^{\frac{p_0}{\ld p}}
 \bigg\}
 \label{C00}
\end{align}

\noi
for $0 < t \le T$, 
where the constant $C$ is independent of  the initial condition  $X_0$
in \eqref{SHE5} and $T > 0$.
\end{proposition}

For the later use, we set
\begin{align}
C_{\bar Z, p, \mu}(t) = 
\sup_{0 \le t' \le t}
\sum_{\l = 0}^k \|Z^{(\l)}(t')\|_{B^{-\eps, \mu_0}_{p_0, p_0}}^{\frac{p_0}{\ld p}}.
\label{CZ}
\end{align}

Before proceeding to the proof of Proposition \ref{PROP:CI1}, 
we first recall the following key lemma; see Lemma 3.8 in \cite{TW18}.

\begin{lemma}\label{LEM:CI2}
Let $F: [0, T] \to [0, \infty)$ be a differentiable function.
Suppose that there exist $\ld > 1$ and $c_1, c_2 > 0$  such that 
\begin{align} 
\dt F(t) + c_1 F^\ld(t) \le c_2
\label{CI1}
\end{align}

\noi
for any $0 \le t \le T$.
Then, we have
\begin{align*}
F(t)
& \le \frac{F(0)}{\big(1 + t F^{\ld-1}(0)(\ld-1)\frac {c_1} 2\big)^\frac{1}{\ld-1}}
\vee \bigg(\frac{2c_2}{c_1}\bigg)^\frac 1\ld\\
& \le \Bigg\{t^{-\frac 1{\ld-1}}\bigg((\ld-1) \frac {c_1} 2\bigg)^{-\frac{1}{\ld-1}}\Bigg\}
\vee \bigg(\frac{2c_2}{c_1}\bigg)^\frac 1\ld
\end{align*}

\noi
for any $0 <  t \le T$.
\end{lemma}

Lemma \ref{LEM:CI2}  shows that in order to establish coming down from infinity, 
it suffices to establish a bound of the form \eqref{CI1}.

We now present the proof of Proposition \ref{PROP:CI1}.

\begin{proof}[Proof of Proposition \ref{PROP:CI1}] 

By H\"older's inequality, it suffices to consider $p \in 2\N$.
 Fix  $p \in 2\N$.
Then, using the equation  \eqref{SHE5}, we have 
\begin{align}
\begin{aligned}
\frac{1}{p} \dt \| Y(t) \|_{L^p_\mu}^p 
&= \jb{ \dt Y(t), Y^{p-1}(t) w_\mu  }_{L^2_x} \\
&= \bigg\langle \Dl Y(t) - Y(t) -  \sum_{\l=0}^k Z^{(\l)}(t) Y^{k-\l}(t), Y^{p-1}(t) w_\mu 
\bigg\rangle_{L^2_x}.
\end{aligned}
\label{C0}
\end{align}

\noi
For simplicity of notation, we will drop the time dependence in the following computations.

From \eqref{weight1} with $0 < \dl < 1$, we have 
\begin{align}\label{C1}
\begin{split}
  \langle \Dl Y, Y^{p-1} w_\mu   \rangle_{L^2}
  &= -\int_{\R^2} \nb Y \cdot \nb( Y^{p-1} w_\mu ) dx \\
  &=- (p-1)\int_{\R^2}  |\nb Y|^2  Y^{p-2}  w_\mu   dx \\
& \hphantom{X}  + \mu\dl \int_{\R^2} (\nb Y \cdot x) Y^{p-1} \jb{x}^{\dl-2} w_\mu dx \\
  &=: -(p-1) K_t + B_0.
\end{split}
\end{align}

\noi
Since $p$ is even, we have
\begin{align}\label{C2}
\begin{aligned}
  \langle -Y,  Y^{p-1} w_\mu   \rangle_{L^2} 
  &=  - \int_{\R^2} Y^p w_\mu dx \le 0.
\end{aligned}
\end{align}

\noi
As for the contribution from the last term in \eqref{C0}, 
recalling that $Z^{(0)}= 1$, we write
\begin{align}
\label{C3}
\begin{aligned}
 \bigg\langle &  -  \sum_{\l=0}^k  Z^{(\l)} Y^{k-\l}, Y^{p-1} w_\mu  \bigg\rangle_{L^2} \\
 &= -\int_{\R^2} Y^{p+k-1} w_\mu dx 
  - \sum_{\l = 1}^{k} \int_{\R^2} Y^{p + k - \l-1 } Z^{(\l)} w_\mu dx\\
  &=: - L_t     - \sum_{\l = 1}^{k} B_\l.
 \end{aligned}
 \end{align}

\noi
Since $p \in 2\N$ and $k \in 2\N+1$, 
we have $K_t$ and $L_t$ are non-negative.
In the following, we control the terms $B_\l$, $\l = 0, \dots, k$, 
by these non-negative terms.
For this reason, we set
\begin{align}
M_t = K_t + L_t. 
\label{C3a}
\end{align}

We first treat $B_0$ in \eqref{C1}.
In view of the fast decay of the weight $w_\mu$, 
it follows from H\"older's inequality that 
\begin{align}
\| Y\|_{L^p_\mu}^p  \le C_{\mu, \dl} \| Y\|_{L^{p+k-1}_\mu}^p
=  C_{\mu, \dl} L_t^\frac{p}{p+k-1}
\le C_{\mu, \dl} M_t^\frac{p}{p+k-1}.
\label{C4}
\end{align}

\noi
Then, by recalling $0 < \dl < 1$
and applying
 Cauchy-Schwarz's inequality
 and Young's inequality with \eqref{C4}
 and \eqref{C3a}, 
we have 
\begin{align}
\begin{aligned}
|B_0| &\le \mu \dl 
 \int_{\R^2} |\nb Y|  \cdot  |Y|^{p-1} \jb{x}^{\dl-1} w_\mu dx \\
&  \le  C_{\mu, \dl}
  \int_{\R^2}
   |\nb Y|    |Y|^{\frac{p-2}{2}}  \cdot  |Y|^{\frac{p}{2}}    w_\mu dx 
 \\
 &\leq C_{\mu, \dl} K_t^{\frac12} \| Y\|_{L^p_\mu}^{\frac p2}  
 \leq  \tfrac{1}{100} K_t + C'_{\mu, \dl} \| Y\|_{L^p_\mu}^{p}   \\
 &\leq \tfrac{1}{100} M_t + C''_{ \mu, \dl}.
  \end{aligned}
  \label{C5}
\end{align}

Next,  we consider $B_\l$, $\l = 1, \dots, k$.
Let $Q= \big[-\frac 12, \frac 12\big)^2$ be the unit cube in $\R^2$
and set $Q_n = n + Q$.
Fix small $\eps > 0$
and $ r > 1$ sufficiently close to $1$.
Then, we have 
\begin{align}
\begin{split}
|B_\l| 
& = 
\bigg|\int_{\R^2} Y^{p + k - \l-1 } Z^{(\l)} w_\mu dx\bigg|
\le \sum_{n \in \Z^2}  \bigg|
\int_{Q_n} Y^{p + k - \l-1 } Z^{(\l)} w_\mu dx\bigg|\\
& \le \sum_{n \in \Z^2}  
\| Y^{p + k - \l-1 }w_\mu\|_{W^{\eps, r}(Q_n)} 
\|Z^{(\l)}\|_{W^{-\eps, r'}(Q_n)}, 
\end{split}
\label{C6}
\end{align}

\noi
where $\frac 1r + \frac 1{r'} = 1$.

From \eqref{besov2}
and the Littlewood-Paley decomposition
with the uniform boundedness on $L^r(\R^2)$
of the Littlewood-Paley projector $\Q_j$, we have
\begin{align}
\begin{split}
& \|   Y^{p + k - \l-1 } w_\mu\|_{W^{\eps, r}(Q_n)} 
\les \| Y^{p + k - \l-1 }w_\mu\|_{B^{2\eps}_{r, \infty}(Q_n)} \\
& \hphantom{X}=   
\sup_{j \in \Z_{\ge 0}}2^{2\eps j}
\| \Q_j (Y^{p + k - \l-1 }w_\mu)\|_{L^r(Q_n)} \\
& \hphantom{X}\le 
\sup_{j \in \Z_{\ge 0}}   2^{2\eps j} \sum_{j_1, j_2 = 0}^\infty
\big\| \Q_j \big(\Q_{j_1}(Y^{p + k - \l-1 })\Q_{j_2}w_\mu\big)\big\|_{L^r(Q_n)} \\
& \hphantom{X}\les    \sum_{j_1 \ge  j_2 +2}  2^{2\eps j_1}
\| \Q_{j_1}(Y^{p + k - \l-1 })\Q_{j_2}w_\mu\|_{L^r(Q_n)} \\
& \hphantom{XX}
 + \sum_{j_1 <   j_2 +2}
2^{2\eps j_2}
\| \Q_{j_1}(Y^{p + k - \l-1 })\Q_{j_2}w_\mu\|_{L^r(Q_n)} \\
&\hphantom{X}
 =: \1 + \II.
\end{split}
\label{C7}
\end{align}

\noi
As for the first term $\1$, it follows from 
\eqref{weight3} and 
\eqref{weight2a} that 
\begin{align}
\begin{split}
\1 
& \les
\sum_{j_1 \ge   j_2 +2}
2^{2\eps j_1}
2^{-\frac  12 j_2}w_\mu (n)
\| \Q_{j_1}(Y^{p + k - \l-1 })\|_{L^r(Q_n)} \\
& \les w_\mu (n)
\| Y^{p + k - \l-1 }\|_{W^{3\eps, r}(Q_n)}.
\end{split}
\label{C8}
\end{align}

\noi
Similarly, we have
\begin{align}
\begin{split}
\II
& \les
\sum_{j_1 <   j_2 +2}
2^{(2\eps -\frac  12) j_2}w_\mu (n)
\| \Q_{j_1}(Y^{p + k - \l-1 })\|_{L^r(Q_n)} \\
& \les w_\mu (n)
\| Y^{p + k - \l-1 }\|_{W^{3 \eps, r}(Q_n)}.
\end{split}
\label{C9}
\end{align}

In the following, we estimate $\| Y^{p + k - \l-1 }\|_{W^{3\eps, r}(Q_n)}$.
Define $M_t(Q_n)$ 
by 
\begin{align}
M_t(Q_n)  = 
\int_{Q_n}  |\nb Y|^2  Y^{p-2} 
+  Y^{p+k-1}  dx .
\label{C9a}
\end{align}

\noi
By the interpolation \eqref{interp}, we have 
\begin{align}
\begin{split}
& \| Y^{p + k - \l-1 }  \|_{W^{3\eps, r}(Q_n)}
 \les 
\| Y^{p + k - \l-1 }\|_{L^r(Q_n)}^{1-3\eps}
\| Y^{p + k - \l-1 }\|_{W^{1, r}(Q_n)}^{3\eps}\\
& \quad \sim 
\| Y^{p + k - \l-1 }\|_{L^r(Q_n)}^{1-3\eps}
\|\nb ( Y^{p + k - \l-1 })\|_{L^r(Q_n)}^{3\eps}
+ 
\| Y^{p + k - \l-1 }\|_{L^r(Q_n)}.
\end{split}
\label{C10}
\end{align}

\noi
By H\"older's inequality
and by choosing $r > 1$ sufficiently close to $1$, 
we have 
\begin{align}
\| Y^{p + k - \l-1 }\|_{L^r(Q_n)}
= 
\| Y\|_{L^{(p+k-\l-1)r}(Q_n)}^{p+k-\l-1}
\leq M_t(Q_n)^\frac{p+k-\l-1}{p+k-1}
\label{C11}
\end{align}

\noi
for any 
$\l = 1, \dots, k$.
On the other hand, by H\"older's inequality, 
we have
\begin{align}
\begin{split}
\|\nb ( Y^{p + k - \l-1 })\|_{L^r(Q_n)}
& \sim \|\nb  Y \cdot|Y|^\frac{p-2}{2}  Y^{\frac p2 + k - \l-1 }\|_{L^r(Q_n)}\\
& \leq M(Q_n)^\frac 12
\|  Y^{\frac p2 + k - \l-1 }\|_{L^\frac{2r}{2-r}(Q_n)}\\
& \leq M(Q_n)^\frac 12
 \big\||Y|^\frac p2 \big\|_{L^q(Q_n)}^{1 + \frac 2p (k - \l - 1)},
 \end{split}
\label{C12a}
\end{align}

\noi
where
$q = \frac 2p \cdot \frac {2r}{2-r} \big(\frac p 2 + k - \l - 1\big)$.
By Sobolev's inequality,\footnote{When $\l = k$, we have $q \le 2$ for $r > 1$
sufficiently close to $1$,  and thus there is no need to apply Sobolev's inequality.} the interpolation \eqref{interp}, 
and Young's inequality followed by H\"older's inequality, we have
\begin{align}
\begin{split}
\big\||Y|^\frac p2 \big\|_{L^q(Q_n)}
& \les \big\||Y|^\frac p2 \big\|_{L^2(Q_n)} 
+ \big\||Y|^\frac {p-2}2 \nb Y\big\|_{L^2(Q_n)}\\
& \le M_t(Q_n)^{\frac 12 \cdot \frac p {p+k-1}}+ 
M_t(Q_n)^\frac 12.
\end{split}
\label{C13}
\end{align}

\noi
Hence, from \eqref{C10}, \eqref{C11}, \eqref{C12a}, and \eqref{C13}, 
we obtain
\begin{align}
& \| Y^{p + k - \l-1 }  \|_{W^{3\eps, r}(Q_n)}
 \les 
M_t(Q_n)^{1-\ta} + 1
\label{C14}
\end{align}

\noi
for some $\ta > 0$, provided that $\eps > 0$ is sufficiently small, 
where the implicit constant is independent of $\l = 1, \dots, k$.

Hence, putting 
\eqref{C6}, \eqref{C7}, 
\eqref{C8}, \eqref{C9},
and \eqref{C14}
together
and applying Young's inequality with \eqref{besov2}, we have
\begin{align}
\begin{split}
|B_\l| 
& \les \sum_{n \in \Z^2}  
w_\mu(n) \Big(M_t(Q_n)^{1-\ta} + 1\Big)
\|Z^{(\l)}\|_{W^{-\eps, r'}(Q_n)}\\
& \le \eps_0 
\sum_{n \in \Z^2}  
w_\mu(n) M_t(Q_n)
+ C_{\eps_0}
\sum_{n \in \Z^2}  
w_\mu(n)
\|Z^{(\l)}\|_{W^{-\eps, r'}(Q_n)}^\frac1\ta
+C_1\\
& \les \eps_0 
M_t + C_{\eps_0}
\|Z^{(\l)}\|_{B^{-\frac \eps2,  \mu_0}_{p_0, p_0}}^{p_0}
+C_2
\end{split}
\label{C15}
\end{align}

\noi
for some small $\eps_0>0$, finite $p_0 \gg1$, 
and $\mu_0 <   \mu$.
Here, the last step follows from 
\eqref{C3a}, \eqref{C9a}, and \eqref{weight2a}
for the first term, 
while it follows from 
\eqref{besov1} and \eqref{weight2a} for the second term.

Therefore, from 
\eqref{C0},
\eqref{C1},
\eqref{C2}, 
\eqref{C3},
\eqref{C5}, and \eqref{C15}, 
we obtain 
\begin{align*}
\frac{1}{p} \dt \| Y(t) \|_{L^p_\mu}^p 
+ c_0 M_t \les 
\sum_{\l = 0}^k \|Z^{(\l)}(t)\|_{B^{-\frac{\eps}{2}, \mu_0}_{p_0, p_0}}^{p_0}.
\end{align*}

\noi
Finally, in view of 
\eqref{C4}, 
there exists $\ld > 1$ such that 
\begin{align}
\begin{aligned}
\frac{1}{p} \dt \| Y(t) \|_{L^p_\mu}^p 
+ c_1 \| Y(t) \|_{L^p_\mu}^{\ld p }
 \les 
\sum_{\l = 0}^k \|Z^{(\l)}(t)\|_{B^{-\frac{\eps}{2}, \mu_0}_{p_0, p_0}}^{p_0}
\end{aligned}
\label{C16}
\end{align}

\noi
for any $0 \le t \le T$.
Finally, the desired bound \eqref{C00} follows
from \eqref{C16} and Lemma \ref{LEM:CI2}.
This concludes the proof of Proposition \ref{PROP:CI1}.
\end{proof}

\subsection{Coming down from infinity on  weighted Sobolev spaces
of positive regularities}
\label{SUBSEC:3.2}

In this subsection, we establish the following coming down from infinity
for a solution $Y$ to \eqref{SHE5} in a weighted Sobolev space
of positive regularity.

\begin{proposition}\label{PROP:CI3}

Let $k \in 2\N + 1$ and  $T> 0$.
Let $Y$ be a solution to  \eqref{SHE5}
on the time interval $[0, T]$. 
Then, 
given any 
$0 < s < 1$, 
finite $p\geq 1$,  and $\mu > 0$, 
there exist $\ld, \ld' > 1$, small $ \ta > 0$, finite $q_0, q \gg 1$, 
and $0 < \mu', \mu_1 < \mu$ such that 
\begin{align}
\begin{split}
\| Y(t)\|_{W^{s, p}_\mu}
& \le C(T) \bigg\{
t^{-\frac s2}
\Big( t^{-\frac{1}{( \ld' - 1) p} } \vee  C_{\bar Z, p, \mu'}(t)\Big)\\
& \hphantom{lXXXX}
+ t^{-\frac{1}{\ld - 1} }\vee
\big(C_{\bar Z, q_0, \mu}(t)\big)^{q_0} + Q_{\bar Z, \ta, q, \mu_1} (t) + 1
\bigg\}
\end{split}
\label{CZ2a}
\end{align}

\noi
for $0 < t \le T$, 
where the constant $C(T)$ is independent of  the initial condition  $X_0$
in \eqref{SHE5}.
Here, $C_{\bar Z, p, \mu}$ is as in \eqref{CZ}
and 
\begin{align}
Q_{\bar Z, \ta, q, \mu_1} (t) = \sup_{0 \le t' \le t} \sum_{\l=0}^k  \|  Z^{(\l)} (t') \|_{B^{-\frac \ta2, \mu_1}_{q, q}}^q.
\label{CZ2}
\end{align}

\end{proposition}

The proof of Proposition \ref{PROP:CI3} 
is based on a Gronwall-type argument, 
utilizing the already established coming down from infinity
in weighted Lebesgue spaces
(Proposition \ref{PROP:CI1}).
While the idea is straightforward, 
an actual implementation requires
careful analysis via a spatial decomposition
(analogous to the proof of Lemma \ref{LEM:heat}),
for which we find our definition~\eqref{WSP}
of the weighted Sobolev spaces $W^{s, p}_\mu(\R^2)$ more convenient
(than $\W^{s, p}_\mu(\R^2)$ defined in \eqref{WSP2}).

\begin{proof}

Fix $0 < r < t$.
Then, from 
  \eqref{SHE5} and
  \eqref{Lpm1}, 
we have 
\begin{align}
\begin{aligned}
&  \| Y  (t) \|_{W^{s,p }_\mu }
\leq
\|P(t-r) Y(r) \|_{W^{s, p }_\mu }
+ \sum_{\l=0}^k   \int_{r}^t\| P(t -t')( Z^{(\l)} Y^{k-\l} )(t') \| _{W^{s, p }_\mu } dt'  
\\
&\le
\| P(t-r) Y(r) \|_{W^{s,p }_\mu }
+ \sum_{\l=0}^k   \sum_{j = 0}^\infty e^{-\frac{\mu}{p}  2^{j \dl}}
\int_r^t  \| \phi_j P(t -t')( Z^{(\l)}Y^{k-\l} )(t') \| _{W^{s,p}} dt' .
\end{aligned}
\label{W1}
\end{align}

Let  $F_{k. \l} = Z^{(\l)}Y^{k-\l} $.
We proceed as in the proof of Lemma \ref{LEM:heat}
and estimate 
\[\| \phi_j P(t -t')F_{k, \l} \| _{W^{s,p}}.\]

\noi
With \eqref{heat4}, the fractional Leibniz rule (Lemma \ref{LEM:bilin1}\,(i)), 
Young's inequality,  and \eqref{heat6}, 
we have 
\begin{align}
\begin{aligned}
  \| \phi_j P(t) F_{k, \l}\| _{W^{s,p}} 
&\les 
e^{-t}
\sum_{(m, m') \in\Ld_{t, j}} 
 \big\| \phi_j \big[  (\phi^t_m p_t)\ast ( \phi_{m'}F_{k, \l}) \big] \big\|_{W^{s, p}} \\
&\les  
e^{-t}
\sum_{(m, m') \in\Ld_{t, j}} 
  \| \phi_j \|_{W^{s,\infty}} \big\|    (\phi^t_m p_t)\ast ( \phi_{m'}F_{k, \l})  \big\|_{W^{s, p}} 
 \\
 &\les 
e^{-t}
\sum_{(m, m') \in\Ld_{t, j}} 
 \|  \phi^t_m p_t  \|_{W^{s+\ta, 1}}   \| \phi_{m'}F_{k, \l} \|_{W^{-\ta, p}}\\
 &\les 
t^{-\frac{s+\ta}{2}} 
 e^{-t}
\sum_{(m, m') \in\Ld_{t, j}} 
e^{ -c_1 4^m }
    \| \phi_{m'}F_{k, \l} \|_{W^{-\ta, p}}
\end{aligned}
\label{W3}
\end{align}

\noi
for any  $\ta > 0$, 
where $\phi^t_m$ and
$\Ld_{t, j}$ are as in~\eqref{heat3} and~\eqref{heat5}.
In the following, we choose $\ta > 0$ sufficiently small
such that $s + \ta < 1$.

From 
 \eqref{ID1},
Lemma \ref{LEM:bilin1}\,(ii) (with finite $q \gg 1$), 
the fractional Leibniz rule (Lemma~\ref{LEM:bilin1}\,(i) with $q_0 \gg 1$),
the interpolation \eqref{interp}, 
and H\"older's inequality with \eqref{phi2}, 
we have
\begin{align}
\begin{aligned}
 \| \phi_{m'}F_{k, \l} \|_{W^{-\ta, p}} 
&=  \|  (\wt \phi_{m'} Y)^{k-\l} \cdot  \phi_{m'} Z^{(\l)} \|_{W^{-\ta, p}} 
 \\
&\les   
\|   (\wt \phi_{m'} Y)^{k - \l}  \|_{W^{\ta, p-}}
\| \phi_{m'}  Z^{(\l)} \|_{W^{-\ta,q}} \\
&\les 
\| \wt \phi_{m'}Y \|_{L^{q_0}}^{k - \l-1} 
\| \wt \phi_{m'}Y  \|_{W^{\ta, p}} 
\| \phi_{m'}  Z^{(\l)} \|_{W^{-\ta,q}}  \\
&\les 
\| \wt \phi_{m'}Y  \|_{L^{q_0}}^{k-\l -1} 
\| \wt \phi_{m'}Y  \|_{W^{s, p}}^{\frac\ta s}  
\| \wt  \phi_{m'}Y  \|_{L^{p}}^{1- \frac\ta s}  
\| \phi_{m'}  Z^{(\l)} \|_{W^{-\ta, q}}\\
&\les 
\| \wt \phi_{m'}Y  \|_{L^{q_0}}^{k-\l -1} 
\| \wt \phi_{m'}Y  \|_{W^{s, p}}^{\frac\ta s}  \\
& \quad 
\times 2^{(\frac{1}{p}- \frac 1{q_0})(1- \frac \ta s) 2m'}
\| \wt  \phi_{m'}Y  \|_{L^{q_0}}^{1- \frac\ta s}  
\| \phi_{m'}  Z^{(\l)} \|_{W^{-\ta, q}}, 
\end{aligned}
\label{W5}
\end{align}

\noi
provided that $0 < \ta \le s$.
Then, it follows from \eqref{W5} with 
\eqref{Lpm1} (for $\wt \phi_{m'}$ defined in \eqref{phi3a}) 
that 
\begin{align}
\begin{aligned}
 \| \phi_{m'}F_{k, \l} \|_{W^{-\ta, p}} 
&  \les
2^{(\frac{1}{p}- \frac 1{q_0})(1- \frac \ta s) 2m'}
e^{\frac{\mu}{q_0} a_1 (k-\l-\frac \ta s)2^{(m'+1)\dl}}
\| Y  \|_{L^{q_0}_\mu}^{k-\l -\frac \ta s} \\
& \quad
\times  e^{\frac{\mu}{p}\frac \ta s  2^{(m'+1)\dl}}
\| Y  \|_{W^{s, p}_\mu}^{\frac\ta s}  \, 
 e^{\frac{\mu}{q}  2^{m'\dl}}
\|  Z^{(\l)} \|_{W^{-\ta, q}_\mu}.
\end{aligned}
\label{W7}
\end{align}

Note that, given any small $\kk_0>0$, 
there exists $C(p, q_0, s, \ta, \kk_0) > 0$ 
such that 
\begin{align}
2^{(\frac{1}{p}- \frac 1{q_0})(1- \frac \ta s) 2m'}
 \le C(p, q_0,  s, \ta, \kk_0) e^{\kk_0 2^{m'\dl}}
\label{W7c}
\end{align}

\noi
for  any $m' \in \Z_{\ge 0}$.
Hence, from 
\eqref{W3}  and \eqref{W7} with \eqref{W7c}, 
 we have
\begin{align*}
&  \sum_{j = 0}^\infty e^{-\frac{\mu}{p}  2^{j \dl}}
  \| \phi_j P(t)F_{k, \l} \| _{W^{s,p}} \\
& \quad 
\les
t^{-\frac{s+\ta}{2}} 
 e^{-t}
 \sum_{j = 0}^\infty e^{-\frac{\mu}{p}  2^{j \dl}}
\sum_{(m, m') \in\Ld_{t, j}} 
e^{ -c_1 4^m }
 e^{\wt \mu2^{m'\dl}}\\
& \quad \quad
\times \| Y  \|_{L^{q_0}_\mu}^{k-\l -\frac \ta s} 
\| Y  \|_{W^{s, p}_\mu}^{\frac\ta s}  
\|  Z^{(\l)} \|_{W^{-\ta, q}_\mu}, 
\end{align*}

\noi
where $\wt \mu$ is given by 
\begin{align*}
\wt  \mu = \frac{\mu}{q_0} a_1 \bigg(k-\l-\frac \ta s\bigg)2^{\dl}
+ \frac{\mu}{p} \frac \ta s 2^{\dl}
+ \frac{\mu}{q}  + \kk_0.
\end{align*}

\noi
 By choosing sufficiently large $q_0, q \gg1$, $0 < \ta \ll s$, 
 and sufficiently small $\kk_0 > 0$, 
we have $p \wt \mu \ll \mu$.
Then, by summing over $m'$ with 
 \eqref{heat5} (as in the proof of Lemma \ref{LEM:heat}), 
 then summing over $m$ (with \eqref{heat7a}; see also \eqref{heat8}),
 and finally summing over $j$ with $p \wt \mu \ll \mu$, 
we obtain
\begin{align}
\begin{split}
&  \sum_{j = 0}^\infty e^{-\frac{\mu}{p}  2^{j \dl}}
  \| \phi_j P(t)F_{k, \l} \| _{W^{s,p}} \\
& \quad 
\les
t^{-\frac{s+\ta}{2}} 
 e^{-t}
 \sum_{j, m = 0}^\infty e^{-\frac{\mu}{p}  2^{j \dl}}
e^{ -c_1 4^m }
 e^{c_0 \wt \mu (2^{j\dl} +  2^{m\dl} t^\frac \dl 2)}\\
& \quad \quad
\times \| Y  \|_{L^{q_0}_\mu}^{k-\l -\frac \ta s} 
\| Y  \|_{W^{s, p}_\mu}^{\frac\ta s}  
\|  Z^{(\l)} \|_{W^{-\ta, q}_\mu}\\
& \quad \les
t^{-\frac{s+\ta}{2}} 
 e^{-c t}
 \sum_{j = 0}^\infty e^{-(\frac{\mu}{p}  -c_0 \wt \mu) 2^{j \dl}}
 \| Y  \|_{L^{q_0}_\mu}^{k-\l -\frac \ta s} 
\| Y  \|_{W^{s, p}_\mu}^{\frac\ta s}  
\|  Z^{(\l)} \|_{W^{-\ta, q}_\mu}  \\
& \quad \les
t^{-\frac{s+\ta}{2}} 
 e^{-c t}
 \| Y  \|_{L^{q_0}_\mu}^{k-\l -\frac \ta s} 
\| Y  \|_{W^{s, p}_\mu}^{\frac\ta s}  
\|  Z^{(\l)} \|_{W^{-\ta, q}_\mu}  
\end{split}
\label{W8}
\end{align}

\noi
for any $t > 0$.

Therefore, from  \eqref{W1}, \eqref{W8},
 Young's inequality,  and Lemma \ref{LEM:equiv}
 (with some $0 < \mu_1 < \mu$), we have
\begin{align}
\begin{split}
\| Y(t)\|_{W^{s, p}_\mu} 
& \les \| P(t-r) Y(r)\|_{W^{s, p}_\mu}\\
& \quad 
+ \sum_{\l=0}^k  \int_r^t (t-t')^{-\frac{s+\ta}{2}}  
 \| Y(t')  \|_{L^{q_0}_\mu}^{k-\l -\frac \ta s} 
\| Y (t') \|_{W^{s, p}_\mu}^{\frac\ta s}  
\|  Z^{(\l)} (t') \|_{W^{-\ta, q}_\mu}  dt'\\
&\les \| P(t-r) Y(r)\|_{W^{s, p}_\mu}
+   \int_r^t (t-t')^{-\frac{s+\ta}{2}}  
\| Y (t') \|_{W^{s, p}_\mu}  dt'\\
& \quad 
+ 
 \int_r^t (t-t')^{-\frac{s+\ta}{2}}  
\bigg( \| Y(t')  \|_{L^{q_0}_\mu}^{q_0} 
+\sum_{\l=0}^k  \|  Z^{(\l)} (t') \|_{B^{-\frac \ta2, \mu_1}_{q, q}}^q + 1 \bigg) dt'.
\end{split}
\label{W10b}
\end{align}

\noi
From Proposition \ref{PROP:CI1}
with \eqref{CZ}, 
there exists $\ld = \ld(q_0, \mu) > 1$ such that 
\begin{align}
\begin{split}
\| Y(t')  \|_{L^{q_0}_\mu}
& \les (t')^{-\frac{1}{(\ld - 1) q_0} }\vee C_{\bar Z, q_0, \mu}(t')\\
& \le r^{-\frac{1}{(\ld - 1) q_0} } \vee  C_{\bar Z, q_0, \mu}(t)
\end{split}
\label{W10c}
\end{align}

\noi
for $r \le t' \le t$.
Then, from \eqref{W10b}, 
Lemma \ref{LEM:heat}, \eqref{W10c}, and \eqref{CZ2}, we
have
\begin{align}
\begin{split}
\| Y(t)\|_{W^{s, p}_\mu} 
&\les (t-r)^{-\frac s2}\|  Y(r)\|_{L^{p}_{\mu'}}\\
& \quad
+C(T) \Big( r^{-\frac{1}{\ld - 1} }\vee
\big(C_{\bar Z, q_0, \mu}(t)\big)^{q_0} + Q_{\bar Z, \ta, q, \mu_1} (t) + 1\Big)\\
& \quad 
+   \int_r^t (t-t')^{-\frac{s+\ta}{2}}  
\| Y (t') \|_{W^{s, p}_\mu}  dt'.
\end{split}
\label{W10d}
\end{align}

\noi 
Let
\begin{align} 
\begin{split}
B(t, r) 
& = (t-r)^{-\frac s2 -}\|  Y(r)\|_{L^{p}_{\mu'}}\\
& \quad +C(T) \Big( 
r^{-\frac{1}{\ld - 1} }\vee
\big(C_{\bar Z, q_0, \mu}(t)\big)^{q_0} + Q_{\bar Z, \ta, q, \mu_1} (t) + 1\Big).
\end{split}
\label{W10e}
\end{align}

\noi
Then, recalling that $s + \ta < 1$, 
it follows from  \eqref{W10d} and Cauchy-Schwarz's inequality that 
\begin{align*}
\| Y(t)\|_{W^{s, p}_\mu}^2 
&\le C B(t, r)^2
+   C(T) \int_r^t 
\| Y (t') \|_{W^{s, p}_\mu}^2  dt'.
\end{align*}

\noi
Hence, by Gronwall's inequality, we obtain
\begin{align}
\| Y(t)\|_{W^{s, p}_\mu}
\le C(T) B(t, r)
\label{W10f}
\end{align}

\noi
for any $r \le t \le T$.
Finally, by choosing $r = \frac t 2$, 
it follows from \eqref{W10f}, 
\eqref{W10e} and Proposition~\ref{PROP:CI1} that 
\begin{align*}
\| Y(t)\|_{W^{s, p}_\mu}
& \le C(T) \bigg\{
t^{-\frac s2}
\Big( t^{-\frac{1}{( \ld' - 1) p} } \vee  C_{\bar Z, p, \mu'}(t)\Big)\\
& \hphantom{lXXXX}
+ t^{-\frac{1}{\ld - 1} }\vee
\big(C_{\bar Z, q_0, \mu}(t)\big)^{q_0} + Q_{\bar Z, \ta, q, \mu_1} (t) + 1
\bigg\}
\end{align*}

\noi
for some $\ld' = \ld(p, \mu') > 1$.
This concludes the proof of Proposition \ref{PROP:CI3}.
\end{proof}

\subsection{Construction of the Gibbs measure $\rhoo_\infty$ on the plane}
\label{SUBSEC:3.3}

We conclude this section by presenting the proof of Theorem \ref{THM:1}\,(i).

Let $X$ be a global-in-time solution to \eqref{SHE1} on $\R^2$
constructed in \cite{MW17a}, 
satisfying the decomposition $X = Y + Z$ in~\eqref{SHE3a}
 with $Y$ and $Z$ satisfying \eqref{SHE4}
and \eqref{SHE2}, respectively.
Fix $s < 0$, finite $p \ge 1$, and $\mu > 0$.
Then, by Lemma \ref{LEM:equiv},
\eqref{WSP3} with $s = 0$,  and Proposition \ref{PROP:CI1}, 
we have
\begin{align}
\begin{aligned}
\| X(t) \|_{W^{s, p}_\mu}
&\les \| Z(t) \|_{W^{s, p}_\mu}  + \| Y(t) \|_{L^p_{\mu'}}  \\
&\les \| Z(t) \|_{B^{s+ \eps, \mu''}_{p, p}}  + 
 \bigg\{ t^{-\frac{1}{(\ld-1)p}} 
\vee   
\sum_{\l = 0}^k \|:\!Z^\l(t)\!:\|_{B^{-\eps, \mu_0}_{p_0, p_0}}^{\frac{p_0}{\ld p}}
 \bigg\}
\end{aligned}
\label{CD1}
\end{align}

\noi
for a suitable choice of parameters $\mu'$, $\mu''$, $\ld$, $p_0$, $\eps$, and $\mu_0$, 
depending on $p$ and $\mu$, 
where 
the implicit constant is independent of the initial data $X_0$.

Given $L \ge 1$, 
let 
$\xi _L$ denote the (spatially) $L$-periodic space-time white noise
on $\R^2\times \R_+$
obtained by first restricting the space-time white noise $\xi$ on
$\R^2\times \R_+$, appearing in \eqref{SHE1}, 
onto $ \big[-\frac L2, \frac L2\big]^2 \times \R_+$
and then extending it periodically (with the spatial period $L$)
onto $\R^2 \times \R_+$.
See Section 5 in \cite{MW17a}
for a further discussion.
See  also \eqref{wh4} and  \eqref{wh5} below.
We now consider the following SNLH
with the $L$-periodic space-time white noise:
\begin{align}
\begin{cases}
\dt X_L +(1 -\Dl) X_L  + :\!X_L^k\!: \, = \sqrt 2\xi_L \\
X_L|_{t=0} = X_{0, L},
\end{cases} 
\label{SHE7}
\end{align}

\noi
where the initial data $X_{0, L}$ is assumed to be $L$-periodic.
This is 
the parabolic $\Phi^{k+1}_2$-model on $\T^2_L$
studied in \cite{DD} (with $L = 1$).
Write $X_L$ as  
\begin{align}
X_L = Y_L + Z_L,
\label{SHE7a}
\end{align}

\noi 
where 
$Z_L$ satisfies 
\begin{align}
\begin{cases}
\dt Z_L + (1-\Dl ) Z_L  = \sqrt 2\xi_L \\
Z_L|_{t = 0} = 0
\end{cases}
\label{SHE8}
\end{align}

\noi
and $Y_L$ satisfies
\begin{align}
\begin{cases}
\dt Y_L + (1-\Dl) Y_L 
+ \sum_{\l=0}^k {k\choose \ell} :\!Z_L^\l\!: \,Y_L^{k-\l} = 0\\
Y_L|_{t=0} = X_{0, L}.
\end{cases}
\label{SHE9} 
\end{align}

\noi
Then, by applying Lemma \ref{LEM:equiv} and Proposition \ref{PROP:CI1}
once again
we have
\begin{align}
\begin{aligned}
\| X_L(t) \|_{W^{s, p}_\mu}
&\les \| Z_L(t) \|_{W^{s, p}_\mu}  + \| Y_L(t) \|_{L^p_{\mu'}}  \\
&\les \| Z_L(t) \|_{B^{s+ \eps, \mu''}_{p, p}}  + 
 \bigg\{ t^{-\frac{1}{(\ld-1)p}} 
\vee   
\sum_{\l = 0}^k \|:\! Z_L^\l (t)\!: \|_{B^{-\eps, \mu_0}_{p_0, p_0}}^{\frac{p_0}{\ld p}}
 \bigg\}
\end{aligned}
\label{CD2}
\end{align}

\noi
where 
the implicit constant is independent of the initial data $X_{0, L}$
and of the period $L \ge 1$.
Moreover, it follows from 
\cite[Theorem 5.1]{MW17a}\footnote{Theorem 5.4 in the arXiv version.} 
that the $p$th moment of the right-hand side of~\eqref{CD2} (and of \eqref{CD1})
is bounded, uniformly in $L \ge 1$.

Let $\rho_L$ denote the $L$-periodic $\Phi^{k+1}_2$-measure on $\T^2_L$, 
appearing in \eqref{Gibbs4}.
By taking $\Law(X_{0, L}) = \rho_L$ in \eqref{SHE7}
(and we assume that $X_{0, L}$ is independent of the noise $\xi_L$), 
it is known 
\cite{DD} that 
\begin{align}
\Law(X_L(t)) = \rho_L
\label{CD2a}
\end{align}

\noi
 for any $t \ge 0$.
Moreover, from \eqref{CD2}
and the observation mentioned right after \eqref{CD2}, we have
\begin{align}
\sup_{L\geq 1} \E\Big[ \| X_L(1) \|_{W^{s, p}_\mu}^p \Big]
\les 1.
\label{CD3}
\end{align}

Now, given  $s < 0$  and $\mu > 0$, 
let $s < s' < 0$ and $0 < \mu' \ll  \mu$.
Then, given $M > 0$, set 
\begin{align*}
K_M = \big\{ f\in W^{s, p}_{\mu} (\R^2) : \| f\|_{W^{s', p}_{\mu'}} \leq M \big\}.
\end{align*}

\noi
Then,  it follows from Chebyshev's inequality
and \eqref{CD3}
that, given $\eps > 0$, 
there exists $M > 0$ such that 
\begin{align*}
\rho_L\big( K_M^c \big)
&= \PP\Big( 
 \| X_L(1; \o)\|_{W^{s', p}_{\mu'}} > M \Big) \les M^{-p} < \eps, 
\end{align*}

\noi
uniformly in   $L\geq 1$.  
On the other hand, 
it follows from Lemma \ref{LEM:cpt}
that 
$K_M$ is compact in $W^{s, p}_{\mu}(\R^2)$.
Therefore, from 
the Prokhorov theorem (Lemma \ref{LEM:Pro}), we conclude that $\{\rho_L\}_{L \in \N}$
is tight
and hence admits a 
subsequence $\{\rho_{L_j}\}_{j \in \N}$
which converges weakly to 
a limiting $\Phi^{k+1}_2$-measure $\rho_\infty$ on $\R^2$
as $j \to \infty$.

\medskip

Next, we prove weak convergence of the $L$-periodic white noise measure $\mu_{0, L}$
defined in~\eqref{mu0} (with $s = 0$)
to the limiting white noise measure on $\R^2$.
Formally, 
a (spatial) white noise $\z$ on $\R^2$
is a centered Gaussian distribution on $\R^2$
with covariance
\begin{align}
 \E \big[\z(x_1) \z(x_2)\big] = \dl (x_1 - x_2).
\label{white1}
\end{align} 

\noi
The expression \eqref{white1} is merely formal
but we can make it rigorous by testing it against a test function.

\begin{definition}\label{DEF:white} \rm
A (spatial) white noise $\z$ on $\R^2$
is a family of centered Gaussian random variables
$\{ \z(\varphi): \varphi \in L^2(\R^2)\}$
such that 
\begin{align*}
\E\big[ \z(\varphi)^2 \big] = \| \varphi\|_{L^2(\R^2)}^2
\quad \text{and} \quad 
\E\big[ \z(\varphi_1) \z( \varphi_2)\big] = \jb{ \varphi_1, \varphi_2}_{L^2(\R^2)}.
\end{align*}

\end{definition}

Given $k \in \N$, let $\eta_k$ be as in \eqref{eta1}.
Then, a direct computation with Lemma \ref{LEM:hyp} shows
\begin{align}
\E\big[ |\z(\eta_k)|^p\big]
\les_p 
\Big(\E\big[ \z(\eta_k)^2\big]\Big)^\frac{p}{2} = \| \eta_k\|_{L^2}^p \sim 2^{kp}
\label{white2a}
\end{align}

\noi
for any finite $p \ge 1$.
Hence, we conclude  from 
(the time-independent version of) Lemma 9 in~\cite{MW17a}\footnote{Lemma 5.2 in the arXiv version.}
that the white noise 
$\z$ on $\R^2$ belongs almost surely to  
$B^{s, \mu}_{p, q}(\R^2) $ for any $s < -1$, $\mu > 0$, $1\le p < \infty$, and $1\le  q\le \infty$, 
with the following bound:
\begin{align}
\E\big[\|\z\|_{B^{s, \mu}_{p, q}}^r\big]
\leq C(r, s, \mu, p, q) < \infty
\label{white2b}
\end{align}

\noi
for any finite $r \ge 1$.
Then, by Lemma \ref{LEM:equiv}, 
we see that 
 the white noise 
$\z$ on $\R^2$ also belongs almost surely to  
the weighted Sobolev space $W^{s, p}_\mu(\R^2)$ for any $s < -1$,  $1\le p <  \infty$, 
and $\mu > 0$,
with a  bound:
\begin{align}
\E\big[\|\z\|_{W^{s, p}_\mu}^r\big]
\leq C(r, s, \mu, p) < \infty
\label{white2c}
\end{align}

\noi
for any finite $r \ge 1$.

\begin{lemma}\label{LEM:white}
Let $\z$ be a white noise on $\R^2$
as in Definition \ref{DEF:white}
and set $\mu_{0, \infty} = \Law(\z)$.
Then, given any $s < -1$, finite $p \ge 1$,  and $\mu > 0$, 
by viewing  the $L$-periodic white noise measure $\mu_{0, L}$
defined in \eqref{mu0} with $s =0$
as a measure  on $W^{s, p}_\mu(\R^2)$, 
the sequence $\{\mu_{0, L}\}_{L \in \N}$
converges weakly 
to $\mu_{0, \infty}$
as $L \to \infty$.

\end{lemma}

By putting the weak convergence of a subsequence $\{\rho_{L_j}\}_{j \in \N}$
with Lemma \ref{LEM:white}, 
we conclude that
as probability measures on $\vec W^{s, p}_\mu(\R^2)
=  W^{s, p}_\mu(\R^2)\times  W^{s-1, p}_\mu(\R^2)$
with $s < 0$, finite $p \ge 1$, and $\mu > 0$, 
the $L_j$-periodic Gibbs measure $\rhoo_{L_j} = \rho_{L_j} \otimes \mu_{0, L_j}$
in \eqref{Gibbs3}
converges weakly to a limiting
Gibbs measure
$\rhoo_\infty = \rho_\infty \otimes \mu_{0, \infty}$
on $\R^2$
as $j \to \infty$.
This proves Theorem \ref{THM:1}\,(i).

\begin{remark}\label{REM:Gibbs2}\rm
In view of the embedding
between weighted Sobolev spaces and weighted Besov spaces
stated in Lemma \ref{LEM:equiv}, 
the weak convergence of 
the $L_j$-periodic Gibbs measure $\rhoo_{L_j} $ to 
the limiting Gibbs measure $\rhoo_\infty$
also holds
as probability measures  
on the weighted Besov space $B^{s, \mu}_{p, q}(\R^2) \times B^{s-1, \mu}_{p, q}(\R^2)$
for any $s <0$, finite $p \ge 1$, $1 \le q \le \infty$, 
and $\mu > 0$.

\end{remark}

We conclude this section by presenting the proof of Lemma \ref{LEM:white}.

\begin{proof}[Proof of Lemma \ref{LEM:white}]
As we have seen in Section \ref{SEC:1}, 
an $L$-periodic white noise  on 
the dilated torus $\T_L^2$ is given by 
the second Gaussian Fourier series for $v_L$ in \eqref{series}
such that $\Law(v_L) =  \mu_{0, L}$.
While it is possible to work with $v_L$ in \eqref{series}, 
by viewing it as an $L$-periodic distribution on $\R^2$, 
and show that 
the $L$-periodic white noise measure $\mu_{0, L} = \Law (v_L)$
converges weakly to the white noise measure $\mu_{0, \infty} = \Law(\z)$ on $\R^2$, 
it is more convenient to work with 
the $L$-periodized version $\z_L$ of the white noise $\z$.
Define $\z_L$ on $\T^2_L$ by setting
\begin{align}
\z_L=  \sum_{\ze \in \Z^2_L} \ft \z_L (\ze) e_\ze^L, 
\label{wh1}
\end{align}

\noi
where $e_\ze^L$ is as in \eqref{exp1}
and $\ft \z_L(\ze)$ is given by 
\begin{align}
\begin{split}
\ft \z_L(\ze) & = \int_{[-\frac L2, \frac L2)^2} \z(x) \cj{e_\ze^L(x)} dx 
= \z \Big( \ind_{[-\frac L2, \frac L2)^2}  \cj{e_\ze^L}\Big)\\
: \! &  = \z \Big( \ind_{[-\frac L2, \frac L2)^2}  \Re (e_\ze^L)\Big)
- i  \z \Big( \ind_{[-\frac L2, \frac L2)^2}  \Im (e_\ze^L)\Big).
\end{split}
\label{wh2}
\end{align}

\noi
Then, by viewing $\z_L$ as an $L$-periodic distribution on $\R^2$, 
we see that $\Law(\z_L) = \mu_{0,  L}$.
Indeed,  it follows from Definition \ref{DEF:white}
and the orthonormality of $\{e_\ze^L\}_{\ze \in \Z^2_L}$
on $\T^2_L = \big[-\frac L2, \frac L2\big)^2$
that $\{ \ft \z_L(\ze)\}_{\ze \in \Z^2_L}$
forms
 a family of independent standard complex-valued  Gaussian random variables conditioned  that 
 $\ft \z_L(-\ze) = \cj{\ft \z_L(\ze)}$, 
 $\ze  \in \Z^2_L$.
Hence, from \eqref{series} and \eqref{wh1}, 
we conclude that $\Law(\z_L) = \Law(v_L) = \mu_{0, L}$.
Before proceeding further, we point out that 
by repeating the computation \eqref{white2a}
for $\z_L$, 
we see that the bounds \eqref{white2b} and \eqref{white2c}
also hold for $\z_L$, uniformly in $L \ge 1$.

By definition,  $\z_L$ agrees with $\z$ on $\T^2_L = \big[-\frac L2, \frac L2\big)^2$.
Now, fix $s < -1$, $1 \le p < \infty$, and $\mu > 0$.
Then, given finite $r \ge 1$, it follows from \eqref{white2c}
for $\z$ and $\z_L$ that 
\begin{align}
\begin{split}
\E\Big[
\|\z - \z_L \|_{W^{s, p}_\mu(\R^2)}^p\big]
& = \E\Big[\|\z - \z_L \|_{W^{s, p}_\mu((\T^2_L)^c)}^p\big]\\
& \le C(r, s, \mu', p) \cdot  w_{\mu - \mu'}(c_0L)
\end{split}
\label{wh3}
\end{align}

\noi
for any $0 < \mu' < \mu$, 
where $c_0>0$ is an absolute constant independent of $L \ge 1$
(and the other parameters).
Then, summing over $L \in \N$, 
we obtain
\begin{align*}
 \E\bigg( \sum_{L=  1}^\infty \| \zeta - \zeta_L \|^p_{W^{s,p}_\mu(\R^2)} \bigg) 
&
= \sum_{L =  1}^\infty \E\Big[ \| \zeta - \zeta_L \|^p_{W^{s,p}_\mu(\R^2)} \Big] \\
&  \les  \sum_{L=  1}^\infty w_{\mu - \mu'} (c_0 L) < \infty, 
\end{align*}
since $\mu> \mu'$. 
This implies immediately that 
$\sum_{L =  1}^\infty \| \zeta - \zeta_L \|^p_{W^{s,p}_\mu(\R^2)}$ is finite almost surely, 
which in particular implies that 
 $\z_L$ converges almost surely to $\z$ in $W^{s,p}_\mu(\R^2)$ as $L\to \infty$
 (with $L \in \N$).
Therefore, we conclude that 
$\mu_{0, L} = \Law (\z_L)$ converges weakly to $\mu_{0, \infty} = \Law(\z)$
as $L \to \infty$ (with $L \in \N$).
This conclude the proof of Lemma \ref{LEM:white}.
\end{proof}

\section{Global well-posedness and invariance
of the Gibbs measure}
\label{SEC:GWP}

In this section, we present the proof of Theorem \ref{THM:1}\,(ii).
Our main goal is to construct global-in-time dynamics
for the hyperbolic $\Phi^{k+1}_2$-model on $\R^2$
with the Gibbsian initial data:
\begin{align}
\begin{cases}
\dt^2 u + \dt u + (1- \Dl)u  \, +  :\!u^k\!: \, = \sqrt{2}\xi\\
(u, \dt u)|_{t = 0} = (u_0, u_1)\quad \text{with } \Law (u_0, u_1) = \rhoo_\infty, 
\end{cases}
\label{NW1}
\end{align}

\noi
where $\rhoo_\infty = \rho_\infty \otimes \mu_{0, \infty}$ is the Gibbs measure on $\R^2$,  constructed
as a limit of the $L_j$-periodic Gibbs measure $\rhoo_{L_j} = \rho_{L_j} \otimes \mu_{0, L_j}$
in the previous section.
For simplicity of notation, 
we set 
\begin{align*}
\A = \{ L_j: j \in \N\} \subset \N
\end{align*}
and only consider the values of $L \in \A$
in the following, unless otherwise specified.

Consider the $L$-periodic hyperbolic $\Phi^{k+1}_2$-model on $\R^2$:
\begin{align}
\begin{cases}
\dt^2 u_L + \dt u_L + (1- \Dl)u_L  \, +  :\!u_L^k\!: \, = \sqrt{2}\xi_L\\
(u_L, \dt u_L)|_{t = 0} = (u_{0, L}, u_{1, L})\quad \text{with } \Law (u_{0, L}, u_{1, L}) = \rhoo_L,
\end{cases}
\label{NW2}
\end{align}

\noi
where $\xi_L$ is an $L$-periodic space-time white noise
(see \eqref{wh4} and \eqref{wh5} below) 
and $\rhoo_L = \rho_L \otimes \mu_{0, L}$
is the $L$-periodic Gibbs measure in \eqref{Gibbs3}; see also Remark \ref{REM:Gibbs}.
As mentioned in Section~\ref{SEC:1},  
the Cauchy problem
\eqref{NW2} is globally well-posed
and the Gibbs measure $\rhoo_L$ is invariant under the resulting dynamics.
In the following, we 
 construct
a solution to \eqref{NW1}
as a limit of the solutions $u_L$ to \eqref{NW2}, $L \in \A$.

Let us first introduce some notations.
Let $\Psi$ be the solution to the following linear stochastic damped wave equation:
\begin{align}
\begin{cases}
\dt^2 \Psi + \dt \Psi + (1- \Dl)\Psi    = \sqrt{2}\xi \\
(\Psi, \dt \Psi)|_{t=0} = ( 0, 0).
\end{cases}
\label{psi1}
\end{align}

\noi
With $\D(t)$ as in \eqref{lin2}, we formally have
\begin{align*}
\Psi(t) =& \sqrt{2} \int_0^t \D(t-t') \xi (dt').
\end{align*}

Let $L \ge 1$.
Given a spatial white noise $\z$ on $\R^2$, 
we let $\z_L$ denote the 
$L$-periodized version of $\z$ defined in  \eqref{wh1} and \eqref{wh2}.
Similarly, let $\xi_L$
be 
the $L$-periodized version (in space) of the space-time white noise $\xi$.
Namely,  we define   $\xi_L$ on $\T^2_L\times \R_+
\cong \big[-\frac L2, \frac L2\big)^2 \times \R_+$ by setting
\begin{align}
\xi_L=  \sum_{\ze \in \Z^2_L} \ft \xi_L (\ze) e_\ze^L, 
\label{wh4}
\end{align}

\noi
where $e_\ze^L$ is as in \eqref{exp1}
and $\ft \xi_L(\ze)$ is 
 given by 
\begin{align}
\ft \xi_L(\ze) = \int_{[-\frac L2, \frac L2)^2} \xi(x) \cj{e_\ze^L(x)} dx 
= \xi \Big(  \ind_{[-\frac L2, \frac L2)^2} \,  \cj{e_\ze^L}\Big)
\label{wh5}
\end{align}

\noi
with the equality understood in the sense of temporal distributions.
Alternatively, $\ft \xi_L(\ze)$ is 
given as the (distributional) time derivative of 
\begin{align*}
\ft B_L(\ze, t) 
= \xi \Big( \ind_{[0, t]} \cdot \ind_{[-\frac L2, \frac L2)^2}(x) \,  \cj{e_\ze^L(x)}\Big), 
\end{align*}

\noi
where the right-hand side is the action of $\xi$ 
on a function in $L^2(\R^2 \times \R_+)$.
It is then easy to check that  $\{\ft \xi_L(\ld) \}_{\ld \in \Z^2_L}$ forms
a family of independent temporal white noises
conditioned that $\ft \xi_L(- \ld) = \cj{\ft \xi_L(\ld)}$, $\ld \in \Z^2_L$.
Compare this with~\eqref{WW1}.  
By the $L$-periodic extension in space, 
we then view $\xi_L$  as a space-time distribution on  $\R^2\times \R_+$.

\medskip

In the following, 
we  construct a solution $u$ to \eqref{NW1}
on the cone $\bC_R$ defined in \eqref{cone}
for each $R >0$
as a limit of the solution $u_L$ to \eqref{NW2}.
Fix $R >0$.
Then, thanks to the finite speed of propagation, 
we have
\begin{align}
\z_L = \z \quad \text{and} \quad \xi_L = \xi
\quad \text{on the cone $\bC_R$}, 
\label{wh7}
\end{align}

\noi
provided that  $L \ge 2R$.
In particular,  on the cone $\bC_R$, the equation \eqref{NW2}
agrees (in law) with 
\begin{align}
\begin{cases}
\dt^2 u_L + \dt u_L + (1- \Dl)u_L  \, +  :\!u_L^k\!: \, = \sqrt{2}\xi\\
(u_L, \dt u_L)|_{t = 0} = (u_{0, L}, u_{1})\quad \text{with } \Law (u_{0, L}) = \rho_L,
\end{cases}
\label{NW3}
\end{align}

\noi
where $\xi$ and $u_1$ with $\Law (u_1) = \mu_{0, \infty}$
are as in \eqref{NW1}.
Namely, the difference between \eqref{NW1}
and~\eqref{NW3} appears only in 
the first component of the initial data.
This motivates the following decomposition of a solution.

We write a  solution $u$ to \eqref{NW1} as 
\begin{align}
\begin{split}
u(t)&  = 
\Phi(t)  + v (t)\\
: \!  & = \S(t)u_0 + \D(t) u_1 + \Psi(t) + v (t),
\end{split}
\label{decomp2}
\end{align}

\noi
where $\S(t)$, $\D(t)$, and $\Psi(t)$ are as in \eqref{lin3}, 
\eqref{lin2}, and \eqref{psi1}, 
and $v$ satisfies the following equation:
\begin{align}
\begin{split}
v(t) & = - \int_0^t \D(t-t') : \!u^k(t')\! : dt' \\
& = - \sum_{\l = 0}^k 
{k\choose \ell}
\int_0^t \D(t - t')\big( \!:\! \Phi^\l(t')\!:
v^{k- \l}(t')\big) dt'.
\end{split}
\label{NW4}
\end{align}

\noi
We also set
\begin{align}
\Phi_0(t) = \Phi_0(u_0)(t)  =  \S(t) u_0
\quad \text{and}
\quad  \Phi_1(t) 
= \Phi_1(u_1, \xi) (t)
 = \D(t)u_1 + \Psi(t)
 \label{NW5}
\end{align}

\noi
such that 
\[\Phi(u_0, u_1, \xi) = \Phi_0(u_0) + \Phi_1(u_1, \xi).\] 

\noi
Then, by the independence of $u_0$ and $(u_1, \xi)$ in \eqref{NW1},\footnote{More precisely, 
recall that the product of independent homogenous Wiener chaoses is a homogenous Wiener chaos 
of the order given by the sum of the two.
} 
we have 
\begin{align*}
:\! \Phi^\l(t)\!:
\,\,  = \,\,  :\! \Phi^\l(u_0, u_1, \xi)(t)\!:
\,\,  = \sum_{m = 0}^\l
{\l \choose m}
:\! \Phi_0^m(u_0) (t)\!:
\, :\! \Phi_1^{\l-m}(u_1, \xi) (t)\!:\, ,
\end{align*}

\noi
and thus
we can rewrite \eqref{NW4} as 
\begin{align}
\begin{split}
v(t) 
& = -  \sum_{\l = 0}^k \sum_{m = 0}^\l
{k\choose \ell}
{\l \choose m}\\
& \hphantom{XXX}
\times \int_0^t \D(t - t')\big( \!
:\! \Phi_0^m(u_0) (t')\!:
\, :\! \Phi_1^{\l-m}(u_1, \xi) (t')\!:
v^{k- \l}(t')\big) dt'.
\end{split}
\label{NW5b}
\end{align}

\noi
In view of \eqref{NW5b}, 
we  define enhanced data sets $\Xi_0$ and $\Xi_1$
by setting 
\begin{align}
\begin{split}
\Xi_0(\phi) 
& = \big( \Phi_0(\phi),  : \!\Phi_0^2(\phi)\! :\,, \dots, \, : \!\Phi_0^k(\phi)\!: \big)\\
& = \big( \S(t)\phi,  : \!(\S(t) \phi)^2\! :\,, \dots, \, : \!(\S(t) \phi)^k\!: \big)
\end{split}
\label{xi1}
\end{align}

\noi
and
\begin{align}
\Xi_1(u_1, \xi) = \big(   \Phi_1, : \!\Phi_1^2\!:\,, \dots, 
\, : \!\Phi_1^k\!: \big), 
\label{xi2}
\end{align}

\noi
where
$ \Phi_1
= \Phi_1(u_1, \xi) $ is as in \eqref{NW5}.

We now introduce  distances\,/\,sizes by which we measure these enhanced data sets.

\begin{definition}\label{DEF:X}\rm
Let $k \in 2\N + 1$.
Let $0 < \eps \ll1$, finite $p \gg1$, and $\mu > 0$
(to be chosen later).

%

\smallskip

\noi
(i) 
We first introduce a  space $\WW$
for 
the first enhanced data set $\Xi_0(\phi)$
in \eqref{xi1}
by setting
\begin{align*}
\WW = \big(C(\R_+;  W^{-\eps, p}_\mu(\R^2))\big)^{\otimes k}.
\end{align*}

\noi
We endow $\WW$ with the following bounded metric:
\begin{align}
\begin{split}
d_\WW(f, g) 
& = d_\WW\big((f_1, \dots, f_k), (g_1, \dots, g_k)\big) \\
& =  \sum_{\l=1}^\infty 
\frac{1}{2^\l}
\frac{ \sum_{j =1} ^k\| f_j-g_j\|_{C  ([ 0, 2^\l]; W^{-\eps, p}_\mu(\R^2))}}
{ 1 + \sum_{j =1} ^k \| f_j-g_j\|_{C  ([ 0, 2^\l]; W^{-\eps, p}_\mu(\R^2))}}, 
\end{split}
\label{norm1}
\end{align}

\noi
which induces the compact-open topology (in time)
on $\WW$.
Then,  the metric space 
$ ( \WW,  d_\WW)$ is a bounded Polish space.\footnote{Recall that 
 the space of continuous functions from 
a  separable metric space 
 $X$ to another separable metric space $Y$ 
 with the compact-open  topology is separable; see \cite{Mi}.
See also the paper \cite[Corollary 3.3]{Khan86}. }

Given $R > 0$, 
we also set
\begin{align}
\WW(R) = \big(L^\infty([0; R];  W^{-\eps, p}(B_R))\big)^{\otimes k}
\label{norm2}
\end{align}

\noi
with the norm given by 
\begin{align*}
\| f \|_{\WW(R)}
= \| (f_1, \dots, f_k) \|_{\WW(R)}
= 
\sum_{j = 1}^k \|f_j\|_{L^\infty([0; R];  W^{-\eps, p}(B_R))}.
\end{align*}

\smallskip

\noi
(ii) 
Let $\Xi_0 = (\Xi_{01}, \dots, \Xi_{0k}) \in \WW(R)$.
Given $R > 0$, let $\|\Xi_1 \|_{\Xi_0, R}$ denote the smallest  constant $K_1 \geq 1$
such that  the following inequality holds  for $1\leq j,\ell \leq k$
with $j + \l \le k$:
\begin{align}
\begin{split}
\int_0^T \|  :\!\Phi_1^j(t)\!:   \Xi_{0\l}(t) \|_{W^{-2(k+1)\eps, p}(B_{R})}^p dt
\leq K_1\int_0^T \| \Xi_{0\l}(t) \|_{W^{-\eps, p}(B_{R})}^p dt 
\end{split}
\label{TA1}
\end{align}

\noi
for any $0 < T \le R$.

\medskip
\noi
(iii) Let 
$\wt \Xi_0 = (\wt \Xi_{01}, \dots, \wt \Xi_{0k}) \in \WW(R)$.
Given $R > 0$, let  $K_2,  K_3 \ge 1$  be the  smallest  constants
such that  the following inequalities hold
for $1\leq j,\ell \leq k$
with $j + \l \le k$:
\begin{align}
\begin{aligned}
\int_0^T  \|  :\!\Phi_1^j(t)\! :   \wt \Xi_{0\l}(t) \|_{W^{-2(k+1)\eps, p}(B_{R})}^p dt
\leq K_2\int_0^T \|  \wt  \Xi_{0\l}(t) \|_{W^{-\eps, p}(B_{R})}^p dt
\end{aligned}
\label{TA2}
\end{align}
and

\noi
\begin{align}
\begin{aligned}
&\int_0^T \| :\!\Phi_1^j(t)\! :  (\,  \Xi_{0\l}(t) -  \wt \Xi_{0\l}(t) )
   \|_{W^{-2(k+1)\eps, p}(B_{R})}^p dt
\\
&\qquad\qquad\qquad
\leq K_3\int_0^T \|  \Xi_{0\l}(t) -  \wt \Xi_{0\l}(t)  \|_{W^{-\eps, p}(B_{R})}^p dt  
\end{aligned}
\label{TA3}
\end{align}

\noi
for any $0 < T \le R$.
Then, we set 
$\|\Xi_1 \|_{\Xi_0, \wt{\Xi}_0, R} = \max(K_1, K_2, K_3)$, where $K_1$ is 
 as in Part (ii).

\end{definition}

In the remaining part of this paper, we fix small $\eps > 0$, finite $p = p(\eps) \gg 1$,\footnote{The condition 
$p = p(\eps) \gg 1$ is needed in Propositions \ref{PROP:LWP}, \ref{PROP:stab},
and \ref{PROP:need23}.}
 and $\mu > 0$
(to be chosen later), 
and often drop  dependence on these parameters.
For simplicity of notation, we set
\begin{align}
\eps_k = 2(k+1) \eps
\label{eps1}
\end{align}

\noi
in the remaining part of this section.
This $\eps_k$  comes from the condition 
appearing in the proof of   Proposition \ref{PROP:need23}\,(i).

\subsection{Local well-posedness and stability of SdNLW}
\label{SUBSEC:4.1}
In this subsection, 
we study 
local well-posedness and stability of SdNLW \eqref{NW1}, 
where we view 
$u_0$, $u_1$, and $\xi$ in \eqref{NW1} as given deterministic spatial\,/\,space-time distributions\footnote{As
for $\xi$, we should really view $\Psi = \Psi(\xi)$ as a given space-time distribution.}
such that the enhanced data sets $\Xi_0(u_0)$ and $\Xi_1(u_1, \xi)$
in \eqref{xi1} and \eqref{xi2} make sense.
In order to emphasize the dependence on $u_0$, $u_1$, and $\xi$, 
we may write  $u = u(u_0, u_1, \xi)$.
By writing $u$ as in \eqref{decomp2}, 
we study the equation~\eqref{NW5b}
for the remainder term 
\begin{align}
\begin{split}
v 
& =u - \S(t)u_0 - \D(t) u_1 - \Psi\\
& = u - \Phi_0(u_0) - \Phi_1(u_1, \xi).
\end{split}
\label{YY0}
\end{align}

Given $\Xi_0 = (\Xi_{01}, \dots, \Xi_{0k})$, 
we  consider \eqref{NW5b}, starting from $t = t_0$
with $(v, \dt v)|_{t = t_0} = (v_0, v_1)$, 
where
we replace $:\! \Phi_0^m(u_0) \!:$ by $\Xi_{0m}$.:
\begin{align}
\begin{split}
v(t) 
& = \S(t-t_0) v_0 + \D(t- t_0) v_1\\
& \quad
-  \sum_{\l = 0}^k \sum_{m = 0}^\l
{k\choose \ell}
{\l \choose m}\\
& \hphantom{XXX}
\times  \int_{t_0}^t \D(t - t')\big( 
\Xi_{0m} (t')
 :\! \Phi_1^\l(u_1, \xi) (t')\!:
v^{k- \l}(t')\big) dt'.
\end{split}
\label{NW5bb}
\end{align}

\noi
Given  $s\in \R$ and $\tau > 0$, 
we set 
\begin{align*}
 X_{t_0, R}^s(\tau)
& = L^\infty([t_0, \max(t_0+ \tau, R)]; 
H^{s}(B_{R - t}) ),\\
 \vec X_{t_0, R}^s(\tau)
& = X_{t_0, R}^s(\tau) \times X_{t_0, R}^{s-1}(\tau).
\end{align*}

\begin{proposition}[local well-posedness]\label{PROP:LWP}
Let 
$R > 0$, $0 \le t_0 < R$, and $(v_0, v_1) \in \vec H^{1-\eps_k}(B_{R- t_0})$, 
where $\eps_k = 2(k+1)\eps$ is as in \eqref{eps1}.
Suppose that there exist $M_0, M_1 \ge1$ such that
\begin{align}
\begin{split}
\| \Xi_0 \|_{\WW(R) }
 \le M_0 <\infty,\\
\| \Xi_1(u_1, \xi) \|_{\Xi_0, R} \leq M_1 < \infty, 
\end{split}
\label{NW66}
\end{align}

\noi
where 
$\WW(R)$ is as in \eqref{norm2}
and 
$\| \Xi_1(u_1, \xi) \|_{\Xi_0, R}$ is as in Definition \ref{DEF:X}\,(ii).
Then, there exist small $\tau = \tau\big(\|(v_0, v_1)\|_{\vec H^{1-\eps_k}(B_{R- t_0})}, M_0, M_1\big)>0$ 
and a unique solution $v$ to~\eqref{NW5bb}
on $\bC_R\cap \big\{ t_0\leq  t \leq \min(t_0 + \tau, R) \big\}$
such that 
\begin{align}
\| (v, \dt v) \|_{ \vec X_{t_0, R}^{1-\eps_k}(\tau)}
 \le C_0 \| (v_0, v_1)\|_{\vec H^{1-\eps_k}(B_{R- t_0})}
\label{NW67}
\end{align}

\noi
for some absolute constant $C_0 > 0$.

\end{proposition}

\begin{proof}
The proof of Proposition \ref{PROP:LWP}
follows from a slight modification of Proposition~4.1 in~\cite{GKOT}
together with Definition \ref{DEF:X}\,(ii).
Denote the right-hand side of \eqref{NW5bb} by 
$\G(v) = \G_{\Xi_0, \Xi_1(u_1, \xi)}(v)$, 
and  set 
$\vec \G(v) = ( \G(v), \dt \G(v))$.
Then, by the finite speed of propagation, 
we have
\begin{align}
\begin{split}
\G(v)(t) 
& =\S(t-t_0) v_0 + \D(t- t_0) v_1 -  \sum_{\l = 0}^k \sum_{m = 0}^\l
{k\choose \ell}
{\l \choose m}\\
& \hphantom{X}
\times \int_{t_0}^t \D(t - t')\big(\ind_{\bC_R}
\cdot \Xi_{0m} (t')
 :\! \Phi_1^{\l- m}(u_1, \xi) (t')\!:
v^{k- \l}(t')\big) dt'
\end{split}
\label{NW5c}
\end{align}

\noi
on the cone $\bC_R$.

Let $(\vv_0, \vv_1)$ be an extension of 
$(v_0, v_1)$ onto $\R^2$.
Given $\vec v = (v, \dt v)$ on  
$\big\{(x, t) \in \bC_R : t \in  [t_0, \max(t_0+ \tau, R)]\big\}$, 
let $\vec \vv = (\vv, \vv')$ be an extension onto $\R^2\times \R_+$.
Similarly, let $\Phi^{m, \l-m}_{0, 1}$
be an extension of 
$ \Xi_{0m} 
 :\! \Phi_1^{\l- m} \!:$ from $B_R \times [0, R]$ to $\R^2\times [0, R]$.
Then, proceeding as in the proof of 
Proposition~4.1 in~\cite{GKOT}, 
it follows 
from \eqref{NW5c},  Lemma~\ref{LEM:bilin1}, and Sobolev's inequality  that
\begin{align*}
 \|\vec \G(v)\|_{\vec X_{t_0, R}^{1-\eps_k}(\tau)}
&  \les \| (\vv_0, \vv_1)\|_{\vec H^{1-\eps_k}(\R^2)} \\
 & \quad  + \tau^\ta  \sum_{\l = 0}^k \sum_{m = 0}^\l
c_{k, \l, m}
\| \Phi^{m, \l-m}_{0, 1}\|_{L^p_{[t_0, t_0+\tau]} W^{-\eps_k, p}_x(\R^2)}
\| \vv\|_{L^\infty_{[t_0, t_0+\tau]}  H^{1-\eps_k}_x(\R^2) }^{k-\l}
\end{align*}

\noi
for some $\ta > 0$ and 
large $p \gg 1$, 
where $\eps_k =  2(k+1)\eps$ is as in \eqref{eps1}
with sufficiently small $\eps > 0$.
Then, by taking infima over all the extensions
$(\vv_0, \vv_1)$, $\vec \vv$,  and $\Phi^{m, \l-m}_{0, 1}$
and then using 
Definition \ref{DEF:X}\,(ii) with \eqref{NW66}, we obtain
\begin{align}
\begin{split}
 \|\vec \G(v)\|_{\vec X_{t_0, R}^{1-\eps_k}(\tau)}
& \les 
\| (v_0,v_1)\|_{\vec H^{1-\eps_k}(B_{R- t_0})}\\
& \quad 
+ 
\tau^\ta M_1 \sum_{\l = 0}^k \sum_{m = 0}^\l
 c_{k, \l, m}
\| \Xi_{0m} \|_{L^p_\tau W^{-\eps, p}_x(B_{R })}
\| \vec v\|_{\vec X_{t_0, R}^{1-\eps_k}(\tau) }^{k-\l}\\
& 
\les\| (v_0,v_1)\|_{\vec H^{1-\eps_k}(B_{R- t_0})}
 + \tau^\ta M_0  M_1 
\Big(1 + \| \vec v\|_{X_{t_0, R}^{1-\eps_k}(\tau)}\Big)^{k}.
\end{split}
\label{NW6}
\end{align}

\noi
A similar argument yields the following difference estimate:
\begin{align}
\begin{split}
&  \|\vec \G(v_1) - \vec \G(v_2)\|_{\vec X^{1-\eps_k}_{t_0, R}(\tau)}\\
& \quad \les \tau^\ta M_0  M_1 
\Big(1 + \| \vec v_1\|_{X_{t_0, R}^{1-\eps_k}(\tau) }^{k-1}+ 
\| \vec v_2\|_{\vec X_{t_0, R}^{1-\eps_k}(\tau)}^{k-1}\Big)
\| \vec v_1 - \vec v_2\|_{\vec X_{t_0, R}^{1-\eps_k}(\tau)}, 
\end{split}
\label{NW7}
\end{align}

\noi
where $\vec v_j = (v_j, \dt v_j)$, $j = 1, 2$.
Hence, by taking $\tau = \tau\big(\|(v_0, v_1)\|_{\vec H^{1-\eps_k}(B_{R- t_0})}, M_0, M_1\big)> 0$
sufficiently small, 
the desired claim follows from a standard contraction argument
with \eqref{NW6} and \eqref{NW7}.
\end{proof}

\begin{proposition}[stability]\label{PROP:stab}
Suppose that
$\Xi_0$, $u_1$, and $\xi$ satisfy \eqref{NW66}
and that 
  $v = v(\Xi_0, u_1, \xi)$  
is 
a solution to~\eqref{NW5bb} \textup{(}with $t_0 = 0$\textup{)}
on the cone $\bC_R$ such that
\begin{align}
\| \vec v\|_{L^\infty([0, R]; \vec H^{1-\eps_k}(B_{R - t}))}
&\leq  M_2
\label{YY1}
\end{align}

\noi
for some $M_2 \ge 1$.
Then,  there exists  $\dl_* = \dl_*(R, M_0, M_1, M_2) > 0$ such that
for any $\wt{\Xi}_0$ satisfying
\begin{align}
\begin{split}
\| \Xi_0 - \wt \Xi_0 \|_{\WW(R)}&  <\dl_*,
\\
\| \Xi_1(u_1, \xi)\|_{\Xi_0, \wt \Xi_0,  R} \leq M_1 & < \infty, 
\end{split}
\label{YY2}
\end{align}

\noi
where 
$\wt \Xi_0 = (\wt \Xi_{01}, \dots, \wt \Xi_{0k})$
and 
$\WW(R)$ is as in \eqref{norm2}, 
 there exists 
$\wt v = \wt v(\wt \Xi_0, u_1, \xi)
\in L^\infty([0,  R]; 
H^{s}(B_{R - t}) )$,  
satisfying 
\begin{align}
\begin{split}
\wt v(t) 
& = - \sum_{\l = 0}^k \sum_{m = 0}^\l
{k\choose \ell}
{\l \choose m}\\
& \hphantom{XXX}
\times \int_0^t \D(t - t')\big( 
\wt \Xi_{0m}(t')
:\! \Phi_1^{\l-m} (t')\!:
\wt v^{k- \l}(t')\big) dt'
\end{split}
\label{YY5a}
\end{align}

\noi
such that
\begin{align}
\begin{split}
 \sup_{0 \le t \le R} & \|  \vec v(t) - \vec {\wt{v}}(t) \|_{\vec H^{1-\eps_k}(B_{R-t})}\\
& \leq C_0(R, M_0, M_1, M_2) \| \Xi_0 - \wt \Xi_0\|_{\WW(R)}, 
\end{split}
\label{YY3}
\end{align}

\noi
where $\vec v = (v, \dt v)$, $\vec {\wt v} = (\wt v, \dt \wt v)$, and $\eps_k = 2(k+1) \eps$ is as in \eqref{eps1}.

\end{proposition}

\begin{proof}

Note that Proposition \ref{PROP:LWP} with \eqref{NW66} and \eqref{YY2}
guarantees existence of a solution 
$\wt v = \wt v(\wt \Xi_0, u_1, \xi)$
to~\eqref{YY5a} 
on a short time interval.
Namely, it
 satisfies
\begin{align}
\begin{split}
\wt v(t) 
& = - \sum_{\l = 0}^k \sum_{m = 0}^\l
{k\choose \ell}
{\l \choose m}\\
& \hphantom{XXX}
\times \int_0^t \D(t - t')\big( 
\ind_{\bC_R}
\cdot 
\wt \Xi_{0m}(t')
:\! \Phi_1^{\l-m} (t')\!:
\wt v^{k- \l}(t')\big) dt'
\end{split}
\label{YY5}
\end{align}

\noi
on $\bC_R\cap \{0 \le t \le \tau\}$
for some $\tau > 0$, 
where 
%
$\Phi_1 = \Phi_1(u_1, \xi)$
 is as in \eqref{NW5}.
 From \eqref{NW67}, we
 have
 \begin{align}
 \|\vec{\wt  v }\|_{L^\infty_\tau \vec H^{1-\eps_k}(B_{R-t})}\les 1.
 \label{YY6}
\end{align}

\noi
As for $v$, 
it satisfies \eqref{NW5bb} (with $t_0 = 0$)
 on $\bC_R$ (with an additional cutoff function $\ind_{\bC_R}$
as in~\eqref{YY5}).

Given an interval $I  \subset \R_+$, let 
\begin{align}
B_I\big(\vec v, \vec{\wt v}\,\big) = \|\vec v \|_{L^\infty(I;  \vec H^{1-\eps_k}(B_{R-t}))}
 + 
 \| \vec{\wt v} \|_{L^\infty(I;  \vec H^{1-\eps_k}(B_{R-t}))}
 \label{YY7}
\end{align}

\noi
Then, 
it follows 
from \eqref{YY1} and \eqref{YY6} that 
\begin{align}
B_{[0, \tau]}\big(\vec v, \vec{\wt v}\,\big)
 \le   M_2 + C_0' \les M_2
 \label{YY7x}
\end{align}

\noi
for $I = [0, \tau]$.
Then, proceeding as in the proof of Proposition \ref{PROP:LWP}
with 
 \eqref{NW5bb}, \eqref{YY5},  Lemma~\ref{LEM:bilin1}, and Sobolev's inequality
 with \eqref{NW66}, \eqref{YY2}, and \eqref{YY7x}
 we have
\begin{align*}
& \|\vec v - \vec{\wt v} \|_{L^\infty_{\tau_0} \vec H^{1-\eps_k}(B_{R-t})}\\
& \quad  \les 
\tau_0^\ta M_1 \sum_{\l = 0}^k \sum_{m = 0}^\l
 c_{k, \l, m}
B_{[0, \tau]}\big(\vec v, \vec{\wt v}\, \big)^{k - \l-1}
 \\
& \hphantom{XXXX}
 \times \Big\{ 
 B_{[0, \tau]}\big(\vec v, \vec{\wt v}\,\big)
\| \Xi_0 - \wt \Xi_0 \|_{\WW(R)}
+ M_0 
\|\vec v - \vec{\wt v} \|_{L^\infty_{\tau_0} \vec H^{1-\eps_k}(B_{R-t})}\Big\}\\
& \quad
\les \tau_0^\ta M_1 M_2^{k-1} 
\Big\{M_2 \| \Xi_0 - \wt \Xi_0 \|_{\WW(R)}
+ M_0 
\|\vec v - \vec{\wt v} \|_{L^\infty_{\tau_0} \vec H^{1-\eps_k}(B_{R-t})}\Big\}
\end{align*}

\noi
for any $0 < \tau_0 \le \tau$.
Hence, by taking $\tau_0 = \tau_0(M_0, M_1, M_2) > 0$ sufficiently small, 
we obtain
\begin{align*}
 \|\vec v - \vec{\wt v} \|_{L^\infty_{\tau_0} \vec H^{1-\eps_k}(B_{R-t})}
 \le C_1 
\| \Xi_0 - \wt \Xi_0 \|_{\WW(R)}.
\end{align*}

By repeating the computation on the second interval $I_2 = [\tau_0, 2\tau_0)$, 
using \eqref{NW5bb} with $t_0 = \tau_0$ and an analogous formulation for $\wt v$, 
we have
\begin{align*}
&  \|\vec v - \vec{\wt v} \|_{L^\infty(I_2; \vec H^{1-\eps_k}(B_{R-t }))}
 \le C_1 (1+ C_2)
\| \Xi_0 - \wt \Xi_0 \|_{\WW(R)}, 
%
\end{align*}

\noi
provided that $\dl_* = \dl_*(M_2) > 0$ in \eqref{YY2}  is sufficiently small that 
\begin{align}
B_{I}\big(\vec v, \vec{\wt v}\,\big)
\les  M_2
 \label{YY7a}
\end{align}

\noi
holds for $I = I_j$, $j = 1, 2$, 
where $B_I\big(\vec v, \vec{\wt v}\,\big)$ is as in \eqref{YY7}.
Then, by iterating the computation on the $j$th interval $I_j = [(j-1) \tau_0, j\tau_0)$, 
we obtain
\begin{align*}
  \|\vec v - \vec{\wt v} \|_{L^\infty(I_j ; \vec H^{1-\eps_k}(B_{R-t}))}
 \le C_1 \sum_{i = 0}^{j - 1}C_2^i \, 
\| \Xi_0 - \wt \Xi_0 \|_{\WW(R)}, 
\end{align*}

\noi
provided that $\dl_* = \dl_*(j,  M_2)  > 0$ in \eqref{YY2} is sufficiently small that 
\eqref{YY7a} holds for $I = I_i$, $i = 1, \dots, j$.

Note that our choice of $\tau_0 = \tau_0(M_0, M_1, M_2) $ does not depend on $\dl_*$
in \eqref{YY2}.
Thus, by iterating this argument $\sim \frac{R}{\tau_0}$ many times
(which requires us to choose 
 $\dl_* = \dl_*\big(\frac{R}{\tau_0},  M_2\big)  
 = \dl_*(R, M_0, M_1, M_2)  > 0$ sufficiently small), 
we obtain \eqref{YY3}.
This concludes the proof of Proposition~\ref{PROP:stab}.
\end{proof}

\subsection{Convergence of the enhanced Gibbs measures}
\label{SUBSEC:4.2}

Let $\Xi_0(\phi)$ be the first enhanced data set defined in \eqref{xi1}.
By viewing $\Xi_0$ as a map sending $\phi$ to $\Xi_0(\phi)$, 
we consider the pushforward measure $\nu_L$ of 
the $L$-periodic $\Phi^{k+1}_2$-measure $\rho_L$ under the map $\Xi_0$.
Namely, $\nu_L$ is given by 
\begin{align}
\nu_L = (\Xi_0)_\# \rho_L.
\label{nuL}
\end{align}

\noi
See also \eqref{ex1}.
In the following, we refer to $\nu_L$ as the {\it enhanced Gibbs measure}.
In this subsection, we study convergence properties of the enhanced Gibbs measure $\nu_L$,
which 
plays an essential role in establishing global well-posedness
of the hyperbolic $\Phi^{k+1}_2$-model \eqref{NW1} on $\R^2$
(as a limit of the $L$-periodic 
hyperbolic $\Phi^{k+1}_2$-model \eqref{NW2})
and  invariance of the Gibbs measure $\rhoo_\infty$
constructed in Theorem \ref{THM:1}\,(i).

\begin{proposition} \label{PROP:tight}
As probability measures on 
$\WW = \big(C(\R_+;  W^{-\eps, p}_\mu(\R^2))\big)^{\otimes k}$, 
the family  $\{\nu_L\}_{L \ge 1}$ 
 is tight, and thus
there exists a sequence  $\{\nu_{L_j}\}_{j \in \N}$
of the $L_j$-periodic enhanced Gibbs measures
that converges weakly to some limit $\nu_\infty$.
 Moreover, 
$\nu_{L_j}$ also  converges
 to the same limit $\nu_\infty$ in the Wasserstein-1 metric.

\end{proposition}

\begin{proof}  
By the definition \eqref{nuL}, 
we have $\nu_L = \Law (\Xi_0(\phi))$
with $\Law (\phi) = \rho_L$.
Let $X_L$ be the solution to the $L$-periodic SNLH \eqref{SHE7}
with $\Law(X_{0, L}) = \rho_L$.
Then, by the invariance \eqref{CD2a} of $\rho_L$ under~\eqref{SHE7},
we  have $\Law(X_L(1)) = \rho_L$, and thus 
$\nu_L = \Law (\Xi_0(X_L(1))$, 
where
\begin{align}
\Xi_0(X_L(1) ) = \big( \S(t) X_L(1), \,  :\!(\S(t) X_L(1) )^2\!: \, , \dots, \,  :\!(\S(t) X_L(1))^k\!: \big) .
\label{view0}
\end{align}

\noi
Hence, tightness of $\{ \nu_L\}_{L \ge 1} $ 
follows once we prove tightness of $\big\{ \Law (\Xi_0(X_L(1))\big\}_{L\ge 1}$.
Once we have tightness of 
$\{ \nu_L\}_{L \ge 1} $, 
the Prokhorov theorem (Lemma \ref{LEM:Pro})
implies that there exists 
a subsequence 
of $\{\nu_L\}_{L \ge 1}$ converging weakly to some limit $\nu_\infty$.
Moreover, noting that the metric $d_\WW$ on $\WW$ defined in \eqref{norm1} is bounded, 
it follows from 
Lemma \ref{LEM:wass} and 
Remark~\ref{REM:wass}
that the same subsequence of 
$\{ \nu_L\}_{L \ge 1} $ converges to the same limit $\nu_\infty$
in the Wasserstein-1 metric, 
which shows 
the second claim in Proposition~\ref{PROP:tight}.
Therefore, 
we focus on 
proving
 tightness of $\big\{ \Law (\Xi_0(X_L(1))\big\}_{L\ge 1}$
 in the following.

\medskip

By the decomposition \eqref{SHE7a}, 
we have 
\begin{align}
\S(t) X_L(1)  = \S(t) Y_L(1) + \S(t) Z_L(1).
\label{view1}
\end{align}

\noi
Then, from \eqref{Herm3}, 
we have 
\begin{align}
:\!(\S(t) X_L(1))^\l\!: \, 
= \sum_{m=0}^\l \binom{\l}{m}    (\S(t)Y_L(1) )^{\l-m}:\!(\S(t)Z_L(1))^m\!: 
\label{view2}
\end{align}

\noi
for $\l = 1 , \dots, k$.
When $m\ne 0$,  Lemma \ref{LEM:bilin2}\,(ii) yields
\begin{align}
\begin{split}
&  \| (\S(t)Y_L(1) )^{\l-m}:\!(\S(t)Z_L(1))^m\!: \|_{W^{-\eps, p}_\mu}\\
& \quad 
\les 
\| :\!(\S(t)Z_L(1))^m\!: \|_{W^{-\eps, p_1}_\mu}
\| (\S(t)Y_L(1) )^{\l-m}\|_{W^{\eps, p_2}_{\frac \mu{2^\dl}}} 
\end{split}
\label{view3}
\end{align}

\noi
for $1 <  p_1, p_2 < \infty$
with $p_1 \gg 1$, 
satisfying $\frac 1p
\le \frac 1{p_1}  + \frac 1{p_2} \leq \frac 1p + \frac \eps 2$.
When $m = 0$, we simply use the bound:
\begin{align*}
  \| (\S(t)Y_L(1) )^{\l} \|_{W^{-\eps, p}_\mu}
\les 
\| (\S(t)Y_L(1) )^{\l}\|_{W^{\eps, p_2}_{\frac \mu{2^\dl}}} 
\end{align*}

\noi
for $p_2 \ge p$
and apply
 the fractional Leibniz rule (Lemma \ref{LEM:bilin2}\,(i));
see~\eqref{view4} below.

\smallskip

\noi
$\bullet$ {\bf Part 1:}
We first estimate $ :\!(\S(t)Z_L(1))^m\!: \, $, 
$m = 1, \dots, k$.
Recall from \eqref{lin3} with \eqref{lin2} that 
\begin{align*}
\S(t) & = \dt \D(t) + \D(t)
 = e^{-\frac{t}{2}} \cos( t \jbb{\nb} ) + \frac{1}{2}  e^{-\frac{t}{2}}  \frac{\sin( t\jbb{\nb} )}{\jbb{\nb}}, 
\end{align*}

\noi
where
\begin{align}
\jbb{\nb}  = \sqrt{\tfrac{3}{4} - \Dl}.
\label{SS0}
\end{align}

\noi
Given $\ld \in \Z^2_L$, we set 
\begin{align}
\ft {\S(t)}(\ld)  = e^{-\frac{t}{2}} \cos( t \jbb{2\pi \ld } ) 
+ \frac{1}{2}  e^{-\frac{t}{2}}  \frac{\sin( t\jbb{2\pi \ld } )}{\jbb{2\pi \ld }}, 
\label{SS1}
\end{align}

\noi
where $\jbb{x} = \sqrt{\frac 34 + |x|^2}$.
From \eqref{SHE8}, we have
\begin{align}
\S(t) Z_L(1) &=  \sqrt 2\int_{0}^1  \S(t) P(1-t') \xi_L(dt'), 
\label{SS1a}
\end{align}

\noi
where $\xi_L$ is as in \eqref{wh4} and \eqref{wh5}.

We first consider the case $m = 1$.
From \eqref{SS1a} with \eqref{wh4} and \eqref{wh5}, we have
\begin{align}
\begin{aligned}
\E & \Big[ \big| \jb{\nb}^{-\eps}\big( \phi_j \S(t) Z_L(1) \big) (x) \big|^2 \Big] \\
&=\frac{2}{L^2} \E \Bigg[  \bigg|    \sum_{n\in\Z^2}
\int_{\R^2} \frac{1}{\jb{\eta}^\eps}
\ft{ \S(t)}\big(\tfrac nL\big)\\
&\hphantom{XXXX}
\times 
\int_{0}^1
e^{-(1-t') \jb{ \frac{2\pi n}{L}  }^2 } \ft{\xi}_L\big( \tfrac{n}{L}, dt' \big) 
\ft \phi_j \big(\eta- \tfrac{n}{L}\big)
  e^{2 \pi i\eta\cdot x} d\eta  \bigg|^2 \Bigg].
\end{aligned}
\label{SS2}
\end{align}

\noi
Recall that $\big\{ \ft{\xi}_L\big(\cdot, \tfrac{n}{L} \big) \big\}_{n \in \Z^2}$
is a family of independent 
temporal white noises
conditioned that $\ft \xi_L\big(- \tfrac nL\big) = \cj{\ft \xi_L\big( \tfrac nL\big)}$, 
$n \in \Z^2$,
and thus the temporal stochastic integrals in \eqref{SS2}
are standard Wiener integrals.
Hence, we have 
\begin{align}
\eqref{SS2} 
&\les \frac{e^{-t} }{L^2} \sum_{n\in\Z^2} \frac{1}{ \jb{ \frac{n}{L}  }^2 } 
 \bigg| \int_{\R^2} \frac{1}{\jb{\eta}^\eps} \ft \phi_j \big(\eta- \tfrac{n}{L}\big)
  e^{2\pi i\eta\cdot x} d\eta  \bigg|^2.
\label{SS3}
\end{align}

From 
the triangle inequality 
$\jb{\frac nL}^\eps\les \jb{\eta}^\eps \jb{\eta - \tfrac nL}^\eps$
with 
\eqref{Damp5a} and \eqref{Damp5b}, we have
\begin{align}
\begin{split}
 \bigg| \int \frac{1}{\jb{\eta}^\eps} \ft \phi_j \big(\eta- \tfrac{n}{L}\big)
  e^{2\pi i\eta\cdot x} d\eta  \bigg|
& \les  \frac{1}{\jb{\frac nL}^\eps}  \int_{\R^2}
\big |\jb{\eta - \tfrac nL}^\eps \ft \phi_j \big(\eta- \tfrac{n}{L}\big)\big| d\eta \\
& \les   \frac{1}{\jb{\frac nL}^\eps} . 
\end{split}
\label{SS4}
\end{align}

\noi
On the other hand, we have
\begin{align}
\int_{\R^2} \frac{1}{\jb{\eta}^\eps} \ft \phi_j \big(\eta- \tfrac{n}{L}\big)
  e^{2\pi i\eta\cdot x} d\eta 
= \int_{\R^2} G_\eps(x-y)   \phi_j (y) e^{2\pi i \frac nL\cdot y} dy, 
\label{SS5}
\end{align}

\noi
where $G_\eps$ is the kernel of the Bessel potential $\jb{\nb}^{-\eps}$
of order $\eps$.
Recall the following bound on the kernel $G_\eps(x)$ (see \cite[Proposition 1.2.5]{Gra2}):
\begin{align}
|G_\eps(x)| 
\les \begin{cases}
e^{-\frac{|x|}2}, & \text{for } |x| \geq 2,\\
|x|^{\eps - 2},  & \text{for } |x| <  2,
\end{cases}
\label{SS6}
\end{align}

\noi
since  $0 < \eps < 2$.	
Thus,  it follows from \eqref{SS5} and \eqref{SS6}  with \eqref{phi} that 
\begin{align}
\begin{split}
\bigg|& \int_{\R^2} \frac{1}{\jb{\eta}^\eps} \ft \phi_j \big(\eta- \tfrac{n}{L}\big)
  e^{2\pi i\eta\cdot x} d\eta \bigg|
\le \int_{\R^2} |G_\eps(x-y)|   |\phi_j (y)| dy\\
& = \int_{|x-y|\ge 2} |G_\eps(x-y)|   |\phi_j (y)| dy
+ \int_{|x-y|<  2} |G_\eps(x-y)|   |\phi_j (y)| dy\\
& \les \int_{|x-y|\ge 2} \frac{1}{|x- y|^{M+2}}  \frac{1}{\jb{\frac{y}{2^j}}^M}  dy
+ \int_{|x-y|<  2}\frac{1}{|x- y|^{2-\eps}}  \frac{1}{\jb{\frac{x-y}{2^j}}^M}  \frac{1}{\jb{\frac{y}{2^j}}^M}  dy\\
& \les   \frac{1}{\jb{\frac{x}{2^j}}^M}  
\end{split}
\label{SS7}
\end{align}

\noi
for any $M \gg1$, where, in the third step, 
we used the fact that $\jb{\frac{x-y}{2^j}} \sim 1$ for $| x - y| \le 2$.
Hence, by interpolating \eqref{SS4} and \eqref{SS7}, 
we obtain from \eqref{SS2},  \eqref{SS3}, 
and a Riemann sum approximation that 
\begin{align}
\begin{aligned}
\E  \Big[ \big|&  \jb{\nb}^{-\eps}\big( \phi_j \S(t) Z_L(1) \big) (x) \big|^2 \Big] 
\les \frac{e^{-t} }{L^2} \sum_{n\in\Z^2} 
\frac{1}{ \jb{ \frac{n}{L}  }^{2+2\eps-} } 
  \frac{1}{\jb{\frac{x}{2^j}}^M}  \\
&\les \frac{e^{-t} }
{\jb{\frac{x}{2^j}}^M}  
\int_{\R^2} 
\frac{1}{ \jb{ y  }^{2+2\eps-} } 
\les \frac{e^{-t} }
{\jb{\frac{x}{2^j}}^M}  
\end{aligned}
\label{SS8}
\end{align}

\noi
for any $M \gg 1$, uniformly in $L \ge 1$.

Given any finite $q \ge p \ge 2$, 
it follows from 
\eqref{WSP}, Minkowski's integral inequality, 
 the Wiener chaos estimate (Lemma \ref{LEM:hyp}),
 and \eqref{SS8} that 
\begin{align}
\begin{split}
\big\| \| \S(t) & Z_L(1) \|_{W^{-\eps, p}_\mu} \big\|_{L^q(\O)}
\le 
\Big\| w_{\frac \mu p} (2^{j})
\big\|
\jb{\nb}^{-\eps}\big(\phi_j \S(t) Z_L(1)\big)(x)\big\|_{L^q(\O)}
\Big\|_{\l^p_jL^p_x }\\
& \les 
q^\frac 12 \Big\| w_{\frac \mu p} (2^{j})
\big\|
\jb{\nb}^{-\eps}\big(\phi_j \S(t) Z_L(1)\big)(x)\big\|_{L^2(\O)}
\Big\|_{\l^p_jL^p_x }\\
& \les 
q^\frac 12 e^{-\frac{t}{2}} \bigg\| w_{\frac \mu p} (2^{j})
\frac{1}{\jb{\frac{x}{2^j}}^\frac M2}  
\bigg\|_{\l^p_jL^p_x }
 \les 
q^\frac 12 e^{-\frac{t}{2}} \| w_{\frac \mu p} (2^{j})
2^{\frac 2p j}
\|_{\l^p_j }\\
& \les 
q^\frac 12 e^{-\frac{t}{2}}, 
\end{split}
\label{SS10}
\end{align}

\noi
uniformly in $L \ge 1$.

From \eqref{SS1} and the mean value theorem, we have 
\begin{align}
|\ft{\S(t + h)}(\ld) - \ft{\S(t)}(\ld)|
\les e^{-\frac t2} \min\big(1, |h| \jb{\ld}\big)^\kk
\label{SS11}
\end{align}
	
\noi
for any $\ld \in\Z^2_L$ and $0 \le \kk \le 1$.
Then, by repeating an analogous computation with \eqref{SS11}, 
we obtain
\begin{align}
\begin{split}
\big\| \| \S(t+h)  Z_L(1) - \S(t)  Z_L(1) \|_{W^{-\eps, p}_\mu} \big\|_{L^q(\O)}
& \les 
q^\frac 12 e^{-\frac{t}{2}}|h|^{\kk}
\end{split}
\label{SS12}
\end{align}

\noi
for any $q \ge 1$
and $0 < \kk < \eps$, 
 uniformly in $L \ge 1$.

Let $T > 0$.
In view of \eqref{SS10}
and \eqref{SS11}, by applying  Kolmogorov's continuity criterion
(\cite[Theorem 8.2]{Bass}), 
we see that 
$ \S(t)Z_L(1) \in C^\al([0, T];  W^{-\eps, p}_\mu(\R^2))$ for some small $\al >0$.
Moreover, 
 the $\al$-H\"older semi-norms of $\{ \S(t)Z_L(1): t\in[0,T] \}$
are uniformly (in $L$)
 bounded in $L^q(\O)$ for any  finite $q\geq 1$.
Therefore,  
by noting that the discussion above holds
for any small $\eps > 0$ and any $\mu > 0$, 
it follows from 
the Arzel\`a-Ascoli theorem with Lemma \ref{LEM:cpt}
that 
 the family  $\{ \S(t)Z_L(1) \}_{L\geq 1}$
in $C([0, T];  W^{-\eps, p}_\mu(\R^2))$
is tight\footnote{Namely, the laws
of $\{ \S(t)Z_L(1): t\in[0,T] \}_{L\geq 1}$
are tight as probability measures on 
 $C([0, T];  W^{-\eps, p}_\mu(\R^2))$.}
 for any finite $T$, and hence
$\{ \S(t)Z_L(1) \}_{L\geq 1}$
in $C(\R_+;  W^{-\eps, p}_\mu(\R^2))$
 is tight.

\begin{remark}\label{REM:bound} \rm
By repeating the  argument above
for $e^{\frac t2}\S(t)Z_L(1) $, 
we see that 
$e^{\frac t2}\S(t)Z_L(1) $
is almost surely 
$\al$-H\"older continuous in time.
In particular, it grows at most polynomially in time.
From this observation, 
we conclude that 
$\S(t)Z_L(1) $
is almost surely bounded on $\R_+$
with values in $W^{-\eps, p}_\mu(\R^2)$.

\end{remark}

We now study the higher powers
$ :\!(\S(t)Z_L(1))^m\!: \, $, $m = 2, \dots, k$.
From Lemma \ref{LEM:W1} and~\eqref{SS1a} with \eqref{wh4} and \eqref{wh5}, we have\footnote{Strictly speaking, in applying Lemma \ref{LEM:W1}, we need to go back to 
the frequency truncated version of the Wick power, just  as in \eqref{Wick2} for the damped wave equation,
and remove the frequency truncation.}
\begin{align}
\begin{split}
& \E\big[ :\!(\S(t) Z_L(1))^m(y)\!: \, :\!(\S(t) Z_L(1))^m(y')\!:\big]\\
& \quad 
= m!
\Big\{\E\big[ (\S(t) Z_L(1))(y) \cdot (\S(t) Z_L(1))(y')\big]\Big\}^m\\
& \quad  \sim \frac{1}{L^{2m}}
\sum_{n \in \Z^2} e^{2\pi i \frac nL \cdot (y - y')}
\sum_{\substack{n_1, \dots, n_m \in \Z^2\\n = n_1 + \cdots + n_m}}
\prod_{j = 1}^m\bigg\{\big(\ft{ \S(t)}\big(\tfrac {n_j}L\big)\big)^2
\frac{1 - e^{-2\jb{\frac {2\pi n_j}L}^2}}{2\jb{\frac {2\pi n_j}L}^2}\bigg\}.
\end{split}
\label{ST1}
\end{align}

\noi
Then, by \eqref{ST1}
and interpolating \eqref{SS4} and \eqref{SS7}, we have
\begin{align*}
\E & \Big[ \big| \jb{\nb}^{-\eps}\big( \phi_j :\!(\S(t) Z_L(1))^m\!: \big) (x) \big|^2 \Big] \\
& = \int_{\R^2}\int_{\R^2}\frac{1}{\jb{\eta}^\eps}\frac{1}{\jb{\eta'}^\eps}
\int_{\R^2}\int_{\R^2}\phi_j(y)\phi_j(y')\\
&\quad 
 \times \E\big[ :\!(\S(t) Z_L(1))^m(y)\!: \, :\!(\S(t) Z_L(1))^m(y')\!:\big]\\
&\quad  \times 
e^{-2\pi i \eta\cdot y}e^{-2\pi i \eta'\cdot y'}dy dy'
e^{2\pi i \eta\cdot x}e^{2\pi i \eta'\cdot x}
d\eta  d\eta'\\
&\sim \frac{1}{L^{2m}} 
\sum_{n \in \Z^2}
\bigg(\int_{\R^2}\frac{1}{\jb{\eta}^\eps}
\ft \phi_j\big(\eta - \tfrac nL\big) e^{2\pi i \eta\cdot x}d\eta\bigg)
\bigg( \int_{\R^2} \frac{1}{\jb{\eta'}^\eps}\ft \phi_j\big(\eta' + \tfrac nL\big) 
e^{2\pi i \eta'\cdot x}
d\eta'\bigg)\\
&\quad  \times
\sum_{\substack{n_1, \dots, n_m \in \Z^2\\n = n_1 + \cdots + n_m}}
\prod_{j = 1}^m\bigg\{\big(\ft{ \S(t)}\big(\tfrac {n_j} L\big)\big)^2
\frac{1 - e^{-2\jb{\frac {2\pi n_j}L}^2}}{2\jb{\frac {2\pi n_j}L}^2}\bigg\}\\
&\les 
  \frac{e^{-mt}}{\jb{\frac{x}{2^j}}^M}  
  \frac{1}{L^{2m}} 
\sum_{n \in \Z^2}
\frac{1}{ \jb{ \frac{n}{L}  }^{2\eps-} } 
\sum_{\substack{n_1, \dots, n_m \in \Z^2\\n = n_1 + \cdots + n_m}}
\prod_{j = 1}^m
\frac{1}{\jb{\frac { n_j}L}^2}.
\end{align*}

\noi
By a Riemann sum approximation, we then obtain
\begin{align}
\begin{aligned}
\E & \Big[ \big| \jb{\nb}^{-\eps}\big( \phi_j :\!(\S(t) Z_L(1))^m\!: \big) (x) \big|^2 \Big] \\
&\les 
  \frac{e^{-mt}}{\jb{\frac{x}{2^j}}^M}  
\int_{(\R^2)^m}  
\frac{1}{ \jb{ y_1 + \cdots + y_m  }^{2\eps-} } 
\prod_{j = 1}^m
\frac{1}{\jb{y_j}^2} dy_1 \cdots dy_m\\
&\les 
  \frac{e^{-mt}}{\jb{\frac{x}{2^j}}^M} 
\end{aligned}
\label{ST2}
\end{align}

\noi
for any $M \gg 1$, uniformly in $L \ge 1$.

We now proceed as in \eqref{SS10}.
Namely, 
given any finite $q \ge p \ge 2$, 
it follows from 
\eqref{WSP}, Minkowski's integral inequality, 
 the Wiener chaos estimate (Lemma \ref{LEM:hyp}),
 and \eqref{ST2} that 
\begin{align}
\begin{split}
\big\| \| :\!(S(t) & Z_L(1))^m\!: \|_{W^{-\eps, p}_\mu} \big\|_{L^q(\O)}\\
& \les q^\frac m2 
\Big\| w_{\frac \mu p} (2^{j})
\big\|
\jb{\nb}^{-\eps}\big(\phi_j :\!(S(t)  Z_L(1))^m\!: \big)(x)\big\|_{L^2(\O)}
\Big\|_{\l^p_jL^p_x }\\
& \les q^\frac m2 e^{-\frac{m}{2}t} \bigg\| w_{\frac \mu p} (2^{j})
\frac{1}{\jb{\frac{x}{2^j}}^\frac M2}  
\bigg\|_{\l^p_jL^p_x }
 \les 
q^\frac m2 e^{-\frac{m}{2}t}  \| w_{\frac \mu p} (2^{j})
2^{\frac 2p j}
\|_{\l^p_j }\\
& \les 
q^\frac m2 e^{-\frac{m}{2}t}, 
\end{split}
\label{ST3}
\end{align}

\noi
uniformly in $L \ge 1$.

We briefly discuss how to handle the time increment
$ :\!(\S(t)Z_L(1))^m\!: 
-  :\!(\S(t)Z_L(1))^m\!: \, $.
The main idea is to proceed as in the second half of 
the proof of Proposition 2.1 in \cite{GKO}.
Given $h \in \R$, define the difference operator
$\dl_h$ by setting\footnote{We assume that $t + h \ge 0$.}
\[ \dl_h F(t) = F(t+h) - F(t).\]

\noi
Then, by Lemma \ref{LEM:W1}, we have
\begin{align*}
\frac{1}{m!} \E& \big[ \dl_h:\!(\S(t) Z_L(1))^m(y)\!: \, \dl_h :\!(\S(t) Z_L(1))^m(y')\!:\big]\\
& 
= 
\Big\{\E\big[ (\S(t+h) Z_L(1))(y) \cdot (\S(t+h) Z_L(1))(y')\big]\Big\}^m\\
&\quad  - \Big\{\E\big[ (\S(t) Z_L(1))(y) \cdot (\S(t+h) Z_L(1))(y')\big]\Big\}^m\\
& \quad - \Big\{\E\big[ (\S(t+h) Z_L(1))(y) \cdot (\S(t) Z_L(1))(y')\big]\Big\}^m\\
& \quad + \Big\{\E\big[ (\S(t) Z_L(1))(y) \cdot (\S(t) Z_L(1))(y')\big]\Big\}^m\\
& 
= 
\E\big[ \dl_h (\S(t) Z_L(1))(y) \cdot (\S(t+h) Z_L(1))(y')\big]\\
& \quad\quad \times\sum_{j = 0}^{m-1}
\Big\{\E\big[ (\S(t+h) Z_L(1))(y) \cdot (\S(t+h) Z_L(1))(y')\big]\Big\}^{m-j-1}\\
& \quad \quad\quad  \quad\times \Big\{\E\big[ (\S(t) Z_L(1))(y) \cdot (\S(t+h) Z_L(1))(y')\big]\Big\}^{j}\\
& \quad - \E\big[ \dl_h (\S(t) Z_L(1))(y) \cdot (\S(t) Z_L(1))(y')\big]\\
& \quad \quad\times\sum_{j = 0}^{m-1}
\Big\{\E\big[ (\S(t+h) Z_L(1))(y) \cdot (\S(t) Z_L(1))(y')\big]\Big\}^{m-j-1}\\
& \quad\quad \quad \quad \times \Big\{\E\big[ (\S(t) Z_L(1))(y) \cdot (\S(t) Z_L(1))(y')\big]\Big\}^{j}
\end{align*}

\noi
Note that in each term on the right-hand side above, 
the difference operator $\dl_h$ appears exactly once.
Then, with \eqref{SS11}, we can repeat the computation above
for $ :\!(\S(t)Z_L(1))^m\!: \, $
and obtain
\begin{align}
\begin{split}
\big\| \| \dl_h :\!(\S(t+h)  Z_L(1))^m\!:  \|_{W^{-\eps, p}_\mu} \big\|_{L^q(\O)}
& \les 
q^\frac m2 e^{-\frac{m}{2}t}|h|^{\kk}
\end{split}
\label{ST4}
\end{align}

\noi
for any $q \ge 1$
and $0 < \kk < \eps$, 
 uniformly in $L \ge 1$.

Therefore, 
applying
  Kolmogorov's continuity criterion
with \eqref{ST3} and \eqref{ST4}
  and arguing as in the $m = 1$ case
  with Lemma \ref{LEM:cpt}
and the Arzel\`a-Ascoli theorem, 
 we conclude that 
the family 
$\{ :\!(\S(t)Z_L(1))^m\!:\}_{L\geq 1}$
in $C(\R_+;  W^{-\eps, p}_\mu(\R^2))$
is tight.

\smallskip

\noi
$\bullet$ {\bf Part 2:}
Next, we consider 
$(\S(t)Y_L(1) )^m$, $m = 1, \dots, k$.
By the fractional Leibniz rule (Lemma~\ref{LEM:bilin2}\,(i)), 
\eqref{sob2}, and Lemma~\ref{LEM:wave}, we have 
\begin{align}
\begin{split}
\| (\S(t)Y_L(1) )^m\|_{L^\infty_t W^{\eps, p}_{\frac \mu{2^\dl}}} 
& \les 
\| \S(t)Y_L(1) \|^m_{L^\infty_t W^{\eps, mp }_{\mu_1}} 
\les 
\| \S(t)Y_L(1) \|^m_{L^\infty_t H^{1 - \frac 2{mp} + \eps}_{\mu_2}} \\
& \les 
\sup_{t  \in \R_+} e^{-\frac m4 t} \| Y_L(1) \|^m
_{H^{1 - \frac 2{mp} + \eps}_{\mu_3}} 
 \le
 \| Y_L(1) \|^m
_{H^{1 - \frac 2{mp} + \eps}_{\mu_3}} 
\end{split}
\label{view4}
\end{align}

\noi
for some $0 < \mu_3 < \mu_2  < \mu_1 < \mu$.
Here, we take $\eps > 0$ sufficiently small such that $ 1 - \frac 2{mp} + \eps< 1$.
It follows from 
 Proposition \ref{PROP:CI3} 
with \eqref{CZ} and \eqref{CZ2}, Lemma~\ref{LEM:hyp}, 
and \cite[Theorem 5.1]{MW17a}\footnote{Theorem 5.4 in the arXiv version.} 
that, given any finite $q\ge 1$, 
 the $q$th moment of the right-hand side of~\eqref{view4}
is bounded, uniformly in $L \ge 1$.

A similar computation with \eqref{SS11} (but with $\ld \in \R^2$)
yields
\begin{align}
\begin{split}
\| \dl_h (\S(t)Y_L(1) )^m\|_{L^\infty_t W^{\eps, p}_{\frac \mu{2^\dl}}} 
 \les
|h|^\kk  \| Y_L(1) \|^m
_{H^{1 - \frac 2{mp} + \eps+\kk }_{\mu_3}} .
\end{split}
\label{view5}
\end{align}

\noi
Then, 
as long as $\eps, \kk > 0$ are  sufficiently small such  that $ 1 - \frac 2{mp} + \eps+ \kk <  1$, 
 Proposition~\ref{PROP:CI3} 
with~\eqref{CZ} and \eqref{CZ2}, Lemma~\ref{LEM:hyp}, 
and \cite[Theorem 5.1]{MW17a}\footnote{Theorem 5.4 in the arXiv version.} 
that, given any finite $q\ge 1$, 
 the $q$th moment of the right-hand side of~\eqref{view5}
is bounded, uniformly in $L \ge 1$.
Therefore, 
applying
  Kolmogorov's continuity criterion
with \eqref{view4} and \eqref{view5}
  and arguing as in the $m = 1$ case
  with Lemma \ref{LEM:cpt}
and the Arzel\`a-Ascoli theorem, 
we conclude that 
the family 
 $\{ (\S(t)Y_L(1) )^m\}_{L\geq 1}$
in $C(\R_+;  W^{\eps, p}_{\frac{\mu}{2^\dl}}(\R^2))$
is tight.

\smallskip

\noi
$\bullet$ {\bf Part 3:}
Finally, from \eqref{view2}, we have
\begin{align}
\begin{split}
\dl_h :\!(\S(t) X_L(1))^\l\!: \, 
& = \sum_{m=0}^\l \binom{\l}{m}    \dl_h (\S(t)Y_L(1) )^{\l-m}\cdot :\!(\S(t+h)Z_L(1))^m\!: \\
& \quad 
+ \sum_{m=0}^\l \binom{\l}{m}    (\S(t)Y_L(1) )^{\l-m}\cdot  \dl_h :\!(\S(t)Z_L(1))^m\!: .
\end{split}
\label{view6}
\end{align}

\noi
By applying Lemma \ref{LEM:bilin2}\,(ii) as in \eqref{view3}, 
we conclude  from \eqref{SS10}, \eqref{SS12}, \eqref{ST3}, \eqref{ST4}. 
\eqref{view4}, and \eqref{view5}
with 
  Kolmogorov's continuity criterion
  (\cite[Theorem 8.2]{Bass}),  
Lemma \ref{LEM:cpt}, 
and the Arzel\`a-Ascoli theorem
that 
the family  $\{ :\!(\S(t)X_L(1))^m\!:\}_{L\geq 1}$
in $C(\R_+;  W^{-\eps, p}_\mu(\R^2))$
is tight.
Namely, 
as probability measures on $\WW = C(\R_+;  W^{-\eps, p}_\mu(\R^2))^{\otimes k}$, 
the family $\{\nu_L\}_{L\in \N}$ is tight.
\end{proof}

Let 
 $\A$  be the index set 
 such that Theorem \ref{THM:1}\,(i) holds.
 Namely, 
 that 
the $L$-periodic $\Phi^{k+1}_2$-measures $\rho_{L}$, $L \in \A$  converges weakly to a
limiting $\Phi^{k+1}_2$-measure  $\rho_\infty$ on $\R^2$.
Note that Proposition \ref{PROP:tight}
also applies to the family $\{ \nu_L\}_{L \in \A}$,
showing that there exists a subsequence
which converges, weakly and also in the Wasserstein-1 metric, to {\it some} limit $\nu_\infty$.
 The next proposition  identifies 
this limiting probability measure $\nu_\infty$.

\begin{proposition} \label{PROP:uniq}
Let the index set $\A \subset \N$ be as in Theorem \ref{THM:1}\,(i).
Then, as probability measures on 
$\WW = \big(C(\R_+;  W^{-\eps, p}_\mu(\R^2))\big)^{\otimes k}$, 
 the entire sequence $\{\nu_L\}_{L\in \A}$
 of the $L$-periodic enhanced Gibbs measures
converges, weakly and also in the Wasserstein-1 metric, to 
the unique limit 
\begin{align}
\nu_\infty = (\Xi_0)_\# \rho_\infty.
\label{need1}
\end{align}

\end{proposition}

\begin{proof}
In the following, we only consider the values of $L$ belonging to $\A = \{ L_j: j \in \N\} \subset \N$.
For simplicity of notation, 
we drop the subscript 
and simply write $L \to \infty$ with the understanding that each $L$ belongs to $\A$.

In view of the tightness of $\{\nu_L\}_{L\in \A}$ (Proposition \ref{PROP:tight})
and Lemma \ref{LEM:wass}, 
we only need to show that the limit in law 
of $\{\nu_L\}_{L\in \A}$ 
is unique.

Let $X_{L} = Y_{L} + Z_{L}$ be the  solution to 
the ${L}$-periodic SNLH  \eqref{SHE7}, 
where $Y_{L}$ and $Z_{L}$ satisfy~\eqref{SHE9} and~\eqref{SHE8}
with the Gibbsian initial data $X_{0, L}$ with 
$\Law(X_{0, L})  =  \rho_{L}$.
Here, we assume that $X_{0, L}$ is independent of the noise $\xi_L$. 
In view of the weak convergence of $\rho_L$
to $\rho_\infty$ as $L \to \infty$
 (with $L$'s belonging to $\A$)
 established 
in Subsection \ref{SUBSEC:3.3}, 
it follows from 
the Skorokhod representation theorem (Lemma \ref{LEM:Sk})
that 
there exist a probability measure $\wt \PP$ (on a new probability space) and random variables 
$\wt X_{0, L}$ with $\Law(\wt X_{0, L})  =  \rho_{L}$, $L  \in \A$, 
such that $\wt X_{0, L}$ converges $\wt \PP$-almost surely to some $\wt X_{0, \infty}$, with 
$\Law(\wt X_{0, \infty})  =  \rho_{\infty}$, 
in $B^{-\eps, \mu}_{p_0, p_0}(\R^2)$ for any $\eps > 0$, finite $p_0 \ge 1$, and $\mu > 0$;
see also Remark \ref{REM:Gibbs2}.
Note that by the independence of $X_{0, L}$ and $\xi _L$, 
the change of the probability spaces due to 
the use of the Skorokhod representation theorem (Lemma \ref{LEM:Sk})
does not affect the noise $\xi_L$ (and hence $Z_L$).

Let $\wt X_L = \wt Y_L + Z_L$ be the solution to \eqref{SHE7}
with the new initial data $\wt X_{0,  L}$, 
where $\wt Y_L$ satisfies~\eqref{SHE9} with the initial data $X_{0, L}$ replaced by $\wt X_{0,  L}$.
By the invariance of $\rho_{L}$ under~\eqref{SHE7}, we have 
$\Law(\wt X_{L}(1))=  \rho_{L}$.
Thus, it suffices to 
show that 
\begin{itemize}
\item[(i)]
$\Xi_0(  \wt X_{L}(1)   )$
converges in law to 
$\Xi_0(  \wt X_\infty(1)   )$
as $L \to \infty$ (with $L \in \A$),
where $\wt X_\infty$ is the solution to \eqref{SHE1}
with the initial data $\wt X_{0, \infty}$ satisfying $\Law(\wt X_{0, \infty})  =  \rho_{\infty}$, 
and $\Xi_0(\phi)$ is the enhanced data set defined in~\eqref{xi1}, 
and

\smallskip

\item[(ii)]
$\Law (\wt X_\infty(1)) = \rho_\infty$.
See Remark \ref{REM:conv2} below.

\end{itemize}

\smallskip

\noi
Before proceeding further, we recall the decomposition $\wt X_\infty = \wt Y_\infty  + Z_\infty$, 
where $\wt Y_\infty$ is the solution to \eqref{SHE4}
with the initial data $\wt X_{0, \infty}$
and $Z_\infty$ is the solution to \eqref{SHE2}.
For simplicity of notation, we denote $\wt X_L$,  $\wt Y_L$, etc. by $X_L$, $Y_L$, etc.~in the remaining part of the proof.

Given  $R > 0$, 
let $B_R \subset \R^2$ denotes the ball of radius $R$ centered at the origin.
In view of the tightness of $\{ \nu_L\}_{L \in \A}$ (Proposition \ref{PROP:tight}), 
it suffices to show that, given any $t \in \R_+$ and $R > 0$, 
$: \!( \S(t)X_{L}(1))^\l \!: $ converges in probability  to
$: \!( \S(t)X_\infty(1))^\l \!:$ 
in  $W^{-\eps, p}(B_R)$
as $L \to \infty$ (with $L \in \A$).
Then, for each fixed $t \in \R_+$, 
this allows us to identify $: \!( \S(t)X_\infty(1))^\l \!:$  with the limit in law of 
$: \!( \S(t)X_{L}(1))^\l \!:$ as a $\D'(\R^2)$-valued random variable
by testing against 
 a test function $\varphi\in  \D(\R^2) = C^\infty_c(\R^2)$
 with $\supp \varphi \subset B_R$.
From  this observation and Proposition \ref{PROP:tight}
with the uniqueness of a limit, 
we then conclude 
that 
the entire sequence
$\{\Xi_0(X_L(1))\}_{L\in \A}$ converges in law to 
$\Xi_0(X_\infty(1))$
as $L \to \infty$ (with $L \in \A$).

\begin{remark}\label{REM:conv2}
\rm

Once we prove (i) in the sense explained above, 
the claim (ii) above 
follows immediately.
On the one hand, 
as probability measures on $B^{-\eps, \mu}_{p_0, p_0}(\R^2)$, 
 $\rho_L = \Law (X_L(1))$
converges weakly to $\rho_\infty$ as $L \to \infty$ (with $L \in \A)$.
On the other hand, we will show that, given $t \in \R_+$, 
$ \S(t)X_{L}(1) $ converges in law   to
$ \S(t)X_\infty(1)$ as $L \to \infty$ (with $L \in \A)$.
With $t = 0$, 
this implies that, 
as probability measures on 
$\D'(\R^2)$, 
 $\rho_L = \Law (X_L(1))$ converges weakly to 
$\Law (X_\infty(1))$.
Therefore, from the uniqueness of a limit, we conclude that 
$\Law ( X_\infty(1)) = \rho_\infty$.

Note that a similar argument shows that 
$\Law ( X_\infty(t)) = \rho_\infty$ for any $t > 0$, 
thus yielding invariance  of the limiting $\Phi^{k+1}_2$-measure $\rho_\infty$ on $\R^2$
under the dynamics of the parabolic $\Phi^{k+1}_2$-model \eqref{SHE1}
claimed in Remark \ref{REM:inv1}.

\end{remark}

Fix $t \in \R_+$ and $R> 0$ in the remaining part of the proof.
From Lemma \ref{LEM:bilin1}\,(ii) (see also Remark \ref{REM:R}) with \eqref{view1} and~\eqref{view2}, 
we have 
\begin{align}
 \|  & :\!  (\S(t) X_L(1))^\l\!: \, 
- :\!(\S(t) X_{L'}(1))^\l\!: \, 
\|_{W^{-\eps, p}(B_R)} \notag \\
& \les  \sum_{m=0}^\l \binom{\l}{m}  
 \Big\{\| :\!(\S(t)Z_{L}(1))^m\!: - :\!(\S(t)Z_{L'}(1))^m\!: \|_{W^{-\eps, p_1}(B_R)}\notag \\
&
\hphantom{XXXXXXX}
 \times 
\| (\S(t)Y_{L}(1) )^{\l-m}\|_{W^{\eps, p_2}(B_R)} 
\label{uni1} \\
& \hphantom{XXXXXX}
+ 
\| :\!(\S(t)Z_{L'}(1))^m\!: \|_{W^{-\eps, p_1}(B_R)}\notag \\
& 
\hphantom{XXXXXXX}
\times 
\| (\S(t)Y_{L}(1) )^{\l-m}
- (\S(t)Y_{L'}(1) )^{\l-m}\|_{W^{\eps, p_2}(B_R)} \Big\}\notag 
\end{align}

\noi
for any $L, L' \in \A$, 
where $1 <  p_1, p_2 < \infty$ 
with $p_1 \gg 1$, 
satisfying $\frac 1p \le \frac 1{p_1}  + \frac 1{p_2} \leq \frac 1p + \frac \eps 2$.
In the following, we separately study 
convergence of $(\S(t)Y_{L}(1) )^{\l-m}$ and 
$ :\!(\S(t)Z_{L'}(1))^m\!:$\,.

\medskip

\noi
$\bullet$ {\bf Part 1:}
Let us first study the terms involving $Y_L$ and $Y_{L'}$.
Proceeding as in \eqref{view4} with 
 the fractional Leibniz rule (Lemma~\ref{LEM:bilin1}\,(i)), 
 Sobolev's inequality, and the boundedness 
 of $\S(t)$
 on $H^s(\R^2)$, 
we have
\begin{align}
\begin{split}
\| (\S(t)Y_L(1) )^{\l - m}\|_{L^\infty_t W_x^{\eps, p_2}(B_R)} 
& \les 
\| \S(t)Y_L(1) \|^{\l - m}_{L^\infty_t W_x^{\eps, (\l - m)p_2 }(B_R)}\\
& \les 
\| \S(t)Y_L(1) \|^{\l - m}_{L^\infty_t H_x^{1 - \frac 2{(\l - m)p_2} + \eps}(B_R)}\\
& \les
 \| Y_L(1) \|^{\l - m}
_{H^{1 - \frac 2{(\l - m)p_2} + \eps}(B_R)}.
\end{split}
\label{uni2}
\end{align}

\noi
Then, by  H\"older's inequality and \eqref{besov2}, we have
\begin{align}
\begin{split}
\| (\S(t)Y_L(1) )^{\l - m}\|_{L^\infty_t W_x^{\eps, p_2}(B_R)} 
& \le C(R) 
 \| Y_L(1) \|^{\l - m}
_{B^{1 - \frac 2{(\l - m)p_2} + 2\eps}_{p_3, \infty}(B_R)}
\end{split}
\label{uni3}
\end{align}

\noi
for any $p_3 \geq 2$.
Proposition \ref{PROP:CI3}
with Theorem 5.1 in \cite{MW17a},\footnote{Theorem 5.4 in the arXiv version.}
controlling the right-hand side of~\eqref{CZ2a}, 
shows that, given any finite $ q\ge 1$, 
the $q$th moment of the right-hand side of \eqref{uni3}
is bounded, uniformly in $L \ge 1$.

A slight modification of \eqref{uni2} and \eqref{uni3} yields
\begin{align}
\begin{split}
 \|&  (  \S(t)Y_L(1) )^{\l - m} - (\S(t)Y_{L'} (1) )^{\l - m}\|_{L^\infty_t W_x^{\eps, p_2}(B_R)} \\
&  \le C(R) 
\Big( \| Y_L(1) \|^{\l - m- 1}
_{B^{1 - \frac 2{(\l - m)p_2} + 2\eps}_{p_3, \infty}(B_R)}
+ \| Y_{L'}(1) \|^{\l - m- 1}
_{B^{1 - \frac 2{(\l - m)p_2} + 2\eps}_{p_3, \infty}(B_R)}
\Big)\\
& \quad  \times
 \| Y_L(1) -  Y_{L'}(1) \|
_{B^{1 - \frac 2{(\l - m)p_2} + 2\eps}_{p_3, \infty}(B_R)}
\end{split}
\label{uni4}
\end{align}

\noi
for any $L, L' \ge 1$.

Given $L \ge 1$, 
let  $\ZB_L$ and $\YB_L$
be the solutions to 
\begin{align*}
\begin{cases}
\dt \ZB_L + (1-\Dl ) \ZB_L  = \sqrt 2\xi_L \\
\ZB_L|_{t = 0} = X_{0, L}
\end{cases}
\end{align*}

\noi
and 
\begin{align*}
\begin{cases}
\dt \YB_L + (1-\Dl) \YB_L 
+ \sum_{\l=0}^k {k\choose \ell} :\!\ZB_L^\l\!: \,\YB_L^{k-\l} = 0\\
\YB_L|_{t=0} = 0, 
\end{cases}
\end{align*}

\noi
respectively.
Note that we have $X_L = Y_L + Z_L = \YB_L + \ZB_L$
and that
\begin{align}
Y_L(1) - Y_{L'}(1)
 = \YB_L(1) - \YB_{L'}(1)
 - P(1) (X_{0, L} - X_{0, L'}).
 \label{Z3}
\end{align}

\noi
By the Schauder estimate (Proposition 5 in \cite{MW17a}\footnote{Proposition 3.11 in the arXiv version.})
and the almost sure convergence of $X_{0, L}$ in 
 $B^{ -\eps, \mu}_{p_3, \infty}(\R^2)$, 
we have 
\begin{align}
\|P(1) (X_{0, L} - X_{0, L'})\|_{B^{1 - \frac 2{(\l - m)p_2} + 2\eps}_{p_3, \infty}(B_R)}
\too 0, 
\label{Z3a}
\end{align}

\noi
almost surely,  as $L, L' \to \infty$
(with $L, L' \in \A$).

Next, we study the difference $\YB_L(1) - \YB_{L'}(1)$ in \eqref{Z3}.
Write  $:\!\ZB_L^\l\!:$ as
\begin{align*}
:\!\ZB_L^\l\!:\, 
= \sum_{m=0}^\l {\l\choose m} \big(P(t)X_{0, L}\big)^{\l - m} :\!Z_L^m\!:.
\end{align*}

\noi
Then, 
it follows from 
the almost sure convergence of $P(t)X_{0, L}$ 
and 
Theorem 5.1 in \cite{MW17a}\footnote{Theorem 5.4 in the arXiv version.}
on convergence in probability of $:\!Z_L^m\!:$
together with Lemma \ref{LEM:bilin2}\,(ii)
that 
$:\!\ZB_L^\l\!:$ converges in probability
to a limit 
in $C([0, 1]; B^{-\eps, \mu}_{p_4, \infty}(\R^2))$
for some $p_4 \gg1$, 
as $L \to \infty$ (with $L \in \A$).
Now, 
recall  from 
Theorem 8.1 and Theorem 9.2
 in \cite{MW17a}\footnote{Theorems 8.1 and 9.5
 in the arXiv version.}
 that the solution map
 for the following (deterministic) equation:
\begin{align*}
\begin{cases}
\dt \YB + (1-\Dl) \YB 
+ \sum_{\l=0}^k {k\choose \ell} \ZB^{(\l)} \YB^{k-\l} = 0\\
\YB|_{t=0} = 0, 
\end{cases}
\end{align*}

\noi
sending $\{ \ZB^{(\l)}\}_{\l = 0}^k$ (with $\ZB^{(0)} = 1$)
to a solution $\YB$
is continuous 
from 
$\big( C([0, 1]; B^{-\eps, \mu}_{p_4, \infty}(\R^2))\big)^{\otimes k}$
to  
$C([0, 1]; B^{2-\eps, \mu}_{p_5, \infty}(\R^2))$
for some $p_5 \gg 1$.
In view of the aforementioned convergence in probability of 
 $:\!\ZB_L^\l\!:$, $\l = 1, \dots, k$, 
 the continuous mapping theorem then yields
 that 
 \begin{align}
\|  \YB_L(1) - \YB_{L'}(1)\|_{B^{1 - \frac 2{(\l - m)p_2} + 2\eps}_{p_3, \infty}(B_R)}
\too 0
\label{Z5a}
\end{align}

 \noi
in probability, 
as $L, L' \to \infty$
(with $L, L' \in \A$).

Therefore, putting \eqref{Z3}, \eqref{Z3a}, and \eqref{Z5a}
together with the aforementioned uniform (in $L$) boundedness 
of 
the $B^{1 - \frac 2{(\l - m)p_2} + 2\eps}_{p_3, \infty}(B_R)$-norm of 
$Y_L(1)$ (see the discussion after \eqref{uni3}), 
we conclude from \eqref{uni4} that 
$(  \S(t)Y_L(1) )^{\l - m}  $
converges in probability to 
a unique limit 
in $ W^{\eps, p_2}(B_R)$
as $L \to \infty$ (with $L \in \A)$.
Moreover, by repeating the argument with $L' = \infty$, 
we see that the limit is given by 
$(\S(t)Y_{\infty} (1) )^{\l - m}$, 
where 
 $ Y_\infty$ is the solution to \eqref{SHE4}
with the Gibbsian initial data $ X_{0, \infty}$.

\medskip

\noi
$\bullet$ {\bf Part 2:}
In view of \eqref{uni1} and the discussion above, 
it suffices to show that, given any  $R > 0$,  
there exists $r \ge  1$ such that 
\begin{align}
: \!( \S(t)Z_{L}(1))^m \!: \ \ 
\text{converges  to}
\ \ 
: \!( \S(t)Z_\infty(1))^m \!: 
\label{AS1}
\end{align}
in  $L^r(\O; W^{-\eps, p_1}(B_R))$
as $L \to \infty$, 
$m = 1, \dots, k$.
We point out that in showing \eqref{AS1}, 
there is no need to restrict $L $ in $ \A$. 
Indeed, once we prove \eqref{AS1}, it follows from \eqref{uni1}, Part~1, 
and~\eqref{AS1} (which in particular implies
convergence in probability)
that 
$ :\!  (\S(t) X_L(1))^\l\!: $ converges in law
to 
$ :\!(\S(t) X_{\infty}(1))^\l\!: $ in 
${W^{-\eps, p}(B_R)}$ as $L \to \infty$ (with $L \in \A$)
for any $t \in \R_+$, $R > 0$,  and each $\l = 1, \dots, k$.

\medskip

Hence, it remains to prove \eqref{AS1}.
Fix $t \in \R_+$ and $R> 0$.
Given $L, L' \ge 1$, $x, y \in \R^2$, 
define $K_{L, L'}(x, y;t)$ by 
\begin{align}
\begin{split}
K_{L, L'}(x, y;t)
&  = \E\big[ \big(\S(t)Z_L(1)\big)(x)  \big(\S(t)Z_{L'}(1)\big)(y) \big] \\
& = K_{L, L'}(y, x;t).
\end{split}
\label{KL1}
\end{align}

\noi
Then, from Lemma \ref{LEM:W1}, we have 
\begin{align}
\E\big[ :\!( \S(t)Z_L(1)  )^m (x) \!: \, :\!( \S(t)Z_{L'}(1)  )^m(y)\!: \big]
= m  ! \big\{K_{L, L'}(x, y;t)\big\}^m .
\label{KL1a}
\end{align}

\noi
Given
 a test function $\varphi\in  \D(\R^2)$
supported on a ball $B$ of radius 1 with $B \subset B_{2R}$, 
we set
\begin{align}
K^{m,\varphi}_{L, L'}(t) 
&= 
\E\Big[ |      \jb{
:\!( \S(t)Z_L(1)  )^m\! : -  :\!( \S(t)Z_{L'}(1)  )^m \!: \, , \varphi } |^2 \Big], 
\label{K3a}
\end{align}

\noi
where $\jb{\cdot, \cdot}$ denotes the $\D'$-$\D$ duality pairing.
Then, from \eqref{KL1a} and H\"older's inequality, we have 
\begin{align}
\begin{aligned}
 K^{m,\varphi}_{L, L'}(t) 
&= m  ! \iint_{(B_{2R})^2} \Big(  
\big\{K_{L, L}(x, y;t)\big\}^m - 2\big\{K_{L, L'}(x, y;t)\big\}^m \\
& \hphantom{XXXXXXX}
+ \big\{K_{L', L'}(x, y;t)\big\}^m \Big)
        \varphi(x)\varphi(y) dxdy \\
 &\les  \big\|  \big\{K_{L, L}(x, y;t)\big\}^m - 2\big\{K_{L, L'}(x, y;t)\big\}^m\\
& \hphantom{XXXXXXX}  + \big\{K_{L', L'}(x, y;t)\big\}^m  \big\|_{L^{q}_{x,y}(B_{2R})}
 \| \varphi\|_{L^{q'}}^2   
\end{aligned}
\label{K3}
\end{align}

\noi
for any $1 \le q \le \infty$ with $\frac 1q + \frac 1{q'} = 1$.

Now, set 
$K_{\infty, \infty}(x, y;t) = \lim_{L, L'\to \infty}
K_{L, L'}(x, y;t)$; see \eqref{KK11} below for the precise definition.
We then claim that there exists $c > 0$ such that 
\begin{align}
\begin{aligned}
|K_{L, L'}(x, y;t) | + |K_{\infty, \infty}(x, y;t) | 
& \les  1+ (-  \log  |x-y| )_+  , \\
|  K_{L, L'}(x, y;t)-  K_{\infty, \infty} (x, y;t)|
& \les e^{-cL} + e^{-cL'}
\end{aligned}
\label{KL1b}
\end{align}

\noi
for any $L, L'  \gg R+t$ and $x, y \in B_{2R} $, 
where 
 $a_+ = \max(a, 0)$.

For now, we assume \eqref{KL1b}
and prove \eqref{AS1}.
By the triangle inequality and  \eqref{KL1b}, we have 
\noi
\begin{align}
\begin{aligned}
\big|  & \big\{K_{L, L}(x, y;t)\big\}^m - 2\big\{K_{L, L'}(x, y;t)\big\}^m + \big\{K_{L', L'}(x, y;t)\big\}^m    \big| \\
&\leq \big| \big\{K_{L, L}(x, y;t)\big\}^m - \big\{K_{\infty,\infty}(x, y;t)\big\}^m  \big|\\
& \quad  + 2 \big|  \big\{K_{L, L'}(x, y;t)\big\}^m  - \big\{K_{\infty,\infty}(x, y;t)\big\}^m \big|   \\
&\quad +  \big|  \big\{K_{L', L'}(x, y;t)\big\}^m  -  \big\{K_{\infty,\infty}(x, y;t)\big\}^m   \big| \\
&\les \big( 1 +  (-\log |x-y| )_+ \big)^{m -1}
 \big( e^{-c L} + e^{-c L'} \big).
\end{aligned}
\label{K4}
\end{align}

\noi
Hence, 
by choosing $q = q(\eps) \gg1 $
and applying Lemma \ref{LEM:Kol}
with 
\eqref{K3a}, \eqref{K3}, and \eqref{K4}, we obtain
\begin{align}
\begin{split}
\E & \Big[ \|  :\!( \S(t)Z_L(1)  )^m \!: -  :\!( \S(t)Z_{L'}(1)  )^m \! : \|_{W^{-\eps, p_1}(B_R)}^r \Big] \\
&
\le C(R, r)  \sup_{ \substack{ \| \varphi\|_{L^{q'}} \le 1 \\ \supp(\varphi)\subset B \subset B_{2R}}}
\big(K^{m, \varphi}_{L, L'}(t)\big)^\frac r2\\
& \les C(R, r)  
 \big( e^{-c L} + e^{-c L'} \big), 
 \end{split}
\label{K4a}
\end{align}
 
 \noi
which tends to zero as $L, L'\to\infty$.
This shows that 
$: \!( \S(t)Z_{L}(1))^m \!: $
converges  to
some limit 
in  $L^r(\O; W^{-\eps, p_1}(B_R))$
as $L \to \infty$.
By repeating an analogous argument with $L' = \infty$, 
we obtain~\eqref{AS1}.
See Remark \ref{REM:XXX} below.

\medskip

The remaining part of the proof is devoted to proving the claim \eqref{KL1b};
see Remark~\ref{REM:XXX} below
for the case $L' = \infty$.
Let $\varphi \in L^2(\R^2)$.
Then, from \eqref{SS1a} with \eqref{wh4} and  \eqref{wh5}, we have
\begin{align*}
\jb{\S(t) Z_L(1), \varphi}
= \sqrt 2 
\xi \Big(  \ind_{[0, 1]}(t')\ind_{[-\frac L2, \frac L2)^2}  
 P(1- t') \S(t) \varphi^\per_L \Big), 
\end{align*}

\noi
where 
$ \varphi^\per_L $ is the $L$-periodized version of $\varphi$ defined by 
\begin{align} \varphi^\per_L (x) 
= \sum_{n \in \Z^2} \varphi(x + n L).
\label{KK1a}
\end{align}

\noi
Given $x \in \R^2$, we then have 
\begin{align*}
\big(\S(t) Z_L(1)\big)(x) 
& = \jb{\S(t) Z_L(1), \dl_x}\\
& = \sqrt 2 
\xi \Big(  \ind_{[0, 1]}(t')\ind_{[-\frac L2, \frac L2)^2}  
 P(1- t') \S(t) (\dl_x)^\per_L \Big), 
\end{align*}

\noi
where $\dl_x$ denotes the Dirac delta function at $x \in \R^2$, 
and hence we can rewrite \eqref{KL1} as 
\begin{align}
\begin{aligned}
K_{L, L'}(x, y;t)
&=2\E\bigg[  
\xi \Big(  \ind_{[0, 1]}(t')\ind_{[-\frac L2, \frac L2)^2}  
 P(1- t') \S(t) (\dl_x)^\per_L \Big)\\
& \hphantom{XXX}
 \times 
 \xi \Big(  \ind_{[0, 1]}(t')
 \ind_{[-\frac {L'}2, \frac {L'}2)^2}  
 P(1- t') \S(t) (\dl_y)^\per_{L'} \Big)
\bigg]\\
&=2 \int_{\R^2}\int_0^1
 \Big( \ind_{[-\frac L2, \frac L2)^2}  
 P(t') \S(t) (\dl_x)^\per_L \Big)(z)\\
& \hphantom{XXX}
 \times 
 \Big( 
 \ind_{[-\frac {L'}2, \frac {L'}2)^2}  
 P(t') \S(t) (\dl_y)^\per_{L'} \Big)(z)dt'  dz .
\end{aligned}
\label{KK3}
\end{align}

In the following, 
we  assume  $x, y \in B_{2R}$.
By the finite speed of propagation, we have
\begin{align}
\supp  \S(t) \dl_{x+ nL} 
\subset 
 x + nL + B_t.
\label{KK3a}
\end{align}

\noi
Let $\{\psi_j^L\}_{j \in \Z_{\ge 0}}$ be 
a smooth partition of unity on $\R^2$
such that 
$\psi_0^L =  1$ on $B_{\frac L8}$, 
$\supp \psi_0^L  \subset B_{\frac{L}{4}}$, 
and 
\begin{align}
\supp \psi_j^L \subset \big\{ \max\big((j - 1), \tfrac 18\big)L  \le |z| \le (j + 1) L\big\}.
\label{KK3b}
\end{align}

\noi
In particular, for $j = 0$, 
we set 
$\psi_0^L(z) = \psi_0\big(\frac zL\big)$ for some $\psi_0 \in C_c^\infty(\R^2)$.
Here, we choose 
$\{\psi_j^L\}_{j \in \Z_{\ge 0}}$ such that their first and second derivatives are
uniformly bounded.
Then, we have
\begin{align*}
\ind_{[-\frac {L}2, \frac {L}2)^2}  
 P(t') \S(t) (\dl_x)^\per_{L} 
 = e^{-t'} \sum_{j = 0}^\infty
\ind_{[-\frac {L}2, \frac {L}2)^2}  
\big\{
(\psi_j^L p_{t'})* (\S(t) (\dl_x)^\per_{L})\big\}, 
\end{align*}

\noi
where $p_{t'}$ is as in \eqref{heat1}.
Under the assumption $L \gg R+t$ and $x \in B_{2R}$, 
it follows from~\eqref{KK1a} and \eqref{KK3a} that 
\begin{align}
\begin{split}
e^{-t'}\ind_{[-\frac {L}2, \frac {L}2)^2}  (z)
\cdot \big\{
(\psi_0^L p_{t'})* (\S(t) (\dl_x)^\per_{L})\big\}(z)
& = e^{-t'}
(\psi_0^L p_{t'})* (\S(t) \dl_x)(z)\\
& = e^{-t'}
\S(t)(\psi_0^L p_{t'})(z-x).
\end{split}
\label{KK5}
\end{align}

\noi
By the boundedness of $\S(t)$ on $H^s(\R^2)$ and 
$\psi_0^L(z) = \psi_0\big(\frac zL\big)$ for some $\psi_0 \in C_c^\infty(\R^2)$, we have
\begin{align}
\begin{split}
\|\eqref{KK5}\|_{H^{-1-\eps}_z}
& \les  e^{-t'} \| \psi_0^L p_{t'}\|_{H^{-1-\eps}}\\
& \les e^{-t'}\bigg\|\int_{\eta = \eta_1 + \eta_2} L^2 \ft \psi_0(L\eta_1) e^{-4\pi^2 t' |\eta_2|^2} d\eta_1\bigg\|_{L^\infty_\eta}\\
& \les e^{-t'}\int_{\R^2} \frac{L^2}{\jb{L\eta_1}^{100}} d\eta_1
\les e^{-t'}, 
\end{split}
\label{KK5a}
\end{align}

\noi
uniformly in $x \in B_{2R}$, $t \ge 0$, and $ t' \in [0, 1]$
with $R + t \ll L$.

For $j \in \N$, it follows from \eqref{KK1a} and \eqref{KK3b} 
with $L \gg R+t$ and $x \in B_{2R}$ that 
\begin{align}
\begin{aligned}
&e^{-t'}\ind_{[-\frac {L}2, \frac {L}2)^2}  (z) 
\cdot 
\big\{
(\psi_j^L p_{t'})* (\S(t) (\dl_x)^\per_{L})\big\}(z) \\
&\quad 
=
e^{-t'}\ind_{[-\frac {L}2, \frac {L}2)^2}  (z) 
 \sum_{\substack{n \in \Z^2\\ |n| = j+O(1) }}  \bigg\{
(\psi_j^L p_{t'})* \Big(\S(t)   \dl_{x+n L}\Big)\bigg\}(z) \\
& \quad
= e^{-t'}\ind_{[-\frac {L}2, \frac {L}2)^2}  (z) 
\bigg\{ \sum_{\substack{n \in \Z^2\setminus\{0\}\\ |n| = j+O(1) }}  \S(t)( \psi_j^L p_{t'})(z-x- nL)\\
& \hphantom{XXXXXXXXXXX}
 +  \ind_{\{|z|\ge \frac L{16}\}}
 \S(t)( \psi_j^L p_{t'})(z-x)\bigg\}, 
\end{aligned}
\label{KK6}
\end{align}

\noi
where the second term of the last factor is relevant only for $j = 1$.
Note that, for 
$|z|\ge \frac L{16}$ and 
$x \in B_{2R}$ with $L \gg R+t$, 
we have $|z-x|\ge \frac L{32}$.
From the definition of $\psi_j^L$ and~\eqref{heat1}, 
we then have 
\begin{align*}
\| \eqref{KK6} \|_{L^2_z}
&\les j  e^{-t'} \| \psi_j^L  p_{t'} \|_{L^2_z(|z| \ge \frac{L}{32})}\\
&  \les 
 \begin{cases}
 \exp\big(- c \frac{j^2 L^2}{t'} - t' \big),  &\text{for $0 < t'\ll jL$},  \\
 \exp(- cjL),  & \text{for $t'\sim j L$},   \\
 \exp(-ct'),  & \text{for $t' \gg j L$} .
 \end{cases}
\end{align*}

\noi
Thus, we obtain 
\begin{align}
\| \eqref{KK6} \|_{L^2_z}
\les e^{-cjL}, 
\label{KK7}
\end{align}

\noi
uniformly in $x \in B_{2R}$, $t \ge 0$, and $ t' \in [0, 1]$
with $R + t \ll L$.

Let $\chi_L \in C^\infty_c(\R^2; [0, 1])$
be a smooth cutoff function
such that $\chi_L \equiv 1$ on $\big[-\frac {L}2, \frac {L}2\big)^2$ 
and $\supp \chi_L \subset  \big[-\frac {(1+ \g)L}2, \frac {(1+\g)L}2\big)^2$
for some small $\g > 0$.
We assume that 
the first and second derivatives of $\chi_L$ are
bounded uniformly in $L\gg 1$.
Define  $F(z) = F_{j, L, x, t, t'}(z)$by 
\begin{align*}
F(z) = F_{j, L, x, t, t'}(z)   =  e^{-t'}\chi(z)\cdot 
\big\{
(\psi_j^L p_{t'})* (\S(t) (\dl_x)^\per_{L})\big\}(z).
\end{align*}

\noi
Then, a similar computation
with the uniform boundedness of the first and second derivatives of~$\psi_j^L$
shows that 
\begin{align}
\| F_{j, L, x, t, t'} \|_{H^2_z}
\les \| F_{j, L, x, t, t'} \|_{L^2_z}
+ \|\Dl_z  F_{j, L, x, t, t'} \|_{L^2_z}
\les e^{-cjL}, 
\label{KK7a}
\end{align}

\noi
uniformly in $x \in B_{2R}$, $t \ge 0$, and $ t' \in [0, 1]$
with $R + t \ll L$.

In view of the discussion above
(in particular, see \eqref{KK5}), 
we write 
$K_{L, L'}(x,y; t)$ in \eqref{KK3}
as
\begin{align}
K_{L, L'}(x,y; t) = K_{L, L'}^{(0)}(x,y; t) + K_{L, L'}^{(1)}(x,y; t),
\label{KK8}
\end{align}

\noi
where
$ K_{L, L'}^{(0)}(x,y; t)$ is defined by 
\begin{align}
\begin{aligned}
K_{L, L'}^{(0)}(x,y; t)
&=2 \int_{\R^2}\int_0^1
 \Big( 
 e^{-t'} (\psi_0^L p_{t'})* (\S(t) \dl_x) \Big)(z)\\
& \hphantom{XXX}
 \times 
 \Big( 
 e^{-t'} (\psi_0^{L'} p_{t'})*
( \S(t) \dl_y) \Big)(z)dt'  dz 
\end{aligned}
\label{KK9}
\end{align}

\noi
and 
$K_{L, L'}^{(1)}( x,y; t) = 
K_{L, L'}(x,y; t) -  K_{L, L'}^{(0)}(x,y; t)$.
Then, from \eqref{KK3} and \eqref{KK9}, we have
\begin{align}
K_{L, L'}^{(1)}(x,y; t)
&=2 \sum_{j = 1}^\infty\sum_{j' = 1}^\infty
\int_{\R^2}\int_0^1
 \Big( \ind_{[-\frac L2, \frac L2)^2}  
\, e^{-t'} (\psi_j^L p_{t'})* (\S(t) (\dl_x)^\per_L) \Big)(z)\notag\\
& \hphantom{XXXXXXX}
 \times 
 \Big( 
 \ind_{[-\frac {L'}2, \frac {L'}2)^2}  
\, e^{-t'} (\psi_{j'}^{L'} p_{t'})*
( \S(t) (\dl_y)^\per_{L'}) \Big)(z)dt'  dz \notag\\
&\quad +
2 \sum_{j = 1}^\infty
\int_{\R^2}\int_0^1
 \Big( \ind_{[-\frac L2, \frac L2)^2}  
 e^{-t'} (\psi_j^L p_{t'})* (\S(t) (\dl_x)^\per_L) \Big)(z)\notag\\
& \hphantom{XXXXXXX}
 \times 
 \Big( 
 e^{-t'} (\psi_{0}^{L'} p_{t'})*
( \S(t) \dl_y) \Big)(z)dt'  dz \notag\\
&\quad + 2 \sum_{j' = 1}^\infty
\int_{\R^2}\int_0^1
 \Big(
 e^{-t'} (\psi_0^{L} p_{t'})* (\S(t) \dl_x) \Big)(z)\notag\\
& \hphantom{XXXXXXX}
 \times 
 \Big( 
  \ind_{[-\frac {L'}2, \frac {L'}2)^2}  
 e^{-t'} (\psi_{j'}^{L'} p_{t'})*
( \S(t) (\dl_y)^\per_{L'}) \Big)(z)dt'  dz \notag\\
& = : \1 (x,y; t)+ \II(x,y; t) + \III(x,y; t).
\label{KK9a}
\end{align}

\noi
Note that, in view of \eqref{KK5}, 
we dropped the cutoff function $\ind_{[-\frac {L'}2, \frac {L'}2)^2}$
(and $\ind_{[-\frac {L}2, \frac {L}2)^2}$)
in $\II$ (and in $\III$, respectively).
By Cauchy-Schwarz's inequality with \eqref{KK6} and \eqref{KK7}, we have
\begin{align}
|\1(x,y; t)|
& \les  \sum_{j = 1}^\infty\sum_{j' = 1}^\infty e^{- cj L} e^{- cj'L'}
 \les  e^{- c( L+L')}.
\label{KK10}
\end{align}

Next, we estimate $\II$ and $\III$ in \eqref{KK9a}.
Without loss of generality, assume $L \ge  L'$.
As for $\II$, 
we first insert $\chi_L$
and then 
drop the cutoff function $\ind_{[-\frac L2, \frac L2)^2}$
in view of 
 \eqref{KK5}. 
Thus, from~\eqref{KK5a} and \eqref{KK7a}, we have 
\begin{align}
\begin{split}
|\II (x, y; t)|
& = 
2 \bigg| \sum_{j = 1}^\infty
\int_{\R^2}\int_0^1
 \Big( 
 e^{-t'} \chi_L(z)(\psi_j^L p_{t'})* (\S(t) (\dl_x)^\per_L) \Big)(z)\\
& \hphantom{XXXXXXX}
 \times 
 \Big( 
 e^{-t'} (\psi_{0}^{L'} p_{t'})*
( \S(t) \dl_y) \Big)(z)dt'  dz \bigg|\\
& \les
 \sum_{j = 1}^\infty
\int_0^1
\big\| 
 e^{-t'} \chi_L (\psi_j^L p_{t'})* (\S(t) (\dl_x)^\per_L) \big\|_{H^2_z}\\
& \hphantom{XXXXXXX}
 \times 
\|
 e^{-t'} (\psi_{0}^{L'} p_{t'})*
( \S(t) \dl_y) \|_{H^{-1-\eps}_z}dt'  \\
& \les e^{-cL}.
\end{split}
\label{KK10a}
\end{align}

\noi
As for the term $\III$, 
we can not drop the cutoff function 
$  \ind_{[-\frac {L'}2, \frac {L'}2)^2}  $
under the current assumption $L \ge L'$, 
and thus we need to proceed with more care.
We first note that by the mean value theorem, we have
\begin{align}
|\F\big(  \ind_{[-\frac {L'}2, \frac {L'}2)^2}  \big)(\eta)|
\les \frac{(L')^2}{\jb{\eta^1}\jb{\eta^2}}, 
\label{KB1}
\end{align}

\noi
where $\eta = (\eta^1, \eta^2)$.
Then, proceeding as in \eqref{KK5a}
with the triangle inequality $\jb{\eta_1}^\eps 
\les \jb{\eta}^\eps\jb{\eta_2}^\eps\jb{\eta_3}^\eps$, 
Young's inequality, and 
 \eqref{KB1}, we have
\begin{align}
\begin{split}
\big\| &  \ind_{[-\frac {L'}2, \frac {L'}2)^2}(z)
\S(t)(\psi_0^L p_{t'})(z-x)\big\|_{H^{-1-2\eps}_z}\\
& \les
\bigg\| 
\frac{1}{\jb{\eta}^{1 + 2\eps}}
\int_{\eta = \eta_1 + \eta_2+ \eta_3} 
|\F\big(  \ind_{[-\frac {L'}2, \frac {L'}2)^2}  \big)(\eta_1)|
L^2|\ft \psi_0(L\eta_2)| e^{-4\pi^2 t' |\eta_3|^2} d\eta_1d\eta_2  \bigg\|_{L^2_\eta}\\
& \les
\bigg\| 
\frac{1}{\jb{\eta}^{1 + \eps}}
\int_{\eta = \eta_1 + \eta_2+ \eta_3} 
\jb{\eta_1}^{-\eps}|\F\big(  \ind_{[-\frac {L'}2, \frac {L'}2)^2}  \big)(\eta_1)|\\
& \hphantom{XXXXXXXXXXX}
\times \jb{\eta_2}^{\eps} L^2|\ft \psi_0(L\eta_2)| 
\jb{\eta_3}^{\eps}
e^{-4\pi^2 t' |\eta_3|^2} d\eta_1d\eta_2 
\bigg\|_{L^2_\eta}\\
& \les (L')^2 (t')^{-\frac \eps 2}.
\end{split}
\label{KB2}
\end{align}

\noi
Hence, from \eqref{KK7a} and \eqref{KB2}, we obtain
\begin{align}
\begin{split}
|\III(x,y; t)|
& = 2\bigg| \sum_{j' = 1}^\infty
\int_{\R^2}\int_0^1
 e^{-t'}
  \ind_{[-\frac {L'}2, \frac {L'}2)^2}(z)
\S(t)(\psi_0^L p_{t'})(z-x)\\
& \hphantom{XXXXXXX}
 \times 
 \Big( 
 e^{-t'} 
 \chi_{L'}(z)
 (\psi_{j'}^{L'} p_{t'})*
( \S(t) (\dl_y)^\per_{L'}) \Big)(z)dt'  dz\bigg| \\
& \les (L')^2 
e^{-cL'}
\int_0^1 (t')^{-\frac \eps 2} dt' \\
& \les e^{-c'L'}.
\end{split}
\label{KK10b}
\end{align}

%

\noi
Therefore, from \eqref{KK9a}, \eqref{KK10}, \eqref{KK10a}, and \eqref{KK10b}, 
we obtain
\begin{align}
|K_{L, L'}^{(1)}(x,y; t)|
\les e^{-cL} +  e^{-cL'}.
\label{KK10c}
\end{align}

Next, we study $K^{(0)}_{L, L'}(x, y;t)  $
for $L, L' \gg 1$.
Define $K_{\infty, \infty}(x, y; t)  $ by 
\begin{align}
\begin{aligned}
K_{\infty, \infty}(x,y; t)
&=2 \int_{\R^2}  \int_0^1
  \big( P(t') (\S(t) \dl_x \big)(z)
 \big(P(t') \S(t) \dl_y \big)(z)dt'  dz \\
& =  
\big\langle(1-\Dl)^{-1}  (1- e^{-2(1-\Dl)})
 \S(t)    \dl_{x},   \S(t)     \dl_{y} \big\rangle_{L^2_z}.
\end{aligned}
\label{KK11}
\end{align}

\noi
From \eqref{KK9} and \eqref{KK11}, we have
\begin{align}
\begin{split}
&  K_{\infty, \infty} (x,y; t)
- K_{L, L'}^{(0)}(x,y; t)\\
& \quad 
 =  2 \int_{\R^2}\int_0^1
 \Big( 
 e^{-t'} \big((1- \psi_0^L) p_{t'}\big)* (\S(t) \dl_x) \Big)(z)
 \\
& \quad \hphantom{XXXXXX}
 \times 
 \Big(  P(t')
 \S(t) \dl_y \Big)(z)dt'  dz \\
&\quad 
\quad + 2 \int_{\R^2}\int_0^1
 \Big( 
 e^{-t'} (\psi_0^L p_{t'})* (\S(t) \dl_x) \Big)(z)
 \\
& \quad \hphantom{XXXXXX}
 \times 
 \Big( 
 e^{-t'} \big((1-\psi_0^{L'}) p_{t'}\big)*
( \S(t) \dl_y) \Big)(z)dt'  dz \\
& \quad
= 
  2 \int_{\R^2}\int_0^1
  \Big( 
 e^{-t'}\S(t) \big((1- \psi_0^L) p_{t'}\big) \Big)(z-x)
 \\
& \quad \hphantom{XXXXXX}
 \times 
 \Big(  e^{-t'}
 \S(t) p_{t'}\Big)(z-y)dt'  dz \\
&\quad 
\quad + 2 \int_{\R^2}\int_0^1
 \Big( 
e^{-t'} \S(t)(\psi_0^L p_t')  \Big)(z-x)
 \\
& \quad \hphantom{XXXXXX}
 \times 
 \Big( 
 e^{-t'} \S(t) \big((1-\psi_0^{L'}) p_{t'}\big)
 \Big)(z-y)dt'  dz \\
& \quad =: \IV(x, y; t) + \text{V}(x, y; t).
\end{split}
\label{KK12}
\end{align}

\noi
We only consider the first term $\IV$ on the right-hand side of \eqref{KK12}
since  the term $\text{V}$ can be treated in a similar manner.
Recalling that $\psi_0^L =  1$ on $B_{\frac L8}$, we have 
\begin{align}
\| e^{-t'}(1- \psi_0^L )p_{t'}\|_{L^2_x} \les e^{- c\frac{L^2}{t'} -t' }
\les 
e^{-cL}, 
\label{KK13}
\end{align}

\noi
where the second inequality follows from separately estimating
the cases (i)~$t' \ll L$, (ii)~$t' \sim L$, and (iii)~$t' \gg L$.
Moreover, since $\psi_0^L = 1$ on $B_{\frac L8}$, $L \gg R + t$, 
and $x, y \in B_{2R}$, we have
\begin{align*}
|z-y|
\ge |z- x| - |x| - |y|
 \ge \frac L 8 - t - 4R \ge \frac L{16}
\end{align*}

\noi
in the domain of integration for $\IV$ in \eqref{KK12}.
Under this condition, we have, as in \eqref{KK7}, 
\begin{align}
\| \ind_{\{|z- y|\ge \frac L{16}\}}e^{-t'}
 \S(t) p_{t'}(z-y) \|_{L^2_z}
 \les e^{-cL}.
\label{KK15}
\end{align}

\noi
Then, 
by applying Cauchy-Schwarz's inequality
with  \eqref{KK13} and  \eqref{KK15}, we obtain
\begin{align}
|\IV(x, y; t)| \les e^{-cL}.
\label{KK16}
\end{align}

\noi
A similar computation yields
\begin{align}
|\text{V}(x, y; t)| \les e^{-cL'}.
\label{KK17}
\end{align}

\noi
Therefore, from \eqref{KK12}, \eqref{KK16}, and \eqref{KK17}, 
we conclude that 
\begin{align}
|  K_{\infty, \infty} (x,y; t)
- K_{L, L'}^{(0)}(x,y; t)|
\les e^{-cL} + e^{-cL'}.
\label{KK18}
\end{align}


It remains to estimate $K_{\infty, \infty} (x,y; t)$.
In view of \eqref{KK11} with \eqref{SS1}, we have
\begin{align}
\begin{aligned}
K_{\infty, \infty}(x,y; t)
&=K^t(x-y) \\
: \! & = \int_{\R^2} \frac{(\ft{\S(t)}(\eta))^2}{\jb{2\pi\eta}^2}(1 - e^{-\jb{2\pi\eta}^2}) e^{2\pi i \eta \cdot (x - y)} d \eta.
\end{aligned}
\label{KZ1}
\end{align}

\noi
From \eqref{SS1} with \eqref{SS0}, we have 
\begin{align}
\big |(\ft{\S(t)}(\eta))^2(1 - e^{-\jb{2\pi\eta}^2})
 - e^{-t} \cos^2( t \jbb{2\pi \eta } )\big |\les \frac{1}{\jb{\eta}}, 
\label{KZ2}
\end{align}

\noi
uniformly in $t \in \R_+$.
Then, by defining $K^t_1$ by 
\begin{align}
K^t_1(z)=  e^{-t} \int_{\R^2}  \frac{  \cos^2(t\jbb{2\pi\eta})    }{\jb{2\pi \eta}^2} 
e^{2\pi i\eta \cdot z}d\eta,
\label{KZ3}
\end{align}

\noi
it follows from \eqref{KZ1} and \eqref{KZ2} that 
\begin{align}
|K_{\infty, \infty}(x,y; t) - K_1^t(x- y)|\les 1,
\label{KZ4}
\end{align}

\noi
uniformly in $t \in \R_+$ and $x, y \in \R^2$.

We now estimate $K_1^t$.
Recall from  Appendix B.5 in  \cite{Gra14}
that the Fourier transform of a radial function $F(x) = f(|x|)$ on $\R^2$
is given by 
\begin{align}
\ft F(\eta) = 2\pi \int_0^\infty f(r) J_0 (2\pi r|\eta|) r dr, 
\label{KZ7}
\end{align}

\noi
where $J_0$ 
is  the Bessel function of  order zero, satisfying 
the following asymptotics as $r \to \infty$:
\begin{align}
J_0(r) =  \sqrt{\frac{2}{\pi r}} \cos\Big(r-  \frac\pi4 \Big) + O(r^{-\frac 32}).
\label{asymp1}
\end{align}

\noi
See Appendix B.8 in  \cite{Gra14}
for the asymptotics \eqref{asymp1}.
Thus, from \eqref{KZ3} and \eqref{KZ7}, we have
\begin{align}
K_1^t(z)
= 2\pi e^{-t}  \int_0^\infty
 \frac{  \cos^2( t \jbb{2\pi r})   }{ \jb{2\pi r}^2}  J_0(2\pi r|z| ) rdr.
\label{KZ8}
\end{align}

\noi
Recalling that 
$J_0$ is  bounded on $\R_+$, it follows from \eqref{KZ8} with \eqref{asymp1}
that 
\begin{align}
\begin{aligned}
|K_1^t(z) |
&\les e^{-t} \int_0^{ \max( \frac{1}{|z|} , 1) } \frac{1}{\jb{r}} dr
+ \int_{ \max( \frac{1}{|z|} , 1) }^\infty  \frac{1}{r^{\frac32} |z|^{\frac12}} dr  \\
&\les  1+ (-  \log  |z| )_+  .
\end{aligned}
\label{KZ9}
\end{align}

\noi
Hence, putting together
\eqref{KK8}, \eqref{KK10c}, 
\eqref{KK18}, \eqref{KZ4},  and \eqref{KZ9}, we obtain
\begin{align}
|K_{L, L'}(x,y; t) | +   |K_{\infty, \infty}(x,y; t) | 
\les 1+  (-  \log  |x-y| )_+  .
\label{KZ10}
\end{align}

\noi
Therefore, the claim 
 \eqref{KL1b} follows
 from  \eqref{KK8},  \eqref{KK10c},  \eqref{KK18}, 
 and 
\eqref{KZ10}.
\end{proof}

\begin{remark}\label{REM:XXX}
\rm
When $L' = \infty$, 
we write
\begin{align}
K_{L, \infty}(x,y; t) 
= K_{L, \infty}^{(0)}( x,y; t) + 
K_{L, \infty}^{(1)}( x,y; t),
\label{KX0}
\end{align} 
where
\begin{align}
\begin{aligned}
K_{L, \infty}^{(0)}(x,y; t)
&=2 \int_{\R^2}\int_0^1
 \Big( 
 e^{-t'} (\psi_0^L p_{t'})* (\S(t) \dl_x) \Big)(z)\\
& \hphantom{XXXXX}
 \times 
 \big(P(t') \S(t) \dl_y \big)(z)dt'  dz 
\end{aligned}
\label{KX1}
\end{align}

\noi
and 
\begin{align*}
K_{L, \infty}^{(1)}(x,y; t)
&=2 \sum_{j = 1}^\infty
\int_{\R^2}\int_0^1
 \Big( \ind_{[-\frac L2, \frac L2)^2}  
\, e^{-t'} (\psi_j^L p_{t'})* (\S(t) (\dl_x)^\per_L) \Big)(z)\notag\\
& \hphantom{XXXXXXX}
 \times 
 \big(P(t') \S(t) \dl_y \big)(z)dt'  dz .
\end{align*}

\noi
Proceeding as in \eqref{KB2}, we have
\begin{align*}
\big\|   \ind_{[-\frac {L}2, \frac {L}2)^2}(z)
 \big(P(t') \S(t) \dl_y \big)(z)\big\|_{H^{-1-\eps}_z}
 \les L^2 (t')^{-\frac \eps 2}.
\end{align*}

\noi
Then, proceeding as in \eqref{KK10b}, we obtain
\begin{align}
|K_{L, \infty}^{(1)}(x,y; t)| \les e^{-cL}.
\label{KX2}
\end{align}

\noi
From \eqref{KK11} and \eqref{KX1}, we have
\begin{align*}
&  K_{\infty, \infty} (x,y; t)
- K_{L, \infty}^{(0)}(x,y; t)\\
& \quad 
 =  2 \int_{\R^2}\int_0^1
 \Big( 
 e^{-t'} \big((1- \psi_0^L) p_{t'}\big)* (\S(t) \dl_x) \Big)(z)
 \\
& \quad \hphantom{XXXXXX}
 \times 
 \Big(  P(t')
 \S(t) \dl_y \Big)(z)dt'  dz, 
\end{align*}

\noi
which is precisely 
 $\IV(x, y; t) $ in \eqref{KK12}.
Hence, from \eqref{KK16}, we have 
\begin{align}
|  K_{\infty, \infty} (x,y; t)
- K_{L, \infty}^{(0)}(x,y; t)|
\les e^{-cL}.
\label{KX3}
\end{align}

\noi
Therefore, putting \eqref{KX0}, 
\eqref{KX2}, and \eqref{KX3} together, 
we obtain
\begin{align}
|  K_{\infty, \infty} (x,y; t)
- K_{L, \infty}(x,y; t)|
\les e^{-cL}.
\label{KX4}
\end{align}

\noi
On the other hand, from 
\eqref{KL1b} and \eqref{KX4}, we have 
\begin{align}
\begin{split}
|K_{L, L}(x, y;t) | + |K_{L, \infty}(x, y;t) | + |K_{\infty, \infty}(x, y;t) | 
& \les  1+ (-  \log  |x-y| )_+  , \\
|  K_{L, L}(x, y;t)-  K_{\infty, \infty} (x, y;t)|
& \les e^{-cL}.
\end{split}
\label{KX5}
\end{align}

A computation analogous to 
\eqref{K4a}
with  Lemma \ref{LEM:Kol}, 
\eqref{K4} (with $L' = \infty$), \eqref{KX4}, and \eqref{KX5} yields
\begin{align*}
\E & \Big[ \|  :\!( \S(t)Z_L(1)  )^m \!: -  :\!( \S(t)Z_{\infty}(1)  )^m \! : \|_{W^{-\eps, p_1}(B_R)}^r \Big] \\
& \les C(R, r)  
 e^{-c L} \too 0, 
\end{align*}
 
\noi
as $L \to \infty$.
This proves \eqref{AS1}.

\end{remark}

\subsection{Estimates
on the enhanced data sets}
\label{SUBSEC:4.3}

In this subsection,
we establish
the following proposition 
on  regularities of the enhanced data sets.

\begin{proposition} \label{PROP:need23}
{\rm (i)} 
Let $R > 0$.
Given small $\eps > 0$, there exists finite $p = p(\eps) \gg1 $ such that 
the following holds true.
Suppose  that 
random $k$-tuple of functions $\Xi_0, \wt \Xi_0 \in \WW(R)$   are independent of $\Phi_1 = \Phi_1(u_1, \xi)$ 
defined in \eqref{NW5}, 
where $ \WW(R)$ is as in \eqref{norm2}.
Then,  we have 

 \noi
\begin{align*}
\E_{u_1, \xi, \Xi_0, \wt \Xi_0}\Big[  \| \Xi_1(u_1,  \xi) \| _{\Xi_0, \wt \Xi_0, R} \Big]  \les 1, 
\end{align*}

\noi
where   the implicit constant is independent of the distributions 
of $\Xi_0$ and $\wt \Xi_0$.
Here,    $\Xi_1(u_1,  \xi)$ is as in \eqref{xi2}
and $\| \cdot  \| _{\Xi_0, \wt \Xi_0, R}$ is as in 
Definition \ref{DEF:X}\,(iii).

\smallskip

\noi
{\rm(ii)}   Let $\Xi_0(\phi)$ be as in \eqref{xi1}. Then,
given any $\eps > 0$, finite $p \gg 1$, and $\mu > 0$, 
we have
\begin{align}
 \sup_{L\in \A}\E_{\rho_L}
 \Big[  \| \Xi_0( \phi) \|_{(L^\infty_t W^{-\eps, p}_\mu)^{\otimes k} } \Big] 
 <\infty.
 \label{need3}
\end{align}

\end{proposition}

We first state and prove an auxiliary lemma.
The proof of  Proposition \ref{PROP:need23}
is presented at the end of this subsection.

\begin{lemma} \label{LEM:H}
Let  $\Phi_1 = \Phi_1(u_1, \xi)$ be as in \eqref{NW5}, 
where  $u_1$ and $\xi$ 
denote the independent spatial and space-time white noises  as in \eqref{NW1}.
Define $H(x, y; t)$ by 
\begin{align}
H(x, y; t) = \E\big[ \Phi_1(x, t) \Phi_1(y, t)\big].
\label{H1}
\end{align}

\noi
Let $R > 0$.
Then,   given small $\eps > 0$
and  $1 \le q < \frac 1\eps$, 
there exists a finite constant $C(R) > 0 $ such that 
\begin{align}
\|  \jb{\nb_x}^\eps \jb{\nb_y}^\eps H(x, y; t) \|_{L^{q}_{x, y}(B_R\times B_R)}  \le C(R), 
\label{H2}
\end{align}
uniformly in $t \in \R_+$.
\end{lemma}

\begin{proof}
With a slight abuse of notation, it follows
from  Definition \ref{DEF:white} and an analogous property 
for the space-time white noise $\xi$ that 
\begin{align*}
H( x-y;t ) & \equiv H(x, y; t) \\
&= \jb{\D(t) \dl_x,  \D(t) \dl_y}_{L^2(\R^2)} + \int_0^t \jb{ \D(t- t') \dl_x,  \D(t- t') \dl_y }_{L^2(\R^2)} dt'.
\end{align*}

\noi
Then, we have
\begin{align}
\jb{\nb_x}^\eps\jb{\nb_y}^\eps H( x-y;t )
= \jb{\nb_z}^{2\eps} H(z; t)|_{z = x-y}.
\label{H5}
\end{align}

Define 
$\wt H^t(z)$ by 
\begin{align*}
\wt{H}^t(z) =  \int_{\R^2} \frac{\sin^2(t \jbb{2\pi \eta})}{\jbb{2\pi \eta}^2} e^{2\pi  i\eta\cdot z} d\eta.
\end{align*}

\noi
Then, from \eqref{lin2}, we have
\begin{align}
H(z; t) 
&=e^{-t}  \wt{H}^t(z)
+   \int_0^t e^{-t'} 
\wt{H}^{t'}(z)  dt' .
\label{H4}
\end{align}

\noi
In view of \eqref{H5} and \eqref{H4}, 
it suffices to study
\[
\jb{\nb}^{2\eps} \wt{H}^t(z) =  
\int_{\R^2}\jb{\eta}^{2\eps} \frac{\sin^2(t \jbb{2\pi \eta})}{\jbb{2\pi \eta}^2} e^{2\pi  i\eta\cdot z} d\eta.
\]

\noi
As the inverse Fourier transform of a radial function, 
$\jb{\nb}^{2\eps} \wt{H}^t$ is radial.
Hence, from \eqref{KZ7}, we have
\begin{align*}
\jb{\nb}^{2\eps} \wt{H}^t(z) =  
2\pi \int_0^\infty
\jb{r}^{2\eps} \frac{\sin^2(t \jbb{2\pi r})}{\jbb{2\pi r}^2} J_0(2\pi r|z|) r dr, 
\end{align*}

\noi
where $J_0$ 
is  the Bessel function of  order zero.
Using the boundedness of $J_0$ on $\R_+$ and \eqref{asymp1}, 
we then have
\begin{align}
\begin{split}
|\jb{\nb}^{2\eps} \wt{H}^t(z) |
&\les \int_0^{ \max( \frac{1}{|z|} , 1) } \frac{1}{\jb{r}^{1-2\eps}} dr
+ \int_{ \max( \frac{1}{|z|} , 1) }^\infty  \frac{1}{r^{\frac32-2\eps} |z|^{\frac12}} dr  \\
& \les 1 + |z|^{-2\eps}.
\end{split}
\label{H6}
\end{align}

\noi
Therefore, from \eqref{H5}, \eqref{H4}, and  \eqref{H6}, 
we obtain 
\[
| \jb{\nb_x}^{\eps} \jb{\nb_y}^{\eps} H(x, y; t) |
\les 1 + |x-y|^{-2\eps},
\]

\noi
uniformly in $t \in \R_+$.
Then, by integrating in $x, y \in B_R$, we obtain
 \eqref{H2},
 provided that $q < \frac 1\eps$.
\end{proof}

We are now ready to present the proof of Proposition \ref{PROP:need23}.

\begin{proof}[Proof of Proposition \ref{PROP:need23}]
(i)
We only show that 
\begin{align}
\E\big[ \|\Xi_1 (u_1, \xi)\|_{\Xi_0, R}\big] < \infty, 
\label{nd1}
\end{align}

\noi
where $\|\Xi_1 (u_1, \xi)\|_{\Xi_0, R}$ is as in Definition \ref{DEF:X}\,(ii).
A similar argument yields finiteness of 
the expectations 
of the constants  $K_2$ and $K_3$ in \eqref{TA2} and \eqref{TA3}, respectively.

Let $j = 1, \dots, k$.
Given a function $f$ on $B_R\subset \R^2$, 
let $\ff$ an extension of $f$ onto $\R^2$.
Given a test function
 $\varphi\in C^\infty_c(\R^2)$  
 supported on a ball $B$ of radius 1 with $B \subset B_{2R}$, 
 it follows from Lemma~\ref{LEM:W1} that 
\begin{align*}
&\E\Big[\big| \jb{  \ff :\! \Phi_1^j(t)\!:\, , \varphi   }_{L^2(\R^2)} \big|^2\Big] \\
&=\int_{(B_{2R})^2} 
\ff(x) \ff(y) \varphi(x) \varphi(y)  \E\big[ : \!\Phi_1^j(x, t))\!: \, : \!\Phi_1^j(y, t)\!:  \big] dxdy \\
&= j! \int_{\R^4} \ff(x) \ff(y) \varphi(x) \varphi(y)  \big(H(x, y;t )\big)^j dxdy, 
\end{align*}

\noi
where $H(x, y;t )$ is as in   \eqref{H1}.
Then, given small $\eps > 0$, 
it follows from  the duality, Lemma~\ref{LEM:bilin1}\,(ii), 
 the bi-parameter fractional Leibniz rule (Lemma \ref{LEM:bilin0}), 
 and  Lemma \ref{LEM:H}
that there exist finite $p = p(\eps) \gg1 $ and $q = q(\eps) \gg1 $
(with $j q < \frac 1\eps \le \frac p2$) such that 
\begin{align*}
\E& \Big[\big| \jb{ \ff  : \!\Phi_1^j(t)\!:, \varphi   }_{L^2(\R^2)} \big|^2 \Big]\\
&\les \| \ff \varphi\|_{W^{-\eps, q'}(B_{2R})}^2
\big \| \jb{\nb_x}^\eps\jb{\nb_y}^\eps  \big( H(x, y; t)\big)^j \big\|_{L^q_{x, y}(B_{2R} \times B_{2R})}\\
&\les \| \varphi\|_{W^{\eps, q'}(B_{2R})}^2 \|\ff\|_{W^{-\eps, p}(B_{2R}  )}^2
 \| \jb{\nb_x}^\eps\jb{\nb_y}^\eps H(x, y; t) \|^j_{L^{jq}_{x, y}(B_{2R} \times B_{2R})}\\
&\les \| \varphi\|_{W^{\eps, q'}(B_{2R})}^2 \|\ff\|_{W^{-\eps, p}(\R^2  )}^2.
\end{align*}

\noi
Hence, by taking an infimum over the extension $\ff$
and applying  Lemma \ref{LEM:Kol} (with $j \le k$),  
we obtain
\begin{align*}
\E\Big[  \| \, f    : \!\Phi^j_1(t) \!:\|_{W^{-2(k+1)\eps, p}(B_R)}^p \Big]
\les  \| f \|_{W^{-\eps, p}(B_R)}^p, 
\end{align*}

\noi
uniformly in $t \in \R_+$.
Then, by first conditioning on $\Xi_0$, 
we obtain
\begin{align*}
\E  \Bigg[\frac{ \int_0^T \|    :\! \Phi^j_1(t)\!: \Xi_{0\l}(t) 
\|_{W^{-2(k+1)\eps, p}(B_R)}^p  dt}
{   \int_0^T \| \Xi_{0\l}(t)  
   \|_{W^{-\eps, p}(B_R)}^p  dt }\Bigg]
\le A_{j, \l} < \infty.
\end{align*}

\noi
Hence, we conclude that 
\begin{align*}
\E \big[ \|\Xi_1 (u_1, \xi)\|_{\Xi_0, R}\big] 
&=\E  \Bigg[ \max_{0\le j, \ell \le k} 
\frac{ \int_0^T \|    :\! \Phi^j_1(t)\!: \Xi_{0\l}(t)\|_{W^{-2(k+1)\eps, p}(B_R)}^p  dt}
{   \int_0^T \| \Xi_{0\l}(t)    \|_{W^{-\eps, p}(B_R)}^p  dt }\Bigg]\\
&\leq \sum_{j,\l = 0}^k 
A_{j, \l} < \infty.
\end{align*}

\noi
This proves \eqref{nd1}.

\smallskip

\noi
(ii)  
The bound 
\eqref{need3}
follows from the proof of Proposition \ref{PROP:tight}.  
Namely, in view of 
  \eqref{view0}, \eqref{view2}, 
  and  \eqref{view6}
with \eqref{view3}, 
\eqref{ST3}, \eqref{ST4}, 
\eqref{view4}, and \eqref{view5}, 
it follows from 
  Kolmogorov's continuity criterion
  (\cite[Theorem 8.2]{Bass}) 
 together with the exponential decay in time (see
Remark \ref{REM:bound})
  that 
\begin{align*}
\rho_L\Big(
\| \Xi_0( \phi) \|_{(L^\infty_t W^{-\eps, p}_\mu)^{\otimes k} }> K \Big)
\les \frac 1{K^q}  
\end{align*}

\noi
for any finite $q \ge 1$, 
from which the desired bound \eqref{need3} follows.
\end{proof}

 \subsection{Global well-posedness 
on the plane}
\label{SUBSEC:4.4}

We  are now ready to prove global well-posedness
of the hyperbolic $\Phi^{k+1}_2$-model \eqref{NW1} 
on the plane.
We do this by constructing a solution 
on each cone $\bC_R$, $R >0$.

\begin{proposition}\label{PROP:GWP}
Let $\rhoo_\infty$ be as in Theorem \ref{THM:1}\,(i)
and 
$\Law (u_0, u_1) = \rhoo_\infty$.
Then, given $R>0$, 
the hyperbolic $\Phi^{k+1}_2$-model \eqref{NW1} 
is well-posed on the cone $\bC_R$.
More precisely, 
there exists 
a function  $u$ of the form \eqref{decomp2}
such that 
the following statements holds true 
$\rhoo_\infty$-almost surely\textup{:}

\begin{itemize}
\item $(u, \dt u) \in L^\infty([0, R];  \vec H^{-\eps_k}(B_{R-t}))$,

\smallskip

\item the remainder term
$v = u -   \S(t)u_0 - \D(t) u_1 - \Psi$
satisfies the equation \eqref{NW5b} on $\bC_R$,
and
$(v, \dt v) \in L^\infty([0, R];  \vec H^{1-\eps_k}(B_{R-t}))$, 

\end{itemize}

\noi
where  $\eps_k = 2(k+1) \eps$ is as in \eqref{eps1}.

\end{proposition}

For the convergence of the $L$-periodic solution $u_L$, $L \in \A$, 
to the solution $u$ constructed in Proposition \ref{PROP:GWP},
see Remark \ref{REM:conv} below.

\begin{proof}

Let  $\A = \{ L_j : j \in \N\} \subset \N$ 
be the index set such that 
 Theorem \ref{THM:1}\,(i)
and
Proposition~\ref{PROP:uniq} hold.
Given a spatial white noise $\z$ on $\R^2$
and a space-time white noise $\xi$ on $\R^2\times \R_+$, 
let 
$\z_L$ be the $L$-periodic spatial white noise defined in \eqref{wh1} 
and \eqref{wh2}, 
and 
$\xi_L$ be the  $L$-periodic space-time white noise defined in \eqref{wh4} 
and \eqref{wh5}, $L \ge 1$, respectively.

In Proposition \ref{PROP:uniq}, 
we proved
 that the $L_j$-periodic enhanced Gibbs measure 
$\nu_{L_j} = (\Xi_0)_\# \rho_{L_j}$
on $\WW = \big(C(\R_+;  W^{-\eps, p}_\mu(\R^2))\big)^{\otimes k}$
converges weakly to 
$\nu_\infty = (\Xi_0)_\# \rho_\infty$
as $j \to \infty$.
Then, by 
the Skorokhod representation theorem
(Lemma \ref{LEM:Sk}), 
there exist
 a probability space $(\wt \O, \wt \F, \wt\PP)$,
and random variables $\Xi_0^j, \Xi_0:\wt \O \to \WW$ 
such that 
\begin{align}
\Law( \Xi_0^j) = \nu_{L_j}
\qquad \text{and}\qquad
\Law(\Xi_0) = \nu_\infty ,
\label{opt0}
\end{align}

\noi
and $\Xi_0^j$ converges $\wt\PP$-almost surely to $\Xi_0$ in $\WW$ as $j \to\infty$.
Given $j \in \N$, 
we define a transport plan $\plan_{L_j} \in \Pi(\nu_\infty, \nu_{L_j})$  by 
\begin{align}
\plan_{L_j} = (\Xi_0, \Xi_0^j)_\#\wt \PP.
\label{opt0a}
\end{align}

\noi
Then, by the bounded convergence theorem, we have 
\begin{align}
\begin{split}
\int_{\WW\times \WW} 
   d_\WW( \Ta, \Ta' ) d\plan_{L_j}(\Ta, \Ta')
& = 
\int_{\wt \O}
   d_\WW(\Xi_0(\wt \o), \Xi_0^j(\wt \o)) d\wt \PP(\wt \o)\\
& \too 0, 
\end{split}
\label{opt1}
\end{align}

\noi
as $j \to \infty$, 
where
the metric $d_\WW$ is   as in  \eqref{norm1}.
See Remark 6.11 in \cite{OOT2}.


%
%
%

From \eqref{nuL} and  \eqref{need1} with \eqref{xi1}, 
we have 
\begin{align}
\rho_{L_j} = (\eval \circ \proj_1 \big)_\# \nu_{L_j} 
\qquad \text{and}
\qquad 
\rho_{\infty} = (\eval \circ \proj_1 \big)_\# \nu_{\infty}   , 
\label{opt3}
\end{align}

\noi
where $\eval$ denotes the evaluation map at time $t = 0$
and $\proj_1$ is the projection onto the first component.
We now define the first components of the initial data
by 
\begin{align}
u_{0, L_j} = \eval \circ \proj_1(\Xi_0^j)
\qquad \text{and}
\qquad 
u_{0} = \eval \circ \proj_1(\Xi_0).
\label{opt3a}
\end{align}

\noi
Then, from \eqref{opt0}, \eqref{opt3}, and \eqref{opt3a}, 
we have 
\begin{align*}
\Law( u_{0, L_j}) = \rho_{L_j}
\qquad \text{and}\qquad
\Law(u_{0}) = \rho_\infty .
\end{align*}

\noi
Finally, we let

\smallskip

\begin{itemize}
\item
 $\vec u = (u, \dt u) = \vec u(u_0, u_1, \xi)$ 
 denote a solution to 
 the hyperbolic $\Phi^{k+1}_2$-model 
 \eqref{NW1} on~$\R^2$
with $u_0$ as in \eqref{opt3a}
and $u_1 = \z$, 
and

\smallskip

\item 
 $\vec u_{L_j} = (u_{L_j}, \dt u_{L_j}) = \vec u_{L_j}(u_{0, L_j}, u_{1, L_j}, \xi_{L_j})$ 
 denote a solution to 
 the $L_j$-periodic hyperbolic $\Phi^{k+1}_2$-model~\eqref{NW2} on $\R^2$
with $u_{0, L_j}$ as in \eqref{opt3a} and 
$u_{1, L_j} = \z_{L_j}$.

\end{itemize}

\smallskip

%

In the following, we fix 
 $R  \in \N$
 and work on the cone $\bC_R$.
When $L_j \gg R$,
it follows from the finite speed of propagation and \eqref{wh7}
with the notation above
that 
\begin{align}
\vec u_{L_j}(u_{0, {L_j}}, u_{1, {L_j}}, \xi_{L_j})  = \vec u_{L_j}(u_{0,{L_j}}, u_{1}, \xi)
\quad \text{on the cone $\bC_R$}.
\label{gwp1}
\end{align}

\noi
In particular, on $\bC_R$, 
$\vec u_{L_j}(u_{0, {L_j}}, u_{1}, \xi)$
satisfies  \eqref{NW1}
with the initial data $(u_{0, {L_j}}, u_1)$.
We denote by $\PP_{u_1, \xi} = \PP_{u_1} \otimes \PP_\xi$
the probability 
distribution of the spatial white noise $u_1$ and the space-time white noise $\xi$
which are independent from each other
and also from the transport plan $\plan_{L_j}$, $j \in \N$, 
defined in \eqref{opt0a}.
For simplicity of notation, we set $\X = \D'(\R^2) \times \D'(\R^2\times \R_+)$, 
where $(u_1, \xi)$ lives.

Given $\Ta = (\Ta_{1}, \cdots, \Ta_{k}) \in \WW$, 
let $v = v(\Ta, u_1, \xi)$ be 
a solution to 
\eqref{NW5b}, 
where
we replace $:\! \Phi_0^m(u_0) \!:$ by $\Ta_{m}$.
Namely,  $v = v(\Ta, u_1, \xi)$ satisfies
\begin{align}
\begin{split}
v(t) 
& = -  \sum_{\l = 0}^k \sum_{m = 0}^\l
{k\choose \ell}
{\l \choose m}\\
& \hphantom{XXX}
\times \int_0^t \D(t - t')\big( 
\Ta_{m}(t')
 :\! \Phi_1^{\l-m}(u_1, \xi) (t')\!:
v^{k- \l}(t')\big) dt'.
\end{split}
\label{NWx}
\end{align}

\noi
Define $E_R \subset \WW \times \X$ by 
\begin{align*}
E_R =  \big\{&
(\Ta, u_1, \xi) \in \WW \times \X : 
\text{there exists a solution 
 $v = v(\Ta, u_1, \xi)$
to \eqref{NWx}} 
\\
&
\text{on the cone $\bC_R$ such that 
$\vec v = (v, \dt v) \in L^\infty([0, R];  \vec H^{1-\eps_k}(B_{R-t}))$}
\big\}.
\end{align*}

\noi
Our goal is to show  
\begin{align}
\rho_\infty \otimes 
\PP_{u_1, \xi}
\bigg(\bigcap_{R \in \N} \big\{(u_0, u_1, \xi):  \big(\Xi_0(u_0) , u_1, \xi\big) \in  E_R \big\}\bigg) = 1, 
\label{gwp0}
\end{align}

\noi
where $\Xi_0(\cdot)$ is the map defined in \eqref{xi1}.


Fix $R \in \N$. Given    $M_0, M_1,  M_2 \ge 1$ and $\dl > 0$, 
define
\begin{align}
\begin{split}
  A_{R, M_0, M_2}
&  = \Big\{(\Ta', u_1, \xi) \in \WW \times \X: 
\| \Ta' \|_{\WW(R)} 
 \le M_0, \\
& \hphantom{XXl}
 \| \vec v(\Ta', u_1, \xi)\|_{L^\infty([0, R]; \vec H^{1-\eps_k}(B_{R - t}))}
\leq  M_2
 \Big\}, \\ 
  B_{R, M_1}
& =  \Big\{
(\Ta, \Ta', u_1, \xi) \in \WW \times \WW \times \X: 
\| \Xi_1(u_1, \xi) \|_{\Ta, \Ta', R} \leq M_1\Big\}, \\
C_{R,  \dl }
& = \Big\{ 
(\Ta, \Ta') \in \WW \times \WW: 
d_\WW( \Ta, \Ta' ) \leq \dl  \Big\}, 
\end{split}
\label{gwp0a}
\end{align}

\noi
where 
$\Ta' = (\Ta_{1}', \cdots, \Ta_{k}')$, 
the $\WW(R)$-norm is   as in \eqref{norm2}, 
and $\| \cdot  \| _{\Ta, \Ta', R}$ is as in 
Definition~\ref{DEF:X}\,(iii).
Furthermore, we define $\wt E_R$, 
$\wt   A_{R, M_0, M_2}$, 
and $\wt C_{R,  \dl }$
by setting 
\begin{align}
\begin{split}
\wt E_R & = \big\{(\Ta, \Ta', u_1, \xi) \in \WW \times \WW \times \X: 
(\Ta, u_1, \xi) \in E_R
\big\}, \\
\wt   A_{R, M_0, M_2}
& = \big\{(\Ta, \Ta', u_1, \xi) \in \WW \times \WW \times \X: 
(\Ta', u_1, \xi) \in   A_{R, M_0, M_2}\big\}, \\
\wt C_{R,  \dl }
& = \big\{(\Ta, \Ta', u_1, \xi) \in \WW \times \WW \times \X: 
(\Ta, \Ta') \in C_{R,  \dl }\big\}
\end{split}
\label{gwp0b}
\end{align}

\noi
such that all the sets live in a common space 
$\WW \times \WW \times \X$.
Then, it follows from 
Propositions~\ref{PROP:LWP}
and 
 \ref{PROP:stab} that 
\begin{align}
\begin{split}
\wt E_R \supset
\wt  A_{R, M_0,M_2}
 \cap B_{R,  M_1} \cap \wt C_{R,  \dl },
\end{split}
\label{gwp2}
\end{align}

\noi
for  any $M_0, M_1, M_2 \ge 1$ and 
$0 < \dl \le \dl^*
= \dl^* (\mu, R, M_0, M_1, M_2)$.
Here, 
$\dl^* = \dl^* (\mu, R, M_0, M_1, M_2)$
is chosen such that, in view of \eqref{norm1} and \eqref{norm2}, 
\begin{align}
 d_\WW(f, g) \le \dl^* \quad \text{implies}\quad \| f - g\|_{\WW(R)} \le \dl_*, 
\label{dl1}
 \end{align}

\noi
where 
$\dl_* = \dl_* (R, M_0, M_1, M_2)> 0$ is as in Proposition~\ref{PROP:stab}.\footnote{It suffices
to take $\dl^* \le C(\mu, R) \dl_*$
for some  $C(\mu, R) > 0$.}

Recalling  \eqref{Wass2} that 
\begin{align}
 \int_{\Ta' \in \WW}  d\plan_{L_j}(\Ta, \Ta')
  = d\nu_\infty(\Ta) \quad \text{and}
\quad 
 \int_{\Ta \in \WW}  
 d\plan_{L_j}(\Ta, \Ta') = d\nu_{L_j}(\Ta'), 
\label{opt2}
\end{align}

\noi
 it follows from 
\eqref{opt2}, 
\eqref{gwp0b}, \eqref{gwp2}, 
and \eqref{gwp0a} that
\begin{align}
 \rho_\infty & \otimes 
 \PP_{u_1, \xi}
\Big(\big\{(u_0, u_1, \xi):  \big(\Xi_0(u_0) , u_1, \xi\big) \in  E_R \big\}\Big)\notag\\
& = \int_\X \int_{\WW\times \WW} \ind_{(\Ta, u_1, \xi) \in E_R}(\Ta, \Ta', u_1, \xi)
d \plan_{L_j} (\Ta, \Ta') d\PP_{u_1, \xi}(u_1, \xi)\notag\\
& = \int_\X \int_{\WW\times \WW} 
\ind_{\wt  E_R}(\Ta, \Ta', u_1, \xi)
d \plan_{L_j} (\Ta, \Ta') d\PP_{u_1, \xi}(u_1, \xi)\notag\\
& \ge
\int_\X \int_{\WW\times \WW} 
\ind_{\wt  A_{R, M_0,M_2}
 \cap B_{R,  M_1} \cap \wt C_{R,  \dl }}
(\Ta, \Ta', u_1, \xi)
d \plan_{L_j} (\Ta, \Ta') d\PP_{u_1, \xi}(u_1, \xi)\notag\\
& \ge 1 -
 \int_{\WW\times \WW} 
\ind_{\{\| \Ta' \|_{\WW(R)} >  M_0\}}
(\Ta')
d \plan_{L_j} (\Ta, \Ta') 
\notag \\
& \quad 
- \int_\X \int_{\WW\times \WW} 
\ind_{
\{  \| \vec v(\Ta', u_1, \xi)\|_{L^\infty([0, R]; \vec H^{1-\eps_k}(B_{R - t}))}
>  M_2\}}
(\Ta', u_1, \xi)
\notag \\
& \hphantom{XXXXXXX}
d \plan_{L_j} (\Ta, \Ta') d\PP_{u_1, \xi}(u_1, \xi)
\notag \\
& \quad
 - \int_\X \int_{\WW\times \WW} 
\ind_{B_{R,  M_1}^c}
(\Ta, \Ta', u_1, \xi)
d \plan_{L_j} (\Ta, \Ta') d\PP_{u_1, \xi}(u_1, \xi)
\notag \\
& \quad
-  \int_{\WW\times \WW} 
\ind_{C_{R,  \dl }^c}
(\Ta, \Ta')
d \plan_{L_j} (\Ta, \Ta') 
\notag \\
& =: 1 - \1 - \II - \III - \IV.
\label{PXX1}
\end{align}

From \eqref{opt2}, \eqref{nuL}, 
\eqref{opt3}, 
 Markov's inequality, 
and Proposition~\ref{PROP:need23}\,(ii), we have
\begin{align}
\begin{aligned}
\1 
& \les_{\mu, R} \frac{1}{M_0}  
\E_{\rho_{L_j}}\Big[  \| \Xi_0(u_{0, L_j}) \|_{(L^\infty_t W^{-\eps, p}_\mu)^{\otimes k} } \Big] 
\les \frac{1}{M_0}, 
\end{aligned}
\label{PX0}
\end{align}

\noi
uniformly in $j \in \N$.
Similarly,  from 
 \eqref{opt2}, \eqref{nuL}, 
\eqref{opt3}, and 
 Markov's inequality, 
 we have 
\begin{align}
\II < \frac 1{M_2}
\E_{\rho_{L_j}\otimes \PP_{u_1, \xi}}
\Big[\| \vec v(\Xi_0(u_{0, L_j}), u_1, \xi)\|_{L^\infty([0, R]; \vec H^{1-\eps_k}(B_{R - t}))}\Big].
\label{PX1}
\end{align}

\noi
In view of  \eqref{gwp1}, 
we have 
\[ \vec v(\Xi_0(u_{0, L_j}), u_1, \xi) = 
\vec u_{L_j} (u_{0, L_j}, u_1, \xi)
- 
 \S(t)u_{0, L_j} + \D(t) u_1 + \Psi(t) 
\]

\noi
on $\bC_R$, 
where $u_{L_j}$
is the solution to the $L_j$-periodic hyperbolic $\Phi^{k+1}$-model
\eqref{NW2}.
Thus,  we see that 
$ v(\Xi_0(u_{0, L_j}), u_1, \xi)$ satisfies
\begin{align}
v(\Xi_0(u_{0, L_j}), u_1, \xi)(t) & = - \int_0^t \D(t-t') (\ind_{\bC_R}: \!u_{L_j}^k(t')\! : ) dt'
\label{PX2}
\end{align}

\noi
on $\bC_R$.
Compare this with 
\eqref{NW3}, 
\eqref{decomp2}, 
and \eqref{NW4}
 (with an additional cutoff function $\ind_{\bC_R}$ as in \eqref{NW5c}).
Then, from \eqref{PX2}
 Minkowski's integral  inequality,
 and H\"older's inequality, 
we have
\begin{align}
\begin{split}
& \| v(\Xi_0(u_{0, L_j}), u_1, \xi)\|_{L^\infty([0, R];\vec H^{1-\eps_k}(B_{R - t}))}\\
&\quad \le \bigg\|\int_0^{t} \big\|  \D(t-t')(\ind_{\bC_R}  : \!u_{L_j}^k(t')\!:) \|_{H^{1-\eps_k}(B_{R - t})} dt' 
\bigg\|_{L^\infty([0, R])} \\
&\quad \les \int_0^R \|   :\!u_{L_j}^k(t')\!: \|_{H^{-\eps_k}(B_{R - t'})} dt' .
\end{split}
\label{PX3}
\end{align}

\noi
Hence, from 
\eqref{PX1}, \eqref{PX3}, 
 the  invariance $\rhoo_{L_j}$ under
\eqref{NW2}, 
and Proposition~\ref{PROP:need23}\,(ii), we obtain
\begin{align}
\begin{aligned}
\II 
& \les_{\mu, R} \frac{1}{M_2}  
\E_{\rho_{L_j}}\Big[  \| \Xi_0(u_{0, L_j}) \|_{L^\infty_t (W^{-\eps, p}_\mu)^{\otimes k} } \Big] 
\les \frac{1}{M_2},
\end{aligned}
\label{PX4}
\end{align}

\noi
uniformly in $j \in \N$.

From  \eqref{gwp0a}  and Proposition~\ref{PROP:need23}\,(i), we have
\begin{align}
\begin{split}
\III &
<  \frac{1}{M_1}   
\E_{ \plan_{L_j} \otimes \PP_{u_1, \xi}}
\Big[  \|\Xi_1(u_1,  \xi) \| _{\Ta, \Ta' , R} \Big] 
\les \frac 1{M_1},
\end{split}
\label{PX5}
\end{align}

\noi
uniformly in $j \in \N$.
Finally, from \eqref{gwp0a}
and 
\eqref{opt1}, 
we have, for each fixed $\dl > 0$,  
\begin{align}
\IV 
& <  \frac 1\dl  \int_{\WW\times \WW} 
 d_\WW( \Ta, \Ta' )
d \plan_{L_j} (\Ta, \Ta') 
 \too 0,  
\label{PX6}
\end{align}

\noi
as $j \to \infty$.

Fix small $\kk > 0$.
We first choose $M_0, M_1, M_2 \gg1$ such that \eqref{PX0}, \eqref{PX4},  and \eqref{PX5} imply
\begin{align}
\1 + \II + \III< \frac \kk 2.
\label{gwp4}
\end{align}

\noi
Then, we choose 
sufficiently small $\dl = \dl (R, M_0, M_1, M_2)> 0$
such that $0 < \dl \le \dl^*$ (and thus~\eqref{PXX1} holds).
Finally, by  taking sufficiently large $j \gg 1$, 
we obtain from \eqref{PX6} that 
\begin{align}
\IV < \frac \kk 2.
\label{gwp5}
\end{align}

\noi
Hence, 
we obtain $\PP(E_R^c) < \kk$.
Since the choice of $\kk > 0$ was arbitrary, 
we then conclude 
from \eqref{PXX1}, \eqref{gwp4}, and \eqref{gwp5}
that 
\begin{align*}
\rho_\infty \otimes 
\PP_{u_1, \xi}
\Big(\big\{(u_0, u_1, \xi):  \big(\Xi_0(u_0) , u_1, \xi\big) \in  E_R \big\}\Big) = 1
\end{align*}

\noi
for any $R \in \N$, 
which in turn implies 
 \eqref{gwp0}.
This concludes the proof of 
Proposition~\ref{PROP:GWP}.
\end{proof}

\begin{remark}\label{REM:conv}\rm

A slight modification of the proof of Proposition \ref{PROP:GWP}
shows that,
on each cone $\bC_R$, $R> 0$, 
the $L_j$-periodic solution $\vec u_{L_j} = (u_{L_j}, \dt u_{L_j})$ to \eqref{NW2}
converges in probability 
to
the solution~$\vec u = (u, \dt u)$ to \eqref{NW1} constructed in Proposition \ref{PROP:GWP}
in the $L^\infty([0, R]; \vec H^{-\eps_k}(B_{R-t}))$-topology.
In particular, 
there exists a subsequence $\{u_{L_{j_\l}}\}_{\l \in \N}$
such that $\vec u_{L_{j_\l}}$
converges almost surely to
$\vec u$ in 
$L^\infty([0, R]; \vec H^{-\eps_k}(B_{R-t}))$
as $\l \to \infty$.
Since 
$\vec u_{L_{j_\l}} \in 
C(\R_+; \vec H^{-\eps_k}_\text{loc}(\R^2))$, 
we then deduce that 
$u \in C(\big[0, \tfrac R2\big]; \vec H^{-\eps_k}(B_{ R/2}))$
almost surely.
Since the choice of $R> 0$ was arbitrary, 
we conclude that 
$(u, \dt u) \in C(\R_+;  \vec H^{-\eps_k}_\textup{loc}(\R^2))$
almost surely.


Let $u_0$ and $u_{0, L_j}$ be as in \eqref{opt3a}.
Then, 
define a set $F_R \subset \wt \O  \times \X$ by
\begin{align}
\begin{split}
& F_R = 
\bigcup_{M_0, M_1, M_2 = 1}^\infty\, 
\bigcap_{\substack{m \in \N\\m \ge   \dl^*(\mu, R, M_0, M_1, M_2)^{-1}}}
\bigg(\bigcap_{\l = 1}^\infty
\bigcup_{j = \l}^\infty
\big(D_{R, M_0, M_1, M_2}^{L_j}
\cap  C_{R,  \frac 1m }^{L_j}\big)
\Big\}\bigg), 
\end{split}
\label{gwp6}
\end{align}

\noi
where
 $\dl^*(\mu, R, M_0, M_1, M_2)$
is as in the proof of Proposition \ref{PROP:GWP} (see also \eqref{dl1}).
Here, 
$D_{R, M_0, M_1, M_2}^{L_j}$ and $C_{R, \dl}^{L_j}$
are given by 
\begin{align}
\begin{split}
D_{R, M_0, M_1, M_2}^{L_j}
& = 
\Big\{ 
\big( 
\Xi_0(u_0(\wt \o)), 
\Xi_0(u_{0, L_j}(\wt \o)), u_1, \xi\big)
\in 
\wt  A_{R, M_0,M_2}
 \cap 
 B_{R,  M_1}
\big)\Big\}, \\
C_{R, \dl}^{L_j}
& = 
\Big\{ 
\big( 
\Xi_0(u_0(\wt \o)), 
\Xi_0(u_{0, L_j}(\wt \o))\big)
\in 
 \wt C_{R,  \dl }
\Big\}, 
\end{split}
\label{gwp6a}
\end{align}

\noi
where 
$\wt  A_{R, M_0,M_2}^{L_j}$, 
$B_{R,  M_1}^{L_j}$, and $\wt C_{R,  \dl }^{L_j}$
are 
as in \eqref{gwp0a} and \eqref{gwp0b}.
As in \eqref{gwp2}, 
we then have 
 \[ F_R \subset 
\Big\{ ( \wt \o, u_1, \xi) \in \wt \O \times \X: 
\big(\Xi_0(u_0(\wt \o)), 
 u_1, \xi\big)\in E_R \Big\} .\]

Given small $\kk > 0$,  by proceeding as in \eqref{PXX1}
with the uniform (in $j$) bounds \eqref{PX0}, \eqref{PX4}, and~\eqref{PX5}, 
choosing sufficiently large $M_0, M_1, M_2 \gg1$, 
and then applying Markov's inequality
and \eqref{opt1}, 
we obtain
\begin{align*}
\wt \PP\otimes \PP_{u_1, \xi} (F_R^c) 
& \le  \kk  + \sup_{\substack{m \in \N\\m \ge   \dl^*(\mu, R, M_0, M_1, M_2)^{-1}}}
\liminf_{j \to \infty}  \plan_{L_j} \Big(
d_\WW\big( \Ta, \Ta' ) \big) >  \tfrac 1m  \Big)\\
& \le  \kk  + \sup_{\substack{m \in \N\\m \ge   \dl^*(\mu, R, M_0, M_1, M_2)^{-1}}}
m \cdot \liminf_{j \to \infty}
\int_{\WW\times \WW} 
 d_\WW( \Ta, \Ta' )
d \plan_{L_j} (\Ta, \Ta') \\
& = \kk.
\end{align*}

\noi
Since the choice of $\kk > 0$ was arbitrary, 
we conclude that 
\begin{align*}
\wt \PP\otimes \PP_{u_1, \xi}(F_R) = 1.
\end{align*}

With a slight abuse of notation, let us denote by $\o'$
an element $(\wt \o, u_1, \xi) \in \wt \O \times \X$.
Given $\o' \in F_R$, 
it follows from 
the definition \eqref{gwp6} of $F_R$
with \eqref{gwp2}
and Propositions~\ref{PROP:LWP} and \ref{PROP:stab}
that 
there exists an $\o'$-dependent subsequence $L_{j_\l} \subset \A$ 
such that $\vec u_{L_{j_\l}}(\o')$ converges 
almost surely 
to $\vec u(\o')$ in $L^\infty([0, R]; \vec H^{-\eps_k}(B_{R-t}))$  as $\l \to \infty$.
On the other hand, 
we claim that $\{u_{L_j}\}_{j \in \N}$
is Cauchy in probability 
in $L^\infty([0, R]; \vec H^{-\eps_k}(B_{R-t}))$.
Then, from the uniqueness of a limit, 
we conclude that 
$\vec u_{L_j}$ converges
 in probability to $\vec u$
in $L^\infty([0, R]; \vec H^{-\eps_k}(B_{R-t}))$ as $j \to \infty$.

It remains to show that 
 $\{\vec u_{L_j}\}_{j \in \N}$
is Cauchy in probability 
(with respect to $\wt \PP \otimes \PP_{u_1, \xi}$)
in $L^\infty([0, R]; \vec H^{-\eps_k}(B_{R-t}))$.
From 
\eqref{opt1} with \eqref{opt3a}, 
we have, for each fixed $\dl_0 > 0$, 
\begin{align}
\begin{aligned}
\wt \PP&  \Big(
d_\WW\big( \Xi_0(u_{0, L_{j_1}}), \Xi_0(u_{0, L_{j_2}} ) \big) >  \dl_0  \Big)\\
& \le \wt \PP  \Big(
d_\WW\big( \Xi_0(u_{0, L_{j_1}}), \Xi_0(u_{0} ) \big) >  \tfrac {\dl_0} 2 \Big)
+ \wt \PP  \Big(
d_\WW\big( \Xi_0(u_{0}), \Xi_0(u_{0, L_{j_2}} ) \big) >  \tfrac {\dl_0} 2 \Big)\\
& < \frac 2{\dl_0}
\int_{\wt \O}
   d_\WW(\Xi_0^{j_1}(\wt \o), \Xi_0(\wt \o)) d\wt \PP(\wt \o)
+   \frac 2{\dl_0}
\int_{\wt \O}
   d_\WW(\Xi_0(\wt \o), \Xi_0^{j_2}(\wt \o)) d\wt \PP(\wt \o)\\
& \too 0, 
\end{aligned}
\label{gwp8}
\end{align}

\noi
as $j_1, j_2 \to \infty$.
Given $\kk > 0$, 
from \eqref{PX0}, \eqref{PX4},  and \eqref{PX5}, 
there exist $M_0, M_1, M_2 \gg 1$ such that 
\begin{align}
\wt \PP
 \otimes \PP_{u_1, \xi}\Big( 
(D_{R, M_0, M_1, M_2}^{L_{j_1}})^c
\cup(D_{R, M_0, M_1, M_2}^{L_{j_2}})^c
\Big)
< \frac \kk 2, 
\label{gwp9}
\end{align}

\noi
uniformly in $j \in \N$, 
where
$D_{R, M_0, M_1, M_2}^{L_j}$  is as in \eqref{gwp6a}.
Then, 
it follows from 
Proposition~\ref{PROP:stab}
with \eqref{gwp8} and \eqref{gwp9}
(see also~\eqref{wh7})
that 
\begin{align*}
& \wt \PP
 \otimes \PP_{u_1, \xi}  \Big( \big\|  \vec u_{L_{j_1}}(u_{0, L_{j_1}}(\wt \o), u_1, \xi) 
- \vec u_{L_{j_2}} (u_{0, L_{j_2}}(\wt \o), u_1, \xi) \big\|_{L^\infty([0, R]; \vec H^{-\eps_k}(B_{R-t}))}
> \dl \Big) \\
& \quad \le \wt \PP  \Big(
d_\WW\big( \Xi_0(u_{0, L_{j_1}}), \Xi_0(u_{0, L_{j_2}} ) \big) >   \dl_0 \Big)\\
& \quad \quad +
\wt \PP\otimes \PP_{u_1, \xi}
\Big( 
(D_{R, M_0, M_1, M_2}^{L_{j_1}})^c
\cup(D_{R, M_0, M_1, M_2}^{L_{j_2}})^c
\Big)\\
& \quad < \kk
\end{align*}

\noi
for any $j_2 \ge j_1 \gg 1$.
Here, 
$\dl_0 = \dl_0(\dl) > 0$ is a small number 
such that 
$ d_\WW(f, g) \le \dl_0 $
implies $\| f - g\|_{\WW(R)} \le \frac 12 C_0^{-1}\dl$, 
where 
$C_0 = C_0(R, M_0, M_1, M_2)> 0 $
are as in Proposition~\ref{PROP:stab}.
This shows that 
 $\{\vec u_{L_j}\}_{j \in \N}$
is Cauchy in probability 
in $L^\infty([0, R]; \vec H^{-\eps_k}(B_{R-t}))$.

\end{remark}

\subsection{Invariance of the Gibbs measure $\rhoo_\infty$}
\label{SUBSEC:4.5}

We conclude this section by briefly discussing invariance of the Gibbs measure $\rhoo_\infty$
under the dynamics of the hyperbolic $\Phi^{k+1}_2$-model~\eqref{NW1} on the plane.
We only consider the values of $L$ belonging to $\A = \{ L_j: j \in \N\} \subset \N$, 
along which we proved Theorem \ref{THM:1}\,(i).

Let $u$ be the solution to 
the  hyperbolic $\Phi^{k+1}_2$-model  \eqref{NW1}
with 
$\Law (u_{0}, u_{1}) = \rhoo_\infty$.
Our goal is to show
\begin{align}
\Law(\vec u(t)) = \rhoo_\infty
\label{inv1}
\end{align}

\noi
for each $t \in \R_+$.
Fix $t > 0$.
Given $j \in \N$, 
let $u_{L_j}$ be the solution to 
the $L_j$-periodic hyperbolic $\Phi^{k+1}_2$-model~\eqref{NW2}
with 
$\Law (u_{0, L_j}, u_{1, L_j}) = \rhoo_{L_j}$
such that 
$(u_{0, L_j}, u_{1, L_j})$ converges almost surely to 
$(u_{0}, u_{1})$ 
in $ \vec H_\textup{loc}^{-\eps}(\R^2)$,  as $j \to \infty$.
By the invariance of the Gibbs measure 
$\rhoo_{L_j}$ under the $L_j$-periodic dynamics, 
we have
$\Law(\vec u_{L_j}(t)) = \rhoo_{L_j}$
for any $t \in \R_+$.
Then, it follows from the weak convergence
of $\{ \rhoo_{L_j}\}_{j \in \N}$ to $\rhoo_{\infty}$
that 
\begin{align}
\rhoo_\infty = \wlim_{j \to \infty}\Law(\vec u_{L_j}(t)), 
\label{inv2}
\end{align}

\noi
where $\wlim$ denotes 
a weak limit of probability measures.

On the other hand, 
given $R> t$, let $\vec \varphi  = (\varphi_0, \varphi_1)\in 
(\D(\R^2))^{\otimes 2} = (C_c^\infty(\R^2))^{\otimes2}$
be a pair of test functions with $\supp \varphi_\l \subset B_R$, $\l = 1, 2$.
Then, it follows from the proof of Proposition~\ref{PROP:GWP}
(see Remark \ref{REM:conv})
that the solution $\vec u_{L_j} $ to \eqref{NW2}
converges in probability to the solution $\vec u$ to~\eqref{NW1}
on  
$\bC_{R + t}$ as $j \to \infty$.
This in particular implies
$\jb{ \vec u_{L_j}(t), \vec \varphi} $
converges in probability to 
$\jb{ \vec u(t), \vec \varphi} $
as $j \to \infty$, 
where $\jb{\cdot, \cdot}$ denotes
$(\D'(\R^2))^{\otimes 2}$-$(\D(\R^2))^{\otimes 2}$
duality pairing.
In particular, 
as a $(\D'(\R^2))^{\otimes 2}$-valued random variable
$\vec u_{L_j}(t)$ converges in law to $\vec u(t)$ and, 
therefore,  together with \eqref{inv2}
and the uniqueness of a limit, 
we obtain \eqref{inv1}.
This concludes the proof of Theorem \ref{THM:1}\,(ii).

\appendix

\section{Embeddings between weighted Sobolev and Besov spaces}
\label{SEC:A}

In this appendix, we present the proof of Lemma \ref{LEM:equiv}.
Fix $ 1\le p < \infty$.
We first prove 
\begin{align}
\| f \|_{W^{s, p}_\mu} \les \|f\|_{B^{s', \mu'}_{p, 1}}.
\label{equiv1}
\end{align}

\noi
for $s < s'$ and $\mu \ge c(p) \mu' > 0$.
Let us state an auxiliary lemma.

\begin{lemma}\label{LEM:equiv2}
Let $s \in \R$, finite $p \ge 1$, and $\mu > 0$.
Then, given any $\eps > 0$, there exists $c_0 > 0$  such that 
\begin{align}
\| \phi_j \Q_k f\|_{W^{s, p}}
\les e^{c_0\frac \mu p  2^{j\dl}}
2^{(s+\eps) k} \| \Q_k f\|_{L^p_\mu}
\label{ee1}
\end{align}

\noi
for any $j , k \in \Z_{\ge 0}$, 
where $\Q_k$ is as in \eqref{LP2}.

\end{lemma}

We first prove \eqref{equiv1} by assuming Lemma \ref{LEM:equiv2}.
We present the proof of Lemma \ref{LEM:equiv2}
at the end of this section.
From 
\eqref{Lpm1} and Lemma \ref{LEM:equiv2}, we have
\begin{align*}
\| f \|_{W^{s, p}_\mu}
& \le \sum_{k = 0}^\infty \|\Q_k  f \|_{W^{s, p}_\mu}
\le \sum_{k = 0}^\infty 
\sum_{j = 0}^\infty
e^{-\frac \mu p 2^{j\dl} } \| \phi_j \Q_k  f \|_{W^{s, p}}\\
& \les \sum_{k = 0}^\infty 
\sum_{j = 0}^\infty
e^{-\frac \mu p 2^{j\dl} } 
 e^{c_0\frac{\mu'}{p} 2^{j\dl}}
2^{(s+\eps) k} 
\|  \Q_k  f \|_{L^p_{\mu'}}\\
&  \les \|f\|_{B^{s', \mu'}_{p, 1}}, 
\end{align*}

\noi
provided that $\mu >  c_0  \mu'$
and $s < s'$.
This proves \eqref{equiv1}.

\medskip

Next, we prove
\begin{align}
\|f\|_{B^{s, \mu}_{p, \infty}} \les \| f \|_{W^{s, p}_{\mu'}}
\label{equiv2}
\end{align}

\noi
for any $s \in \R$ and 
$0 < \mu' <  c_1  \mu$
for some small $c_1 > 0$.

From \eqref{phi1a} and \eqref{phi2}, we have 
\begin{align}
\begin{split}
2^{sk}\| \Q_k f\|_{L^p_\mu}
& \le \sum_{j, j'  = 0}^\infty2^{sk}\| \phi_{j'} \Q_k (\phi_j f)\|_{L^p_\mu}\\
& \le \sum_{j, j'  = 0}^\infty
e^{-c \frac \mu p 2^{j' \dl}} 
2^{sk}\| \phi_{j'} \Q_k (\phi_j f)\|_{L^p}.
\end{split}
\label{eg1}
\end{align}

\noi
Note that 
we have
\begin{align}
\big(\phi_{j'} \Q_k (\phi_j f)\big)(x) 
= \phi_{j'}(x) \int_{\R^d}  \eta_k (x-y) (\phi_j f)(y) dy, 
\label{eg2}
\end{align}

\noi
where $\eta_k$ is as in \eqref{eta1}.
In view of \eqref{phi}, 
we have $|x| \sim 2^{j'}$ and $|y|\sim 2^j$
in the integration above.

We first estimate the contribution to \eqref{eg1} from the case $j' \ge j - 2$.
By first summing over $j'$
and applying Bernstein's inequality and \eqref{Lpm1}, we have 
\begin{align}
\begin{split}
&  \sum_{j = 0}^\infty \sum_{j' = j - 2}^\infty
e^{-c \frac \mu p 2^{j' \dl}} 
2^{sk}\| \phi_{j'} \Q_k (\phi_j f)\|_{L^p}\\
&\quad
 \les
 \sum_{j = 0}^\infty 
e^{-c' \frac \mu p 2^{j \dl}} 
2^{sk}\|  \Q_k (\phi_j f)\|_{L^p}\\
&\quad\les
 \sum_{j = 0}^\infty 
e^{-c' \frac \mu p 2^{j \dl}} 
e^{\frac{\mu'}{p} 2^{j\dl}}
\|   f\|_{W^{s, p}_{\mu'}}\\
&\quad \les \|   f\|_{W^{s, p}_{\mu'}}, 
\end{split}
\label{eg3}
\end{align}

\noi
uniformly in $k \in \Z_{\ge 0}$, 
provided that $0 < \mu' <  c'  \mu$.

Next, we consider the case $j' \le j - 3$.
In this case, we have $|x - y| \sim |y|\sim 2^j$ in \eqref{eg2}.
Then, it follows from  \eqref{eta2} that 
\begin{align}
 |\eta_k(x-y)|
\les  2^{d k} e^{-c |2^k (x-y)|^\frac{1}{\ta_0}} 
\les  2^{dk}  
e^{- c' |2^k \cdot 2^j |^\frac{1}{\ta_0}}
e^{-c' |2^k (x-y)|^\frac{1}{\ta_0}}.
\label{eg4}
\end{align}

\noi
In the following, 
it is understood that,  given $j \in \Z_{\ge 0}$, 
we have $|x- y|\sim |y|\sim 2^j$ in \eqref{eg2}.
Hence, 
by first summing over $j'$
and applying Bernstein's inequality followed by Young's inequality with~\eqref{eg4}
and \eqref{Lpm1}, we obtain
\begin{align}
\begin{split}
&  \sum_{j = 0}^\infty \sum_{j' = 0}^{j-3}
e^{-c \frac \mu p 2^{j' \dl}} 
2^{sk}\|\phi_{j'}(  \eta_k* (\phi_j f))\|_{L^p}
\les   \sum_{j = 0}^\infty 
 \|  \eta_k* \jb{\nb}^s (\phi_j f)\|_{L^p}\\
&\quad
 \les   \sum_{j = 0}^\infty 
e^{-c'' |2^k \cdot 2^j |^\frac{1}{\ta_0}} 
\|   \phi_j f\|_{W^{s, p}}
 \les   \sum_{j = 0}^\infty 
e^{- c'' |2^k \cdot 2^j |^\frac{1}{\ta_0}} 
e^{\frac{\mu'}{p}2^{j\dl}}
\|  f\|_{W^{s, p}_{\mu'}}\\
& \quad 
\les \|  f\|_{W^{s, p}_{\mu'}}
\end{split}
\label{eg5}
\end{align}

\noi
for any $\mu' \in \R$ since we have $0< \dl < \frac{1}{\ta_0}$;
see
 the definition \eqref{weight1}
 of the weight $w_\mu$.

Putting \eqref{eg1}, \eqref{eg3}, and \eqref{eg5}
together, we obtain \eqref{equiv2}.

\medskip

We conclude this paper by presenting the proof of Lemma \ref{LEM:equiv2}.

\begin{proof}[Proof of Lemma \ref{LEM:equiv2}]
We first consider the case $s \le 0$.
From \eqref{Lpm1}, we have
\begin{align}
\| \phi_j \Q_k f\|_{W^{s, p}}
\le \| \phi_j \Q_k f\|_{L^p}
\le e^{\frac{ \mu}p  a_1 2^{j\dl}} \|\Q_k f \|_{L^p_\mu}
\label{ee2}
\end{align}

\noi
for any $\mu > 0$, 
uniformly in $k \in \Z_{\ge 0}$.
By \eqref{cpt1} and Bernstein's inequality, we have
\begin{align}
\| \phi_j \Q_k f\|_{W^{s, p}}
\les  \| \Q_k f\|_{W^{s, p}}
\les  2^{sk}\| \Q_k f\|_{L^p}, 
\label{ee3}
\end{align}

\noi
uniformly in $k \in \Z_{\ge 0}$.
Then, \eqref{ee1} follows from interpolating \eqref{ee2} and \eqref{ee3};
see \cite[Theorem~5.4.1]{BL}
for an interpolation of weighted Lebesgue spaces.

Next, we consider the case $s > 0$.
In this case, \eqref{ee1} follows from interpolating
\begin{align*}
 \| \phi_j \Q_k f\|_{L^p}
\le e^{ \frac \mu p a_1 2^{j\dl}} \|\Q_k f \|_{L^p_\mu}
\end{align*}

\noi
and \eqref{ee3} (with $s$ replaced by $s + \eps$).
\end{proof}


\section{Declarations}

\noi
{\bf Funding.}
T.O.~and~G.Z.~were supported by the European Research Council
(grant no. 864138 ``SingStochDispDyn'').
T.O.~was also supported 
 by the EPSRC 
Mathematical Sciences
Small Grant  (grant no.~EP/Y033507/1).
Y.W. was supported by  the EPSRC New Investigator
Award (grant no. EP/V003178/1).

\medskip

\noi
{\bf Competing interests.}
The authors have no competing interests to declare that are relevant to the content of this article.

\medskip

\noi
{\bf Data availability statement.}
This manuscript has no associated data.


\begin{ackno}\rm
 The authors would like to express their gratitude to the anonymous referees for the helpful comments which improved the quality of the paper

%
\end{ackno}

\end{document}